\documentclass[12pt,a4paper]{amsart}

\usepackage{amssymb,enumerate}

\usepackage{graphicx,amssymb,amsfonts,epsfig,amsthm,a4,amsmath,url,verbatim}
\usepackage[latin1]{inputenc}

\newtheorem{thm}{Theorem}[section]
\newtheorem{cor}[thm]{Corollary}
\newtheorem{lem}[thm]{Lemma}

\newtheorem{prop}[thm]{Proposition}

\theoremstyle{definition}
\newtheorem{defn}[thm]{Definition}

\newtheorem{que}[thm]{Question}

\newtheorem{exe}[thm]{Example}

\newtheorem{rem}[thm]{Remark}

\numberwithin{equation}{section}

\newcommand{\N}{\mathbf{N}}
\newcommand{\Z}{\mathbf{Z}}
\newcommand{\R}{\mathbf{R}}
\newcommand{\C}{\mathbf{C}}
\newcommand{\Q}{\mathbf{Q}}
\newcommand{\SL}{\textnormal{SL}}

\newcommand{\PGL}{\textnormal{PGL}}
\newcommand{\PSL}{\textnormal{PSL}}
\newcommand{\Aut}{\mathrm{Aut}}
\newcommand{\Isom}{\text{Isom}}
\newcommand{\HHH}{\mathcal{H}}
\newcommand{\eps}{\varepsilon}
\newcommand{\tu}{\bigtriangleup}
\newcommand{\CPD}{\mathrm{CPD}}
\newcommand{\PC}{\mathrm{PC}}
\newcommand{\IET}{\mathrm{IET}}
\newcommand{\IETbw}{\mathrm{IET}^{\bowtie}}
\newcommand{\mks}{\mathfrak{S}}
\newcommand{\PB}{\mathcal{I}}
\newcommand{\ICO}{\mathcal{I}^{\mathrm{cof}}}
\newcommand{\ICOtop}{\mathcal{I}^{\mathrm{cof}}_{\mathrm{top}}}
\renewcommand{\thesubsection}{{\thesection.\Alph{subsection}}} 

\begin{document}
\title[Groups of piecewise continuous transformations]{Commensurating actions for groups of piecewise continuous transformations}
\author{Yves Cornulier}%
\address{CNRS and Univ Lyon, Univ Claude Bernard Lyon 1, Institut Camille Jordan, 43 blvd. du 11 novembre 1918, F-69622 Villeurbanne}
\email{cornulier@math.univ-lyon1.fr}
\thanks{Supported by project ANR-14-CE25-0004 GAMME}





\date{October 15, 2021}

\begin{abstract}
We use partial actions, as formalized by Exel, to construct various commensurating actions. We use this in the context of groups piecewise preserving a geometric structure, and we interpret the transfixing property of these commensurating actions as the existence of a model for which the group acts preserving the geometric structure. We apply this to many groups with piecewise properties in dimension 1, notably piecewise of class $\mathcal{C}^k$, piecewise affine, piecewise projective (possibly discontinuous).

We derive various conjugacy results for subgroups with Property FW, or distorted cyclic subgroups, or more generally in the presence of rigidity properties for commensurating actions. For instance we obtain, under suitable assumptions, the conjugacy of a given piecewise affine action to an affine action on possibly another model. By the same method, we obtain a similar result in the projective case. An illustrating corollary is the fact that the group of piecewise projective self-transformations of the circle has no infinite subgroup with Kazhdan's Property~T; this corollary is new even in the piecewise affine case.

In addition, we use this to provide of the classification of circle subgroups of piecewise projective homeomorphisms of the projective line. The piecewise affine case is a classical result of Minakawa.
\end{abstract}

\maketitle

\section{Introduction}

The study of commensurating actions, as surveyed and systematized in \cite{CorFW}, is closely related to group actions on CAT(0) cube complexes, which have been developed in the last two decades and now plays a prominent role in geometric group theory.

In the present work, we use them to obtain rigidity results in certain groups such as the group $\PC_{\mathbf{Aff}}(\R/\Z)$ of piecewise affine self-transformations of the circle (allowing discontinuities). This group and various of its subgroups appear in many places, notably in connection to Thompson's groups. The piecewise orientation-preserving subgroup $\PC_{\mathbf{Aff}+}^0(\R/\Z)$ was explicitly defined by M.~Stein \cite{SteM}, and is sometimes referred to as group of affine interval exchanges. Its subgroup $\PC_{\mathbf{Aff}+}^0(\R/\Z)$ of self-homeomorphisms was earlier studied by Brin and Squier in \cite{BrS}.

We also consider the larger group $\PC_{\mathbf{Proj}}(\mathbf{P}^1_\R)$ of piecewise projective self-transformations of $\mathbf{P}^1_\R$. 
Its subgroup $\PC^0_{\mathbf{Proj}}(\mathbf{P}^1_\R)$ of continuous elements appeared in work of Strambach and Betten \cite{Str,Bet}, identifying it as automorphism group of a Moulton plane (an affine plane in the sense of incidence geometry, for which Desargues's theorem does not hold). A survey of these results can be found in \cite{Low}. Its derived series is computed in \cite{BW}.

Another classical occurrence of the group $\PC^0_{\mathrm{Proj}}(\mathbf{P}^1_\R)$ come from the fact that it includes a subgroup isomorphic to Thompson's group $T$, namely its subgroup for which breakpoints are in $\mathbf{P}^1_\Q$ and modeled over the pseudogroup of orientation-preserving integral projective transformations (i.e., piecewise $\PSL_2(\Z)$) \cite{Im}. The group $\PC^0_{\mathbf{Proj}+}(\mathbf{P}^1_\R)$ was further investigated by Greenberg, in relation to foliations.
He also introduced its class-$\mathcal{C}^1$ subgroup $\PC^1_{\mathbf{Proj}+}(\mathbf{P}^1_\R)$.

 A recent renewal of the interest on the piecewise projective self-homeo\-morphism groups comes from Monod's remarkable observation \cite{Mon} that the stabilizer of $\infty$ in the group $\PC^0_{\mathbf{Proj}+}(\mathbf{P}^1_\R)$ is non-amenable, yet has no nonabelian free subgroup.
Lodha and Moore \cite{LM} used this to produce explicit finitely presented subgroups with the same property.

In this paper, we obtain restrictions on subgroups of such groups. Results of this flavor appear for instance in work of Novak \cite{Nov}, showing that the group $\IET^+$ of interval exchanges (the subgroup of piecewise translations in $\PC_{\mathbf{Aff}}(\R/\Z)$) has no distorted cyclic subgroup. Also Dahmani, Fujiwara and Guirardel \cite[Theorem 6.1]{DFG} proved that $\IET^+$ has no infinite subgroup with Kazhdan's Property~T, a result improved in \cite[Theorem 4.1]{JS} (see also \cite[Theorem 7.1(2)]{Cor2}).

We now remind what Property FW is.

\begin{defn}
Given a group $G$ acting on a set $Y$, a subset $X$ is commensurated if, $\tu$ denoting the symmetric difference, $X\tu gX$ is finite for every $g\in G$. We call $(Y,X)$ a commensurating action. A stronger condition is being transfixed: this means that there exists a $G$-invariant subset $X_0$ such that $X\tu X_0$ is finite.
\end{defn}

Being transfixed is the ``obvious" reason for being commensurated and the richness of the theory comes from the failure of the converse. The simplest example of a non-transfixing commensurating action is the action of $\Z$ on itself by translation, commensurating $\N$.

\begin{defn}\label{def_fw}
A group $G$ has Property FW if every commensurating action of $G$ is transfixing. More generally, given a subgroup $H$ of $G$, the pair $(G,H)$ has relative Property FW if every commensurating action of $G$ is transfixing in restriction to $H$. 
\end{defn}

This is notably the case (see \cite{CorFW} for details and further examples)
\begin{itemize}
\item when $H$ has Property FW;
\item when $(G,H)$ has relative Property FH (in the sense that any isometric action of $G$ on a Hilbert space has an $H$-fixed point) --- for $G$ countable this is also known as relative Property~T;
\item when $H$ is cyclic and distorted in $G$;
\item when $H=\langle c\rangle$ is cyclic and unboundedly divisible in $G$ (in the sense that for every $m$ there exists $n\ge m$ and $\gamma\in G$ such that $\gamma^n=c$);
\end{itemize}

See Corollary \ref{exfw} for the last two items.
Property FW is established for various lattices in semisimple groups in \cite{Cor2}, including cases without Property T.

The following is an application of Theorems \ref{rel_fw_conj_aff}
and \ref{rel_fw_conj_pro} below.

\begin{cor}\label{c_nofw}
(a) (See Corollary \ref{affinekaz}) The group $\PC_{\mathbf{Aff}}(X)$ of piecewise affine self-transformations of $X=\R/\Z$ has no infinite subgroup with Property FW (and hence none with Kazhdan's Property T).

(b) (See Corollary \ref{notpro}) The group $\PC_{\mathbf{Proj}}(X)$ of piecewise projective self-transformations of $X=\R/\Z$ has no infinite subgroup with Property T.
\end{cor}

A short companion note \cite{CFW_short} has been written, isolating the proof of the corollary. As an addendum to (b), see Corollary \ref{psl2pro} for strong restrictions on possible subgroups with Property FW in the group of piecewise projective self-transformations of $\R/\Z$.

In \cite[Corollary 1.3]{LMT}, Lodha, Matte Bon and Triestino independently establish a particular case of Corollary \ref{c_nofw}(b): the subgroup $\PC_{\mathbf{Proj}}^0(X)=\PC_{\mathbf{Proj}}^0(X)\cap\mathrm{Homeo}(X)$ of continuous piecewise piecewise projective self-homeomorphisms of $X=\R/\Z$ has no infinite subgroup with Property~T. They make use of a commensurating action allowing to reduce to a theorem of Navas saying that the group of class-$\mathcal{C}^2$ diffeomorphisms of the circle has no infinite subgroup with Kazhdan's Property~T. They also provide an alternative more direct argument applying to the continuous piecewise affine case.

Our approach is based on using geometric structures, namely, in Corollary \ref{c_nofw}, affine structures and projective structures.

Let us start with the affine case. We define an affine curve as a Hausdorff topological space, locally modeled on open subsets of the affine line (with affine change of charts). For instance, $\R/\Z$ is naturally an affinely modeled curve, using (inverses of) homeomorphic restrictions of the projection $\R\to\R/\Z$ as charts. 

We say that an affinely modeled curve is finitely-charted if it has a finite covering by open intervals affinely isomorphic to intervals of $\R/\Z$. Beware that $\R$ itself is not finitely-charted, while every compact affinely modeled curve is finitely-charted, and every finitely-charted affinely modeled curve has finitely many components.

As a classical result of Kuiper (see Appendix \ref{s_projcur}), every connected, finitely-charted affinely modeled curve is isomorphic to:
\begin{itemize}
\item the interval $\mathopen] 0,1\mathclose[$,
\item the standard circle $\R/\Z$, or
\item a non-standard circle: $\R_{>0}/\langle h_\alpha\rangle$ for some unique $\alpha>1$, where $h_\alpha$ denotes the homothety $x\mapsto\alpha x$.
\end{itemize}

Given two affinely modeled curves $X$, $X'$, a piecewise affine isomorphism is by definition an affine isomorphism $X\smallsetminus F\to X'\smallsetminus F'$ between cofinite subsets, modulo identification of two such partial isomorphisms if they are equal on some cofinite subset. It is denoted $X\dashrightarrow X'$ (since it is not really a map). In particular, piecewise affine isomorphisms from $X$ to itself form the group $\PC_{\mathbf{Aff}}(X)$. A piecewise affine isomorphism as above induces by composition a group isomorphism $\PC_{\mathbf{Aff}}(X)\to \PC_{\mathbf{Aff}}(X')$.

Our approach roughly consists in the following. Let $G$ be a group and $H$ a subgroup such that $(G,H)$ has relative Property FW. Let $G\to\PC_{\mathbf{Aff}}(X)$ be a piecewise affine action, with $X=\R/\Z$. Then in restriction to $H$, after suitably ``modifying" $X$, the subgroup $H$ acts by affine automorphisms. ``Modifying" $X$ means removing adding finitely many points (and defining an affine structure at the added points).

\begin{thm}\label{rel_fw_conj_aff}
Let $H$ be a subgroup of $G=\PC_{\mathbf{Aff}}(X)$, with $X=\R/\Z$, such that $(G,H)$ has relative Property FW. Then there exists an affinely modeled curve $X'$ and a piecewise affine isomorphism $\varphi:X\dashrightarrow X'$ such that the induced embedding of $H$ into $\PC_{\mathbf{Aff}}(X')$ maps into $\Aut_{\mathbf{Aff}}(X')$.

If in addition $H$ is infinite and contained in $\PC_{\mathbf{Aff}}^0(X)$, then we can assume in addition that $\varphi$ is given by a piecewise affine homeomorphism $X\to X'$ (in particular, $X'$ is an affine circle, possibly non-standard). 
\end{thm}

For instance, if $H$ has Property FW, then the theorem applies, so that $H$ appears as a subgroup of $\Aut_{\mathbf{Aff}}(X')$, and the latter is a virtually abelian group, which forces $H$ to be finite, yielding Corollary \ref{c_nofw}(a).

Another important case is when $H$ is a distorted cyclic subgroup, in which case the conclusion can also be refined. We call ``standard curve" a finite union of copies of $\R/\Z$ and bounded intervals (this means a finitely-charted curve with no non-standard circle).

Recall that an element $g$ of a group $G$ is distorted if there exists a finite subset $S$ of $G$ such that $g\in\langle S\rangle$ and $\lim_n|g^n|_S/n=0$, where $|\cdot|_S$ denotes word length with respect to $S$. For instance elements of finite order are distorted. A cyclic subgroup $\langle g\rangle$ is by definition distorted in $G$ if $g$ has infinite order and is distorted (so finite order elements are distorted but finite cyclic subgroups are undistorted!).

\begin{cor}\label{distortedaff}
(a) Let $\sigma$ be a distorted element in $\PC_{\mathbf{Aff}}(X)$, with $X=\R/\Z$. Then there exists a standard curve $X'$ and a piecewise affine isomorphism $\varphi:X\dashrightarrow X'$ such that the image of $\sigma$ in $\PC_{\mathbf{Aff}}(X')$ belongs to $\Aut_{\mathbf{Aff}}(X')$. 

(b) If moreover $\sigma\in\PC^0_{\mathbf{Aff}}(\R/\Z)$ and has infinite order, then there exists $\varphi\in\PC^{0,+}_{\mathbf{Aff}}(\R/\Z)$ such that $\varphi\sigma\varphi^{-1}$ is an irrational rotation.
\end{cor}

Here, what is not covered by Theorem \ref{rel_fw_conj_aff} is that the regularization holds in a standard circle.
In (b), $+$ means that $\varphi$ is orientation-preserving.

Corollary \ref{distortedaff}(b) is close to a result recently obtained by Guelman and Liousse \cite{GuL} while this paper was in preparation: they obtain, by a different method, and assuming that $\sigma$ is distorted in $\PC^0_{\mathbf{Aff}}(\R/\Z)$, the conjugation for some power of $\sigma$.

Actually, it is unknown whether the group $\PC_{\mathbf{Aff}}(\R/\Z)$ admits a distorted cyclic subgroup; this question is originally due to Navas \cite[Cor.\ 1.10]{HL}. The question is also open for $\PC^0_{\mathbf{Aff}}(\R/\Z)$. See Corollary \ref{eqque} for some equivalent restatements of these questions.

Let us pass to the projective case. Say that a homeomorphism between open intervals of the projective line $\mathbf{P}^1_\R$ is projective (or homographic) if is the restriction of some homography (i.e., of some element of $\PGL_2(\R)$). A projectively modeled curve here means a Hausdorff topological space endowed with a system of charts in intervals of $\mathbf{P}^1_\R$, with projective changes of charts. It is finitely-charted if it has a finite cover by open intervals projectively isomorphic to intervals in $\mathbf{P}^1_\R$. As in the affine case, compact implies finitely-charted, which implies having finitely many connected components. The universal covering of $\mathbf{P}^1_\R$ is an example of a connected projectively modeled curve that is not finitely-charted.

The classification of connected projectively modeled curves is closely related to the classification of conjugacy classes in the universal covering of $\SL_2(\R)$. Actually, while appearing at many places, it is often written in a vague way, or even with an old chestnut mistake; see Appendix \ref{s_projcur}.

Here is the projective analogue to Theorem \ref{rel_fw_conj_aff}.

\begin{thm}\label{rel_fw_conj_pro}
Let $H$ be a subgroup of $G=\PC_{\mathbf{Proj}}(X)$, with $X=\R/\Z$, such that $(G,H)$ has relative Property FW. Then there exists a projectively modeled curve $X'$ and a piecewise projective isomorphism $\varphi:X\dashrightarrow X'$ such that the induced embedding of $H$ into $\PC_{\mathbf{Proj}}(X')$ maps into $\Aut_{\mathbf{Proj}}(X')$.

If in addition $H\subset\PC_{\mathbf{Proj}}^0(X)$, then we can assume in addition that there is an $H$-invariant finite subset $F$ such that $\varphi$ is given by a piecewise projective homeomorphism $X\smallsetminus F\to X'$. 
\end{thm}

Thus $H$ appears as a subgroup of $\Aut_{\mathbf{Proj}}(X')$, which virtually is the direct product of identity components automorphism groups $L_i$ of its connected components $X'_i$. Then for each $i$, the group $L_i$ is metabelian unless $X'_i$ is projectively isomorphic to a finite cover of $\mathbf{P}^1_\R$ (see the tables in \S\ref{ap_proj}, or \cite[Lemma 3.2]{CFW_short} for a direct argument), in which case $L_i$ is isomorphic to a finite cover of $\PSL_2(\R)$. This discards infinite subgroups with Property~T (Corollary \ref{c_nofw}(a)) and drastically restricts the possible structures for its subgroups with Property FW.

\begin{rem}
Although in a different context, let us mention another result concluding the existence of an invariant projective structure. Navas \cite[Proposition 2.1]{Nav} proves, for a group of $\mathcal{C}^3$-diffeomorphisms of the circle, assuming that it has no invariant probability measure and a certain technical condition, concludes that there is an invariant projective structure.
The technical condition roughly says that some cocycle, called Liouville cocycle, related to the distortion of cross-ratios, is a coboundary,  As in the current paper, he uses this proposition in combination with the knowledge of the automorphism group of projectively modeled curves, namely the fact that the orientation-preserving automorphism group is metabelian unless the projectively modeled curve is a finite covering of the projective line.
\end{rem}

\begin{rem}\label{aff_proj_stru}
Theorem \ref{rel_fw_conj_pro} notably applies to distorted elements (in this case we do not separately state a corollary). This result would be particularly complicated to state and prove without reference to projective structure, since the classification of those is significantly more elaborate than that of the affine case. Indeed, while classifying subgroups of $\PC^0_{\mathbf{Aff}}(\R/\Z)$ that are conjugate to $\mathrm{SO}(2)$ in $\mathrm{Homeo}(\R/\Z)$, Minakawa exhibits elements, which turn out to be those automorphisms of non-standard affinely modeled circles, conjugated into the standard circle by a suitable piecewise affine homeomorphism. A projective analogue would therefore require to exhibit, for every connected projectively modeled circle $X$, a piecewise projective homeomorphism $f$ onto $\R/\Z$ and in each case write down explicitly elements of $f\Aut_{\mathbf{Proj}}(X)f^{-1}$ as elements of $\PC^0_{\mathbf{Proj}}(\R/\Z)$.
\end{rem}

As a complement to Corollary \ref{c_nofw}(b), let us state:

\begin{cor}\label{cor_fwpcpj}
For every subgroup $\Gamma$ of $\PC_{\mathbf{Proj}}(\R/\Z)$ with Property FW, there exists $n\ge 0$, a finite index subgroup $\Gamma'$ of $\Gamma$ and a homomorphism $\Gamma'\to\mathbf{PSL}_2(\R)^n$ with finite kernel, such that each projection $\Gamma'\to\mathbf{PSL}_2(\R)$ has Zariski-dense image. In particular, $\Gamma$ is either finite or has a non-abelian free subgroup.
\end{cor}

\begin{exe}\label{exfwpp}
Using Theorem \ref{rel_fw_conj_pro}, one obtains that the group $\Gamma=\PSL_2(\Z[i,\sqrt{2}])$ (which does not have any infinite subgroup with Property~T) does not embed into $\PC_{\mathbf{Proj}}(\R/\Z)$. 
Indeed, to start with, it has Property FW: this uses bounded generation by elementary unipotent elements due to Carter-Keller \cite[Theorem 25.11]{Wi}. Given this, the easy argument to deduce Property FW is the same as the one \cite{CorFW} for $\PSL_2(\Z[\sqrt{2}])$. Then using Corollary \ref{rel_fw_conj_pro} and the classification of connected finitely-charted projectively modeled curves, one would deduce that some finite index subgroup $\Gamma'$ of $\Gamma$ embeds into $\PSL_2(\R)$.

The group $\Gamma$ (and hence $\Gamma'$) lies as an irreducible arithmetic lattice in $\PSL_2(\C)^2$. Hence, $\Gamma'$ has no homomorphism with infinite image into $\PSL_2(\R)$, by Margulis' superrigidity \cite{Mar}.
\end{exe}

We finish this introduction by repeating \cite[Question 1.19(2)]{Cor2}:

\begin{que}
Does there exist any infinite, finitely generated amenable group with Property FW?
\end{que}

I conjecture a positive answer.

\begin{rem}I owe to Nicol\'as Matte Bon the remark that the absence of infinite Property FW groups in $\IET$ (and in the piecewise affine group), established here, discards some tempting candidates for infinite finitely generated amenable groups with Property~FW. Indeed, every amenable group with Property FW has no homomorphism with infinite image into the piecewise projective group, by Corollary~\ref{cor_fwpcpj}.
\end{rem}

Let us now mention results about ``exotic circles". Start from a simple observation: every continuous faithful action of the topological group $\R/\Z$ on the topological space $\R/\Z$ is conjugate to the standard action by translation. Equivalently, all subgroups of $\mathrm{Homeo}(\R/\Z)$ that are topologically isomorphic to $\R/\Z$ are conjugate to the group of translations. Assuming now that the action preserves some structure, one can wonder about the classification by homeomorphism preserving this structure. This question was raised by Minakawa in the affine context, as already mentioned in Remark \ref{aff_proj_stru}. Reinterpreted using affinely modeled structures, Minakawa's result is that there is a natural bijective correspondence between the set of isomorphism class of affinely modeled circles and the set of subgroups of $\PC^0_{\mathbf{Aff}}(\R/\Z)$ that are topologically isomorphic to $\R/\Z$, modulo conjugacy in $\PC^0_{\mathbf{Aff}}(\R/\Z)$. This correspondence goes as follows: start from an affinely modeled circle $X$, choose a piecewise affine homeomorphism $\varphi:X\to\R/\Z$: then the corresponding subgroup is $f\Aut_{\mathbf{Aff}}(X)^\circ f^{-1}$, which up to conjugacy in $\PC^0_{\mathbf{Aff}}(\R/\Z)$ only depends on the isomorphism type of $X$ (which depends on a parameter in $\R_{\ge 1}$).

In the projective case, we establish a similar result. For it we need not only the classification of projectively modeled circles, but also of their 1-dimensional compact connected subgroups of automorphisms, which is fully described in Appendix \ref{s_projcur}. Only those homogeneous ones (i.e., with transitive automorphism group) are relevant here. For a homogeneous projectively modeled circle, the maximal connected compact subgroups of automorphisms are all conjugate (and indeed equal to the identity component of the automorphism group, with the exception of finite covers of $\mathbf{P}^1_\R$, see Proposition \ref{conjuo2}).

\begin{thm}\label{proj_minakawa}
There is a natural bijective correspondence between the set of homogeneous projectively modeled circles and the set of subgroups of $\PC^0_{\mathbf{Proj}}(\R/\Z)$ that are topologically isomorphic to $\R/\Z$ (for the topology induced by inclusion in $\mathrm{Homeo}(\R/\Z)$), modulo conjugacy in $\PC^0_{\mathbf{Proj}}(\R/\Z)$. This correspondence maps $X$ to the conjugacy class $fK_Xf^{-1}$, where $f$ is a piecewise projective homeomorphism $\varphi:X\to\R/\Z$, and $K_X$ is a maximal compact connected subgroup in $\Aut_{\mathbf{Proj}}(X)$.

The same result holds if $\PC^0_{\mathbf{Proj}}(\R/\Z)$ is replaced with its class $\mathcal{C}^1$ subgroup $\PC^1_{\mathbf{Proj}}(\R/\Z)$.
\end{thm}

The homogeneous projectively modeled circles form a proper subclass of the class of projectively modeled circles. Namely, this subclass consists of 
\begin{itemize}\item the affinely modeled circles (the standard circle $\Theta_1$ and the non-standard circle $\Theta_t=\R_{>0}/\langle t\rangle$ for $t>1$), and \item the metaelliptic curves $\Omega_t=\Sigma_\infty/\langle\xi_t\rangle$ for $t>0$. \end{itemize}
Here $\Sigma_\infty$ is the universal covering of $\mathbf{P}^1_\R$, and $(\xi_t)_{t\in\R}$ is the lift of the group of rotations $\mathrm{PSO}(2)=(\bar{\xi}_t)_{t\in\R/\Z}$; in particular $\Omega_n$ for $n\in\N_{>0}$ is the connected $n$-fold cover of $\mathbf{P}^1_\R$. See Theorem \ref{tsg} for a more detailed account.

The group $\PC^1_{\mathbf{Proj}}(X)$ of piecewise projective $\mathcal{C}^1$-diffeomorphisms was introduced by Greenberg \cite{Gre}. The question of classifying circle groups within this group is explicit in Sergiescu's notes \cite{Ser}; it is solved by the second statement of Theorem \ref{proj_minakawa}.

At a topological level, the only other connected Lie group transitive actions on the circle are given by the action of $\PSL_2(\R)$ on the projective line, and its finite coverings. These actions preserve a projective structure, by definition. Answering a question in \cite{LMT}, we show here there are no ``exotic" versions of such actions:

\begin{thm}[see Corollary \ref{lmtminakawa}]
Let $G$ be the image of a continuous and injective homomorphism from the $n$-fold covering $\PSL_2^{(n)}(\R)$ to $\mathrm{Homeo}(\R/\Z)$. Suppose that $G\subset\PC^0_{\mathbf{Proj}}(\R/\Z)$. Then $G$ is uniquely determined up to conjugation in $\PC^0_{\mathbf{Proj}}(\R/\Z)$. If moreover $G\subset\PC^1_{\mathbf{Proj}}(\R/\Z)$ then $G$ is uniquely determined up to conjugation in $\PC^1_{\mathbf{Proj}}(\R/\Z)$.
\end{thm}
 
\bigskip

\noindent{\bf Acknowledgements.} I am grateful to Misha Kapovich for suggesting me to study interval exchanges under the angle of commensurating actions. I thank \'Etienne Ghys for useful references about affine and projective 1-dimensional structures. I thank Thierry Bousch for a very useful remark (see \ref{doubpoints}). I am grateful to Michele Triestino for various corrections, and for bringing my attention to the question of exotic piecewise projective circles in \cite{Ser}, and to the authors of \cite{LMT} for sending me a copy of their preprint. I am also grateful to Nicol\'as Matte Bon for corrections, and for mentioning me that some topological restriction was unnecessary in my initial statement of Proposition \ref{semic}.

I thank the UNAM (CNRS Laboratory Solomon Lefschetz) in Mexico City for its hospitality during my stay from January to April of 2017, in which most of this work was undertaken.

\setcounter{tocdepth}{1}
\tableofcontents

\section{Method and results}\label{s_methods}

The purpose of this work is to show how the notion of partial group action, formalized by Exel \cite{E}, naturally fits into this context, and obtain various applications, notably to group actions on the circle ``piecewise" preserving some geometric structure, such as piecewise affine or piecewise projective actions.
The use of the formalism of partial actions punctually appeared in geometric group theory already, namely in B.\ Steinberg's work \cite{Ste}, but not in the context of commensurating actions.

\subsection{Partial actions}

For $X$ a topological space, let $\mathcal{I}_{\mathrm{top}}(X)$ be the set of partial homeomorphisms of $X$ between open subsets. Each element $f$ of $\mathcal{I}(X)$ is a homeomorphism $D_f\to D'_f$ with $D_f,D'_f$ open subsets of $X$. The open subset $D_f$ is called the domain of $f$. Then $\mathcal{I}(X)$ is a monoid, where $D_{g\circ f}$ is by definition $\{x\in D_f:f(x)\in D_g\}$. ,The inverse partial homeomorphism $D'_f\to D_f$ of $f$, denoted $f^{-1}$ is called preinverse of $f$ (rather than inverse, since $f^{-1}\circ f$ is not the identity, but the partial identity of $D_f$). Note that $D'_f=D_{f^{-1}}$. See \S\ref{s_ism} for details. For $f,g$, the notation $f\subset g$ means that the graph of $f$ is contained in the graph of $g$; that is, $g$ extends $f$. 

Following Exel \cite{E}, a continuous partial action of a (discrete) group $G$ on $X$ is a map $\alpha:G\to \mathcal{I}_{\mathrm{top}}(X)$ satisfying the axioms: $\alpha(1)=\mathrm{id}_X$, $\alpha(g^{-1})=\alpha(g)^{-1}$, and $\alpha(g)\alpha(h)\subset\alpha(gh)$ for all $g,h\in G$. We call $X$ a continuous partial $G$-space.

Every partial action has a natural orbit decomposition.

We say that a partial action is cofinite if it has cofinite domains: $X\smallsetminus D_{\alpha(g)}$ is finite for every $g\in G$.

Let $G$ acts continuously on a topological space $E$ and $X$ is an arbitrary open subset of $E$. We say that $X$ is $G$-essential in $E$ if every $G$-orbit meets $X$. In general, $G$ partially acts on $X$ by restriction. Namely, denoting by $\beta$ the action on $E$, and thinking $\beta(g)$ as its graph (a subset of $E^2$), one defines $\alpha(g)=X^2\cap\beta(g)$: this makes $X$ a continuous partial $G$-space.

It turns out that conversely, the $G$-space $E$ can be constructed out of the partial action on $X$, provided $X$ is $G$-essential in $E$. Moreover, every continuous partial action arises this way: every continuous partial $G$-space $X$ has an essentially unique ``universal globalization", namely a $G$-set $\hat{X}=\hat{X}_G$ containing $X$ as a $G$-essential open subset. The point of view of partial actions is nevertheless useful, because it often occurs that $X$ is a ``nice" object while $E$ is not (for instance, $X$ is Hausdorff but $E$ is not).

Let $X$ be a continuous partial $G$-space. Say that $X$ is $G$-transfixed above if $\hat{X}\smallsetminus X$ is finite. In general, $X$ is $G$-commensurated in $\hat{X}$ if and only if the partial $G$-action on $X$ is cofinite, and $G$-transfixed in $\hat{X}$ if and only if there exists a finite subset $F$ of $X$ such that the partial action of $G$ on $X$ is transfixed above.

A rich source of partial actions is given by the following:

\begin{prop}
[Proposition \ref{semic}]
Let $X$ be a Hausdorff topological space with no isolated point. Then every $f\in\PC(X)$ has a unique maximal representative, that is, a homeomorphism $\alpha(f):D_f\to D'_f$ with $D_f,D'_f$ cofinite representing $f$, that contains all other representatives of $f$. The assignment $f\mapsto\alpha(f)$ is a continuous partial action of $\PC(X)$ on $X$.
\end{prop}

On the one hand, for a compact manifold (possibly with boundary) without component of dimension $\le 1$, the canonical inclusion $\mathrm{Homeo}(X)\subset\PC(X)$ is an equality, see Remark \ref{r_homeo_pc}; in particular the universal globalization $\hat{X}$ of $X$ under the partial action of $\PC(X)$ is reduced to $X$. If $X=\R^n$ for $n\ge 2$, it follows that this universal globalization is reduced to the 1-point compactification $\R^n\cup\{\infty\}$. 

 On the other hand, for $X$ a circle, or a Cantor space, the group $\PC(X)$ is much larger than the homeomorphism group. In these cases, $\hat{X}$ is not Hausdorff, as it includes the compact space $X$ as a dense open proper subset. Also, since $\PC(X)$ acts on $\hat{X}$ and every orbit meets the open subset $X$, the space $\hat{X}$ is locally homeomorphic to $X$. In particular, for $X$ a circle, $\hat{X}$ is a connected 1-dimensional non-Hausdorff topological manifold.
  
Using geometric structures, we obtain below further naturally occurring partial actions.

\subsection{Pseudogroups}\label{si_pseudo}

We define a pseudogroup on a topological space as a set of partial homeomorphisms between open subsets with a suitable stability condition (see \S\ref{s_pseudogroup}). A pseudogroup $S$ on a (topological) space $A$ makes meaningful the notion, due to Ehresmann, of space modeled on $(A,S)$ (see \S\ref{s_pseudogroup} for details). For instance, a space modeled on the pseudogroup consisting of $\mathcal{C}^k$-diffeomorphisms between arbitrary subsets of $\R^n$ is the same as a purely $n$-dimensional manifold of class $\mathcal{C}^k$. If $X$ is a space modeled on $S$, the set of $S$-preserving homeomorphisms between open subsets of $X$ is itself a pseudogroup. For instance, from the pseudogroup of partial affine homeomorphisms on $\R$, the circle $\R/\Z$ inherits an affine structure for which the projection $\R/\Z$ is a local affine isomorphism. In turn, through the local homeomorphisms $\R/\Z\leftarrow\R\to\mathbf{P}^1_R$, the circle $\R/\Z$ inherits a projective structure. We refer to \S\ref{s_pseudogroup} for precise definitions.

Let $S$ be a pseudogroup on a space $A$ and $X$ is a Hausdorff $S$-modeled space. The parcelwise-$S$ group $\PC_S(X)$ is the subgroup of $X$ having a representative that is an $S$-preserving homeomorphism between open cofinite subsets. Then, if $X$ has no isolated point, then $\PC_S(X)$ also has a canonical $S$-preserving cofinite partial action on $X$, given by $g\mapsto\alpha_S(g)$, where $\alpha_S(g)$ is the maximal representative of $g$ that is an $S$-preserving homeomorphism between cofinite subsets of $X$ (see \S\ref{s_alphas} for details). Then the universal globalization $\hat{X}_G$ of $X$ with respect to every subgroup $G\subset\PC_S(X)$ is also an $S$-modeled space; still it can fail to be Hausdorff. 

\begin{thm}[Regularization theorem]\label{reg_t}
Let $A$ be a Hausdorff space with no isolated point and $S$ a pseudogroup on $A$. Let $X$ be an $S$-modeled space.

Let $H$ be a subgroup of $\PC_S(X)$ such that $X$ is transfixed under the partial $S$-preserving action of $H$. Then there exists a cofinite subset $Y$ of $X$ such that $\hat{Y}_H$ is Hausdorff. In particular, $\hat{Y}_H$ is an $S$-modeled Hausdorff space, and the identity map of $Y$ is an $S$-preserving homeomorphism between cofinite subsets of $X$ and $\hat{Y}_H$, intertwining the parcelwise-$S$ action on $X$ to an $S$-preserving action on $\hat{Y}_H$.
\end{thm}

Roughly, the conclusion says that after ``cutting-and-pasting" $X$ at finitely many points, the action becomes continuous and $S$-preserving.

The proof of this theorem is easy; the main contribution here was to gather the various features and write it down.

In the continuous case (i.e., $H\subset\PC^0_S$), one obtains a stronger conclusion.

\begin{thm}[Regularization theorem, continuous case]\label{regc_t}
Let $A$ be a Hausdorff space with no isolated point and $S$ a pseudogroup on $A$. Let $X$ be an $S$-modeled space.

Let $H$ be a subgroup of $\PC^0_S(X)$ such that $X$ is transfixed under the partial $S$-preserving action of $H$. Then there exists cofinite subsets $Y\subset Y'$ of $X$ with $Y'$ $H$-invariant, and an $H$-invariant $S$-modeled structure on $Y'$ that coincides with the original one on $Y$. In particular, if $H$ has no finite orbit on $X$, then $Y'=X$.
\end{thm}

Roughly, it means that after removing a finite invariant subset, one can ``repair" the $S$-structure at finitely many points so that the resulting $S$-structure is preserve (unlike in the discontinuous case, there is no need to change the topology). In the language of Theorem \ref{reg_t}, this means that $\hat{Y}_H$ identified by an parcelwise-$S$ $H$-equivariant homeomorphism to an $H$-invariant cofinite subset $Y'$ of $Y$.

Now let us discuss the transfixing assumption. Denoting by $D^S_g\subset X$ the domain of $g\in H$ and $\Delta^S_g$ its complement, $X$ is transfixed by $H$ if and only $\sup_{g\in H}|\Delta^S_g|<\infty$.

\begin{itemize}
\item when $S$ the pseudogroup of all partial homeomorphisms of $X$ (Hausdorff with no isolated point), $\Delta_g$ is the set of discontinuity points;
\item when $X=\R/\Z$ with the pseudogroup of isometries or oriented isometries, again $\Delta^S_g$ is the set of discontinuity points;
\item when $X=\R/\Z$ with the affine pseudogroup or its oriented analogue, $\Delta^S_g$ is the set of ``breakpoints", that is, discontinuity points as well as continuity points at which $g$ is not differentiable (or equivalently at which the left and right derivatives are not equal);
\item when $X=\R/\Z$ with the projective pseudogroup or its oriented analogue, $\Delta^S_g$ is the set of points at which $g$ is not of class $\mathcal{C}^2$.
\end{itemize}

For $S$ a pseudogroup on model space with no isolated point, and for $X$ a $S$-modeled Hausdorff space, the function $\PC_S(X)\to\N$, $g\mapsto |\Delta_g^S|$ has specific properties that are common to all such functions defined from cofinite partial actions. For instance, the behavior of such functions on cyclic subgroups is very restricted for every $g$ there exists $m_g^S\in\N$ such that $|\Delta_{g^n}^S|=m_g^Sn+b_g(n)$ for all $n\in\N$, for some bounded non-negatively valued function $b_g$. In particular, the limit $\lim_{n\to\infty}\frac{|\Delta_{g^n}^S|}{n}$ exists, is a non-negative integer, and is zero if and only if $\sup_k|\Delta_{g^n}^S|<\infty$.

Let us apply this to the $\mathcal{C}^k$-pseudogroup. Let $k\in\N=\{0,1,\dots\}$ be an integer. For $\sigma$ be a parcelwise-$\mathcal{C}^k$ self-transformation of $\R/\Z$ (or of an open interval), define $K_0(\sigma)$ as the subset of outer discontinuity points of $\sigma$ (that is, the set of $x$ such that $\sigma(x^+)\neq\sigma(x^-)$. For $x\notin K_0(\sigma)$, we can choose the value of $\sigma(x)$ as the one making $\sigma$ continuous at $x$. For $i\in\{1,\dots,k\}$ let $K_i(\sigma)$ be the subset of those $x\notin K_{i-1}(\sigma)$ at the neighborhood of which $\sigma$ is not of class $\mathcal{C}^i$. For $i\in\{0,\dots,k\}$, let $k_{i}(\sigma)$ be the cardinal of $K_{i}(\sigma)$ and $k_{\le i}(\sigma)=\sum_{j=0}^ik_j(\sigma)$. Note that $k_i(\sigma)=k_i(\sigma^{-1})$. 

\begin{cor}[See Corollary \ref{singro}]\label{countingsing}
For every $k\in\N$ and parcelwise-$\mathcal{C}^k$ self-transformation $\sigma$ of $\R/\Z$, there exist integers $0\le m_0\le \dots \le m_k$ and bounded non-negative even functions $b_i:\Z\to\N$ such that for all $i\in\{0,\dots,k\}$, we have $k_{\le i}(\sigma^n)=m_i|n|+b_i(n)$.

In particular, $k_{\le i}$ and $k_i$ have the property of growing either linearly or being bounded.
\end{cor}

This lets us retrieve or improve some known results in a unified way. Applied when $k=0$ and in the case of interval exchanges, this concerns the discontinuity growth for interval exchanges and this counting result was obtained in \cite{Nov,DFG}. When $k=1$, in the case of piecewise affine self-transformations, it was obtained by Guelman and Liousse \cite[\S 4]{GuL} that the sequences $(k_1(\sigma^n))_{n\ge 0}$ (and $(k_0(\sigma^n))_{n\ge 0}$) are either bounded of have linear growth. See also Corollary \ref{Ckconju} for a specification of Theorem \ref{reg_t} to this context.

\section{Preliminary definitions}

\subsection{Inverse symmetric monoids}\label{s_ism}

A semigroup is a set endowed with an associative binary law. A monoid is a semigroup endowed with a unit element (which is then unique). In a semigroup, a preinverse of an element $x$ is an element $y$ such that $xyx=x$ and $yxy=y$. A semigroup (resp.\ monoid) is called an inverse semigroup (resp.\ inverse monoid) if every element $x$ has a unique preinverse, then denoted $x^{-1}$. Homomorphisms of monoids (resp.\ inverse semigroups, resp.\ inverse monoids) are required to map unit to unit (resp.\ preinverse to preinverse, resp.\ both). In inverse semigroup theory, preinverses are often called ``inverses" but we rather use the more usual terminology of inverses in monoids (an inverse for $x$ is $y$ such that $yx=xy=1$; such $y$ is unique and is then a preinverse of $x$).

Given sets $X,Y$, the set $\mathcal{P}(X\times Y)$ of subsets of $X\times Y$ (the set of binary relations on $X,Y$) is endowed with its usual composition: given $A\subset X\times Y$ and $B\subset Y\times Z$, 
\[A\circ B=\{(x,z)\in X\times Z:\exists y\in Y:(x,y)\in A,(y,z)\in B\}.\]
In particular, this makes $\mathcal{P}(X^2)$ a monoid (the diagonal $\mathrm{id}_X$ being the unit element). For $u\in\mathcal{P}(X\times Y)$, define $D_u$ and $D'_u$ as its projections on $X$ and $Y$, called its domain and range. Denote by $s:A\mapsto A^{-1}$ the involution $\mathcal{P}(X\times Y)\to\mathcal{P}(Y\times X)$ induced by $(x,y)\mapsto (y,x)$. Beware that it is not an inverse map on $\mathcal{P}(X^2)$ as soon as $X$ is nonempty, and not a preinverse map as soon as $X$ contains two distinct elements. 

Let $\mathcal{I}(X,Y)$ be the set of subsets of $X^2$ both of whose projections are injective. These are called partial bijections of $X$, namely each $u\in\mathcal{I}(X)$ is a bijection between its domain $D_u$ and its range $D'_u$. These are stable under composition. In particular, $\PB(X)$ is a well-known submonoid of $X^2$, called inverse symmetric monoid of $X$. Indeed, this is an inverse monoid: the preinverse of $u$ being $u^{-1}$. 

Let $\ICO(X)$ be the set of partial bijections of $\sigma\in\mathcal{I}(X)$ with cofinite domain and range. This is an inverse submonoid of $\PB(X)$.

Given topological spaces $A$, $B$, let $\mathcal{I}_{\mathrm{top}}(A,B)$ be the subset of $\PB(A,B)$ consisting of those $u$ such that both $D_u$ and $D'_u$ are open subsets of $A$, and such that $u$ induces a homeomorphism $D_u\to D'_u$. In particular, $\PB_{\mathrm{top}}(A)=\PB_{\mathrm{top}}(A,A)$ is an inverse submonoid of $\PB(A)$. We write $\PB_{\mathrm{top}}^{\mathrm{cof}}(A,B)=\PB_{\mathrm{top}}(A,B)\cap\PB^{\mathrm{cof}}(A,B)$ and 
$\PB_{\mathrm{top}}^{\mathrm{cof}}(A)=\PB_{\mathrm{top}}^{\mathrm{cof}}(A,A)$.

\subsection{Partial actions}
\subsubsection{Definition}
\begin{defn}[Exel \cite{E}]
A partial action of a group $G$ on a set $X$ is a map $\alpha:G\to\PB(X)$ satisfying:
\begin{enumerate}
\item $\alpha(1)=\mathrm{id}_X$;
\item $\alpha(g^{-1})=\alpha(g)^{-1}$ for all $g\in G$;
\item for all $g,h\in G$, $\alpha(g)\circ\alpha(h)\subset \alpha(gh)$.
\end{enumerate}
We say that $(X,\alpha)$ is a partial $G$-set (and if $\alpha(G)\subset\mks(X)$, the group of permutations of $X$, we say that it is a global $G$-set, or just $G$-set).
\end{defn}

When $\alpha$ is valued in $\ICO(X)$, it is called a cofinite-partial action.

Given a partial action of $G$ on $X$, a topology $\mathcal{T}$ on $X$ is said to be preserved by the partial action if for every $g\in G$, we have $\alpha(g)\in\mathcal{I}_{\mathrm{top}}(X_\mathcal{T})$. Here $X_\mathcal{T}$ means $X$ endowed with the topology $\mathcal{T}$ (mostly the topology is implicit and is omitted from the notation). When $X$ is endowed with $\mathcal{T}$, we call it a topological partial action of $G$ on $X$; we call $X$ a partial topological $G$-space, and (global) topological $G$-space when the underlying partial action is an action.

A homomorphism between partial $G$-sets $(X,\alpha)$, $(Y,\beta)$ is a map $f:X\to Y$ that is $G$-equivariant, in the sense that $\alpha(g)\subset (f\times f)^{-1}(\beta(g))$ for all $g\in G$. That $f$ is $G$-equivariant just means that for every $g\in G$ every $x\in X$ such that $\alpha(g)x$ is defined, then $\beta(g)f(x)$ is also defined and $f(\alpha(g)x)=\beta(g)f(x)$.

We say that such a homomorphism $f$ is full if the above inclusion is an equality: $\alpha(g)= (f\times f)^{-1}(\beta(g))$ for all $g\in G$. We then say that $f$ is fully $G$-equivariant.

The homomorphism $f$ is essential if for every $y\in Y$ there exists $g\in G$ such that $\beta(g)y$ is defined and belongs to $f(X)$.

A bijection $(X,\alpha)\to (Y,\beta)$ is called $G$-biequiv\-ariant if both $f$ and $f^{-1}$ are $G$-equivariant. For a bijective homomorphism, this means that $f$ is fully $G$-equivariant. 

Beware that a bijective $G$-equivariant map can fail to be $G$-biequiv\-ariant, unlike in the setting of global actions. For instance, $\beta$ could be an action but not $\alpha$ (see Proposition \ref{glo_tran} for a result in this context).

\subsubsection{Universal globalization}\label{uniglo}
Given a partial $G$-set $(X,\alpha)$, partial $G$-sets $(Y,\beta)$ endowed with a homomorphism $(X,\alpha)\to (Y,\beta)$ form a category, whose maps are the $G$-equivariant maps so that the obvious triangle commutes. It has a full subcategory, consisting of those (global) $G$-sets $(Y,\beta)$ endowed with a homomorphism $(X,\alpha)\to (Y,\beta)$. An initial element in this category is called a universal globalization for $(X,\alpha)$.

\begin{thm}[Megrelishvili \cite{Me1,Me2}, Abadie \cite{AT,A}, Kellendonk-Lawson \cite{KL}]\label{tunivglo}
Every partial $G$-set $(X,\alpha)$ admits a universal globalization $\iota:(X,\alpha)\to(\hat{X},\hat{\alpha})$. Moreover
\begin{enumerate}
\item the map $\iota$ is injective;
\item every $\hat{\alpha}(G)$-orbit in $\hat{X}$ meets $\iota(X)$ (that is, $\iota$ is essential).
\end{enumerate}
\end{thm}

In detail, the first assertion means that there is a set $\hat{X}$, a map $\iota:X\to\hat{X}$, a $G$-action $\beta$ on $X$ such that $\iota$ is $G$-equivariant and such that for every other $G$-set $Y$ and $G$-equivariant map $f=X\to Y$, there exists a unique $G$-equivariant map $u_f:\hat{X}\to Y$ such that $f=u_f\circ\iota$.

Let us recall the simple construction. Denote by $D_g\subset X$ the domain of $\alpha(g)$. Start from $\tilde{X}=G\times X$ and the map $X\to\tilde{X}$ given by $x\mapsto (1,x)$. Endow $\tilde{X}$ with the $G$-action given by $g\cdot (h,x)=(gh,x)$. Define $\hat{X}$ by modding out by the equivalence relation $\sim$ given by $(h,x)\sim (k,y)$ if $x\in D_{k^{-1}h}$ and $\alpha(k^{-1}h)(x)=y$. The $G$-action passes to the quotient $\hat{X}$ and the resulting map $X\to\hat{X}$ is a universal globalization. This can be checked as an exercise: indeed, the virtue of the theorem is, first and foremost, to have been formulated.

The following easy lemma is a convenient way to recognize a universal globalization.

\begin{lem}[{\cite[Prop.\ 3.5]{KL}}]\label{reglo}
Any full, essential and injective homomorphism from a partial $G$-action to a global $G$-action is a universal globalization \cite[Prop.\ 3.5]{KL}.\qed
\end{lem}

\begin{rem}\label{tunivr}
Megrilishvili \cite{Me1,Me2} used a less general, related formalism of ``preactions" and then constructed a ``universal action" with the same construction as the one used for the universal globalization in \cite{AT,A,KL} in the framework of partial actions.
Abadie's result was written in 1999 in his PhD \cite{AT} and published only in 2003 \cite{A}. The construction of the universal globalization was independently later obtained, using the same construction, by Kellendonk and Lawson \cite{KL} (published in 2004, but quoted as a preprint in \cite{Me2} published in 2000).
\end{rem}

\begin{thm}[Abadie \cite{AT,A,KL}]\label{tunivglot}
Let $(X,\alpha)$ be a partial $G$-set and $\iota:(X,\alpha)\to(\hat{X},\hat{\alpha})$ a universal globalization. For every topology $\mathcal{T}$ on $X$ preserved by $\alpha(G)$, there is a unique topology $\hat{\mathcal{T}}$ on $\hat{X}$ preserved by $\hat{\alpha}(G)$ such that $\iota$ is an open continuous map. Moreover, $(\hat{X},\hat{\alpha})$ is an initial element in the category of topological $G$-sets (topological spaces endowed with a $G$-action by self-homeomorphisms) endowed with a continuous homomorphism of partial actions from $(X,\alpha)$.
\end{thm}

Actually, Abadie's result is directly stated while defining the globalization in the context of actions on topological spaces. But the underlying action only depends on the underlying partial action, so we have found convenient to restate it as above, which is closer to the subsequent formulation from \cite[\S 3.2]{KL}.

Note that in partial actions, we always consider the acting group as a discrete group, although Abadie \cite{A} more generally addresses partial actions of topological groups.

Let us mention for future reference
\begin{lem}\label{glo_t1}
If $X$ is $T_1$ (that is, singletons are closed), so is $\hat{X}$.
\end{lem}
\begin{proof}
Describe $\hat{X}$ as above as quotient of $G\times X$ by an equivalence relation. Then the equivalence class of $(g,x)$ meets each layer $\{g'\}\times X$ in at most a singleton. Hence if $X$ is $T_1$, equivalence classes are closed, so that the quotient topological space is $T_1$.
\end{proof}

\subsection{Encoding commensurating actions as partial actions and vice versa}\label{encocp}

Given a partial action $\alpha$ of $G$ on a set $X$, let us define $\ell^-_X(g)$ as the cardinal of the complement $X\smallsetminus D_{\alpha(g)}$; define $\ell^+_X(g)=\ell^-_X(g^{-1})$. Call $\ell^+_X$ and $\ell^-_X$ the loanshark and prodigal semi-index functions of the partial $G$-set $X$. This can actually be interpreted in previous setting: indeed, consider a universal globalization $X\to Y$. Then $\ell^-_X(g)$ coincides with its definition as in \S\ref{s_coma}, in the setting of commensurating actions.

We say that the partial action commensurates $X$ if $\ell^+_X$ (or equivalently $\ell^-_X$) takes finite values. For a partial action, the condition of commensurating $X$ is equivalent to be a cofinite-partial action. This shows that all the theory of commensurating actions translates in the setting of cofinite-partial actions. We say that it transfixes $X$ if there exists a finite subset $F$ of $X$ such that the induced partial action of $G$ on $X\smallsetminus F$ has finite complement in its universal globalization. Equivalently, this means that $X$ is transfixed in its universal globalization.  In particular, given a subgroup $H$ of $G$, relative Property FW of $(G,H)$ can be restated as: every cofinite-partial action of $G$ is transfixing in restriction to $H$.
Moreover, by the dictionary between cofinite-partial actions and commensurating actions, Theorem \ref{t_transf} can be restated as:

\begin{thm}
A partial action of a group $G$ on a set $X$ is transfixing if and only if $\ell^+_X$ (or equivalently $\ell^-_X$) has a finite supremum on $G$.\qed
\end{thm}

\subsection{Parcelwise group $\PC(X)$}\label{s_parcelwise}
Let $X$ be a topological space. Recall from \S\ref{s_ism} that $\ICOtop(X)$ is the inverse submonoid of $\ICO(X)$ consisting of partial homeomorphisms between two open cofinite subsets of $X$.

Identifying two elements in $\ICOtop(X)$ whenever they coincide outside a finite subset, we obtain a group, which we denote by $\PC(X)$, and call it the group of parcelwise continuous self-transformations of $X$.

\subsection{Pseudogroups, modeled structures, piecewise and parcelwise monoids}\label{s_pseudogroup}
\subsubsection{Definition}
 We define a pseudogroup on a topological space $A$ as an inverse submonoid $S$ of $\mathcal{I}(A)$ such that $\{U:\mathrm{id}_U\in S\}$ is a basis of the topology. Note that the latter condition implies that the topology is determined by $S$.
  
The completed pseudogroup $\hat{S}$ consists of the set of elements of $u\in\PB(A)$ that can be written, for some index set $I$, as a union $\bigcup_{i\in I}u_i$ with $u_i\in S$ for all $i$. When $S=\hat{S}$, the pseudogroup is called complete.

In the literature, pseudogroups are often defined as what is called here ``complete pseudogroup", and furthermore are also often supposed to contain $\mathrm{id}_U$ for every open subset $U$. We find convenient here to get rid of these restrictions.

\begin{exe}
Let $G$ be subgroup of the self-homeomorphism group of $A$. Then the set of restrictions of elements of $G$ to open subsets is a pseudogroup on $A$. It is usually not complete: for instance, if $G$ consists of translations on $A=\R$, it is not complete, since the completed pseudogroup contains local translations (with non-connected domain) that are not translations.
\end{exe}

For instance, if $G$ is a subgroup of the group of self-homeomorphisms of $X$, we can define the pseudogroup induced by $G$, starting from those restrictions of elements of $G$ to open subsets.

\subsubsection{$S$-structures, $S$-modeled topological spaces}

Let $S$ be a pseudogroup on a topological space $A$. 

Let $X$ be a topological space. An $S$-atlas on $X$ is a subset $\HHH$ of $\PB(X,A)$, whose elements are called charts, such that $\emptyset\in\HHH$ and such that for any $f,g\in\HHH$, we have $g\circ f^{-1}\in \hat{S}$.

An $S$-atlas is called complete if it satisfies the following two additional conditions:
\begin{itemize}
\item for every $f\in\HHH$ and $u\in S$, we have $u\circ f\in\HHH$;
\item for any index set $I$ and family $(g_i)_{i\in I}$ in $\HHH$ such that $g:=\bigcup_{i\in I}g_i$ belongs to $\PB(X,A)$, we have $g\in\HHH$.
\end{itemize}

The data of a complete $S$-atlas $\HHH$ on $X$ is called an $S$-structure on $X$. A topological space endowed with an $S$-structure is called an $S$-modeled topological space. We call it finitely-charted if $X$ has a finite covering by domains of charts.

Actually, every $S$-atlas on $X$ endows $X$ with an $S$-structure: indeed every $S$-atlas generates a complete $S$-atlas in a natural way.

\subsubsection{Parcelwise-$S$ inverse monoid and group}\label{parcelws}

 Now consider an $S$-modeled topological space $X$, with complete atlas $\HHH=\PB_S(X,A)$. 
 
We then denote by $\PB_S(X)$ the set of elements $h\in\PB(X)$ that can be written, for some index set $I$ and families $(f_i)_{i\in I}$ and $(g_i)_{i\in I}$ in $\HHH$, as $h=\bigcup_{i\in I}g_i^{-1}\circ f_i$.

Then $\PB_S(X)$ is an inverse submonoid of $\PB(X)$ (and thus a pseudogroup), and $\PB_S(X,A)$ is stable under precomposition with $\PB_S(X)$ and postcomposition with $\hat{S}=\PB_S(A)$. By composition, any topological space endowed with a $\PB_S(X)$-structure inherits a canonical $S$-structure.

Consider the inverse submonoid $\ICO_S(X)=\PB_S(X)\cap\ICOtop(X_\delta)$, that is, the set of elements in $\PB_S(X)$ with cofinite domain and range. We call it the parcelwise-$S$ inverse monoid of $X$.
Identifying two elements of $\ICO_S(X)$ whenever they coincide on a cofinite subset, we obtain a subgroup $\PC_S(X)$ of $\PC(X)$, called the group of parcelwise-$S$ self-transformations of $X$.

When $X$ has no isolated point, the canonical homomorphism $\mathrm{Homeo}(X)\to\PC(X)$ is injective. The inverse image of $\PC_S(X)$ in $\mathrm{Homeo}(X)$ is denoted by $\PC_S^0(X)$.

In particular, we define $S^\wp$ as the pseudogroup on the model space $A$ of restrictions of elements of $\PC_S^0(A)$. It is easy to check that $\PC_S(X)=\PC_{S^\wp}(X)$.

\subsubsection{Piecewise-$S$ inverse monoid and group}\label{s_inmopa}

Define $\ICO_{S\sharp}$ as the set of $f\in\ICO_S(X)$ such that
there exist a finite partition $(D_i)_{i\in I}$ of the domain of $f$ into open subset, a family $(g_i)_{i\in I}$ in $\PB_S(X)$ such that the domain $E_i$ of $g_i$ contains the closure $\overline{D_i}$, and such that $g_i$ extends $f_i$.

It is easy to check that this is an inverse submonoid of $\ICO_S(X)$. We call it the piecewise-$S$ inverse monoid of $X$. Its image in $\PC(X)$ is a subgroup $\PC_{S\sharp}(X)$ of $\PC(X)$, called group of piecewise-$S$ self-transformations of $X$.

When $S$ consists of all local homeomorphisms, we denote these by $\ICO_\sharp(X)$ and $\PC_\sharp(X)$.

\subsection{Curves and doubling tricks}\label{doubpoints}

By curve we mean a purely 1-dimensional Hausdorff paracompact topological manifold with finitely many connected components. A connected curve is homeomorphic to the circle or an open interval.

Start from a curve $A$ with a pseudogroup $S$. A finitely-charted $S$-modeled curve is a curve endowed with an $S$-structure definable by finitely many charts (the condition is automatic for compact $S$-modeled curves).

By standard curve we mean finite disjoint union of open bounded intervals and circles (where a circle is a copy of $\R/a\Z$ for some $a>0$). A more intrinsic (but somewhat using more formalism than necessary) is to define a standard curve as a purely 1-dimensional oriented Riemannian manifold, with finitely many components.

Given a standard curve $X$, local orientation-preserving measure-preserving homeomorphisms yield a canonical $\mathbf{Iso}^+$-structure on $X$. Therefore, for every pseudogroup $S$ on the circle including the local orientation-preserving isometries, it endows $X$ with a canonical $S$-structure.

The idea of doubling points in one-dimensional dynamics is very classical, and often attributed to Denjoy. 

The idea underlying the following lemma appears in the context of interval exchanges in Danthony-Nogueira's article \cite{DN}. I thank Thierry Bousch for bringing this important observation to my attention.

\begin{lem}\label{bousch}
For any curve $X$ endowed with an orientation, endow $X^\pm=X\times\{1,-1\}$ with the product topology and an orientation:
that of $X$ on $X\times\{1\}$ and the reverse one on $X\times\{-1\}$. 

Then there is a natural injective group homomorphism $\Phi$ of $\PC(X)$ into $\PC^+(X^\pm)$, whose image is the centralizer of the involution $\sigma:(x,1)\leftrightarrow (x,-1)$ in $\PC^+(X^\pm)$, and which makes the projection $X^\pm\to X$ equivariant.
\end{lem}
\begin{proof}
For $f\in\ICOtop(X)$ and $x\in D_f$, define the reduced derivative $f^\imath(x)$ as equal to $1$ or $-1$ according to whether $f$ is locally orientation-preserving or orientation-reversing at $x$ (this is well-defined because $X$ is endowed with an orientation). It can be thought of as the sign of the derivative, but is defined for arbitrary piecewise strictly monotonic functions.

It satisfies the same property as the derivative for composition: when defined, we have $(f_2\circ f_1)^\imath(x)=f_1^\imath(x)f_2^\imath(f_1(x))$. With a suitable derivability assumption we would obtain an action on the tangent bundle, and then, modding out by the action of positive scalars, an action on the orientation bundle. This idea works directly thanks to the above formula.
Namely, for any $f\in\ICOtop(X)$, and $(x,\eps)\in X\times\{\pm 1\}$, define $\tau(f)(x,\eps)=(f(x),f^\imath(x)\eps)$. Note that it obviously commutes with $s$, and preserves the given orientation. That it defines a partial action with domain of definition $D_f\times\{\pm 1\}$ is immediate from the composition formula. It then induces a group homomorphism of $\PC(X)$ into $\PC^+(X^\pm)$, which is clearly injective. That the image consists of the centralizer of $\sigma$ is straightforward, as well as the additional statement.
\end{proof}

See also Lemma \ref{bousch_s} in the additional presence of a geometric structure.

\begin{rem}
In the context of piecewise isometric maps, the elements in the image of $\Phi$ were called ``linear involutions" in \cite{DN}; thus ``linear involutions" denote one way to represent the group of ``interval exchanges with flips", thus encoded as interval exchanges in this model. Since these are neither linear nor involutions, we prefer avoid this terminology.
\end{rem}

\begin{rem}
There is another interesting topology on $X^\pm$, instead of the product topology as above, namely the Denjoy topology $\tau_{\mathrm{Denj}}$. It is defined as the (local) ordering (in the case of the interval, it consists of the lexicographic order consisting in splitting each $x$ into two consecutive elements). It is compact if $X$ is compact. Its main interest, essentially mentioned by Keane in \cite{Kea}, is that the cofinite-partial action of $\PC(X)$ on $X^\pm$ extends to a continuous action on $(X^\pm,\tau_{\mathrm{Denj}})$ is continuous. Unlike the product topology, it is not (locally) metrizable.
\end{rem}


\section{General results}

\subsection{Commensurating actions}\label{s_coma}

Denote by $\tu$ the symmetric difference between subsets of a set, and $|\cdot|$ the cardinal function.

Given a group $G$ acting on a set $Y$, a subset $X\subset Y$ is commensurated if $X\tu gX$ is finite for every $g\in G$. We then call $(Y,X)$ a commensurating action of $G$. We say that the commensurating action $(Y,X)$ is transfixing, or that the subset $X\subset Y$ is transfixed, if there exists a $G$-invariant subset $X_0$ such that $X\tu X_0$ is finite. The function $\ell_X:g\mapsto |X\tu gX|$ is called the cardinal definition function of the commensurating action $(Y,X)$. If $X$ is transfixed then $\ell_X$ is obviously bounded; the converse also holds:

\begin{thm}[\cite{BPP}]\label{t_transf}
The commensurating $G$-action $(Y,X)$ is $G$-transfixing if and only if its cardinal definite function $\ell_X$ is bounded on $G$.
\end{thm}

Continue with an arbitrary subset $X$ of $Y$. Let us denote $\ell^+_X:g\mapsto |X\smallsetminus gX|$ and $\ell^-_X(g)=\ell^+_X(g^{-1})=|X\smallsetminus g^{-1}X|$. Call $\ell^+_X$ and $\ell^-_X$ the loanshark and the prodigal semi-index functions of $(Y,X)$. Note that $Y$ being commensurated means that either of these functions takes finite values, and transfixed means that either of these functions is bounded.
Each of these functions, say $\ell$, satisfies $\ell(1)=0$ and $\ell(gh)\le\ell(g)+\ell(h)$ for all $g,h\in G$. Note that $\ell_X=\ell^+_X+\ell^-_X$.  (When $X$ is $G$-commensurated, the difference $g\mapsto\ell^+(g)-\ell^+(g^{-1})$ is a well-defined map $G\to\Z$, and actually a group homomorphism, called index character of $(Y,X)$, see \cite[\S 4.H]{CorFW}.)

\begin{prop}\label{cardefz}
1) \cite[Cor.\ 6.A.2]{CorFW} Every cardinal definite function on $\Z$ has the form $n\mapsto m|n|+b(n)$ for some unique $m\in\N$ and bounded non-negative function $b:\Z\to\N$ (where $\N=\{0,1,\dots\}$).

2) For every commensurating action of $\Z$, the corresponding prodigal semi-index function has, in restriction to $\N$, the form $n\mapsto m'|n|+b(n)$ for some unique $m'\in\N$ and bounded non-negative function $b':\Z\to\N$.
\end{prop}
\begin{proof}
2) Let $(Y,X)$ be a commensurating action of $\Z=\langle u\rangle$. Let $(Y,X)$ be a commensurating action of $\Z=\langle u\rangle$ and $X$ a commensurating subset; let $\ell=\ell_{Y,X}$ be the corresponding cardinal definite function and $\ell^+=\ell^+_{Y,X}$.

We start with the case when $Y$ is $\langle u\rangle$-transitive, that is, consists of a single cycle. If $Y$ is finite, then $\ell$ is bounded. 
Otherwise, $\Z$ acts simply transitively on $Y$ and hence we can suppose that $Y=\Z$ with $u(n)=n+1$. Then $X$ is a subset with finite boundary, which therefore has a finite symmetric difference with some $X'\in\{\emptyset,-\N,\N,\Z\}$. So $\ell^+_{X}-\ell^+_{X'}$ is bounded. We have $\ell^+_{X'}|_\N=0$ for $X'\in\{\emptyset,-\N,\Z\}$, and $\ell^+_{\N}(n)=n$ for all $n\in\N$. We can thus write $\ell^+(n)=m^+n+b(n)$ with $m^+\in\{0,1\}$, for all for $n\in\N$. (Note that we have a similar formula for $\ell^-$ with some $m^-\in\{0,1\}$ and that $(m^+,m^-)\in\{(0,0),(1,0),(0,1)\}$.)
A simple argument using sub-additivity of $\ell^+$ shows that $b\ge 0$.

Adding over finitely many orbits, we obtain the result when $Y$ consists of finitely many orbits.

In general, let $W$ be the union of orbits of elements of the finite subset $X\tu uX$; it consists of finitely many orbits $W_i$. Then $X\cap W^c$ is invariant, and hence we have $\ell=\ell_{W,X\cap W}$ and $\ell^+=\ell_{W,X\cap W}$. This reduces to the case when there are finitely many orbits, which has been settled.

1) From (2) and since $\ell(n)=\ell^+(n)+\ell^+(-n)$, we can write $\ell(n)=m_\pm n+b_\pm(n)$ for all $n\in\pm\N$. Since $\ell(n)=\ell(-n)$ for all $n$, taking the limit of $(\ell(n)-\ell(-n))/n$ when $n\to\infty$ yields $m_+=m_-$, and in turn we deduce $b_+=b_-$.
\end{proof}

The first consequence below was originally observed as a consequence of a more difficult result of Haglund \cite{Hag} on isometries of CAT(0) cube complexes.

\begin{cor}\label{exfw}
Let $G$ be a group and $\langle c\rangle$ a cyclic subgroup. Suppose that $c$ is distorted (as defined before Corollary \ref{distortedaff}), or $c$ is unboundedly divisible (as defined after Definition \ref{def_fw}). Then $(G,\langle c\rangle)$ has relative Property FW.
\end{cor}
\begin{proof}Let $\ell$ be a cardinal definite function on $G$.

Suppose that $c$ is distorted in some finitely generated subgroup $\Gamma$ of $G$, and let $|\cdot|$ be the word length on $\Gamma$ with respect to some finite generating subset. Since $\ell$ is subadditive, there exists $C$ such that $\ell(g)\le C|g|$ for all $g\in \Gamma$. In particular, $\ell(c^n)\le C|c^n|$. If $c$ is distorted, then $\lim |c^n|/n=0$, and we deduce $\lim\ell(c^n)/n=0$. By Proposition \ref{cardefz}, it follows that $\sup_n\ell(c^n)<\infty$.

Also by Proposition \ref{cardefz}, the limit $m(g)=\lim\ell(g^n)/n$ belongs to $\N$ for all $g\in G$. Clearly, $m(g^k)=km(g)$. It immediately follows that if $c$ has roots of unbounded order, then $m(c)=0$, and hence, again by Proposition \ref{cardefz}, we have $\sup_n\ell(c^n)<\infty$. 
\end{proof}

\begin{prop}\label{cardefab}
Let $A$ be a finitely generated abelian group and $\ell$ a cardinal-definite function on $A$. Then there exist subgroups $B,A'$ of $A$ such that\begin{itemize}
\item $B\cap A'=\{0\}$;
\item $B+A'$ has finite index in $A$;
\item the length $\ell$ is bounded on $B$;
\item the length $\ell$ has growth equivalent to the word length on $A'$.
\end{itemize}
\end{prop}
\begin{proof}
Let $B(A,\ell)$ be the maximal subgroup on which $\ell$ is bounded (it clearly exists). We first assume that $B(A,\ell)=\{0\}$. Hence, $A$ is torsion-free, so we can suppose that $A=\Z^d$, and we have to show that $\ell$ is equivalent to the word growth.

Let $f$ be a cardinal-definite function on $A$. First suppose that it is associated to a transitive commensurating action $A/E$, with commensurated subset $M$, and that $f$ is unbounded. Then $A/E$ has more than one end, and hence is 2-ended. Let $\chi:A/E\to\Z$ a surjective homomorphism (this is unique up to sign). Then, up to replace $\chi$ with $-\chi$, the subset $M$ has finite symmetric difference with $\chi^{-1}(M)$, and then $\ell(v)=m|\chi(v)|+O(1)$, where $m$ is the cardinal of the kernel of $\chi$.

In general, as in the proof of Proposition \ref{cardefz}, $f$ is equal to the cardinal-definite function associated to an action with finitely many orbits (this actually holds, by the same arguments, for arbitrary finitely generated groups). Hence there exist $k$ and homomorphisms $A\to\Z$ such that we have $f(v)=\sum_{i=1}^k|\chi_i(v)|+O(1)$. Note that $\chi_i$ extends to a linear form on $\R^d$ with integral coefficients. Then $\nu=\sum_{i=1}^k|\chi_i|$ defines a seminorm on $\R^d$, whose vanishing subspace is $\bigcap_i\mathrm{Ker}(\chi_i)$. The latter is a rational subspace. Since $\ell$ is unbounded on any nonzero subgroup of $\Z^d$, we deduce that this rational subspace is zero. Hence $\nu$ is a norm, hence $\nu\ge c\|\cdot\|_1$ for some $c>0$. Since the $\ell^1$-norm $\|\cdot\|_1$ coincides with the word length on $\Z^d$, we thus have the required inequality.

In general ($B(A,\ell)$ arbitrary), let $A'$ be any maximal subgroup among those with $B(A,\ell)\cap A'=\{0\}$. Since $B(A',\ell)=\{0\}$, we obtain the result by applying the previous case to $A'$.
\end{proof}


The following lemma is far from optimal but we find convenient to write it for reference.

\begin{lem}\label{fwrelie}
Let $G$ be a Lie group of dimension 1, of the form $G^\circ\rtimes F$ with $F$ either trivial or cyclic of order 2, acting by inversion on the abelian unit component $G^\circ$. Then $G$ has a finitely generated dense subgroup $\Gamma$ such that $(G,\Gamma)$ has relative Property FW as an abstract group.
\end{lem}
\begin{proof}
Let $f$ be a cardinal definite function on $G$. 
Choose $u,v\in G^\circ$ generating a dense subgroup $\Lambda$ of $G^\circ$, and set $\Gamma=\Lambda\rtimes F$.
Since each of $u$, $v$ is divisible in $G$, we have $f$ bounded on both $\langle u\rangle$ and $\langle v\rangle$ by Corollary \ref{exfw}. Since $f$ is subadditive and since every element of $\Gamma$ can be written as $u^nv^mk$ with $(n,m,k)\in\Z^2\times F$, we deduce that $f$ is bounded.
\end{proof}

\subsection{Lemmas of B. Neumann and P. Neumann}

We use the following lemma holding for arbitrary group actions, which for convenience we refer to as Neumann's lemma:

\begin{lem}\label{neulem}
Let $G$ be a group and $Z$ a $G$-set. Let $U$ be a cofinite subset of $Z$ including all finite $G$-orbits. Then for any finite subset $F$ of $Z$ there exists $g\in G$ such that $gF\subset U$. 
\end{lem}
\begin{proof}
A result of P.\ Neumann \cite[Lemma 2.3]{Neu76} states ($\ast$) that for every $G$-set $W$ with no finite orbit and every finite subset $F'$ of $W$, there exists $g\in G$ such that $gF'\cap F'$ is empty. (This is an easy consequence of B. Neumann's result \cite{Neu54} that a group is never covered by finitely many left cosets of infinite index subgroups.)

Let $Z_\infty$ be the union of all infinite $G$-orbits; by assumption $U\cup Z_\infty=Z$. Define $F'=(Z\smallsetminus U)\cup (F\cap Z_\infty)$. Then $F'\subset Z_\infty$, and we can apply ($\ast$) to $W=Z_\infty$ and $F'$: there exists $g\in G$ such that $gF'\cap F'=\emptyset$. In particular, $gF'\subset U$. Since $g(F\smallsetminus Z_\infty)\subset Z\smallsetminus Z_\infty\subset U$ by assumption, we deduce that $gF\subset U$.
\end{proof}

Neumann's lemma can be used to obtain separation properties:

\begin{cor}\label{neulemc}
Let $G$ be a group and $Y$ a topological space on which $G$ acts by self-homeomorphisms. Let $U$ be a cofinite Hausdorff subset of $Y$ including all finite $G$-orbits. Then $Y$ is Hausdorff.
\end{cor}
\begin{proof}
Apply Lemma \ref{neulem} to 2-element subsets of $Y$, one first obtains that for all $x\neq y$, the point $y$ does not belong to the closure of $\{x\}$. Hence $U$ is open. Again apply Lemma \ref{neulem} to 2-element subsets of $Y$ to obtain the Hausdorff Property.
\end{proof}

In turn, this has the following application.

\begin{lem}\label{haudo}
Let $G$ be a (discrete) group acting continuously on a topological $T_1$-space $Y$. Let $X,X'$ be subsets of $Y$, such that $X$ is $G$-essential (i.e., $X$ meets all $G$-orbits), $X$ is open, $X'$ is $G$-invariant, and the symmetric difference $X\tu X'$ is finite. Suppose that $X$ is Hausdorff. Then $X'$ has a cofinite $G$-invariant subset $X''$ that is Hausdorff and open in $Y$.
\end{lem}
\begin{proof}
Let $F$ be the union of finite $G$-orbits in $X'$ meeting $X'\smallsetminus X$. Define $X''=X'\smallsetminus F$. Hence $X''$ is a cofinite $G$-invariant subset of $X'$ and every element of $X''\smallsetminus X$ belongs to an infinite $G$-orbit. By Corollary \ref{neulemc} applied to $Y=X''$ and $U=X\cap X''$, we infer that $X''$ is Hausdorff.

Since $Y$ is $T_1$ (singletons are closed), $X\cap X''$ is open; hence $X''=\bigcup_{g\in G}g(X\cap X'')$ is open too.
\end{proof}

In terms of partial actions, this has the following consequence:

\begin{prop}\label{transfihau}
Consider a topological partial action of a (discrete) group $G$ on a Hausdorff topological space $X$. Suppose that $X$ is $G$-transfixed. Then there exists a $G$-invariant subset $X'$ of $\hat{X}$ with $X'\tu X$ finite, that is $G$-invariant, Hausdorff and open in $\hat{X}$.
\end{prop}
\begin{proof}
Since $X$ is transfixed, it is transfixed as a subset of its universal globalization $\hat{X}$, which is $T_1$ by Lemma \ref{glo_t1}. Let $X''$ be a $G$-invariant subset of $\hat{X}$, such that $X\tu X''$ is finite. By Lemma \ref{haudo}, $X''$ has a $G$-invariant, cofinite subset that is both Hausdorff and open in $\hat{X}$.
\end{proof}

We are led to consider partial actions that are obtained from actions by restricting domains of definitions. For instance, the group $\PC^0_{\mathbf{Aff}}(\R/\Z)$ of piecewise affine self-homeomorphisms of $\R/\Z$ acts on $\R/\Z$, and it is natural to restrict to a partial action preserving the affine structure, by restricting the domain of definition of $f$ to the set of points at which $f$ is affine (or equivalently derivable). The following proposition typically addresses such partial actions, characterizing the condition of being transfixing.

\begin{prop}\label{glo_tran}
Let $G$ be a group and $\alpha$ an action of $G$ on a set $Y$. Consider a cofinite-partial action $\beta$ of $G$ on $Y$ such that the identity map $(Y,\beta)\to (Y,\alpha)$ is $G$-equivariant.
Suppose that $Y$ is transfixed for the partial action $\beta$. Then there exist cofinite subsets $X\subseteq X'\subseteq Y$ with $X'$ $\alpha(G)$-invariant, such that the inclusion of $(X,\beta)$ into $(X',\alpha)$ is a universal globalization.

In particular if there is no finite $\alpha(G)$-orbit in $Y$ then the conclusion can be written as: there exists a cofinite subset $X\subseteq Y$ such that the inclusion of $(X,\beta)$ into $(Y,\alpha)$ is a universal globalization.
\end{prop}
\begin{proof}
Let $i:Y\to\hat{Y}$ be a universal globalization of $\beta$.
By the universal property, there exists a $G$-equivariant map $\pi:(\hat{Y},\hat{\beta})\to (Y,\alpha)$ such that $\pi\circ i=\mathrm{id}_Y$. 

For every subset $X\subseteq Y$, its universal globalization is the union of all $\hat{\beta}(G)$-orbits of $\hat{Y}$ meeting $i(X)$. Let $X$ be a cofinite subset of $Y$ that is transfixed above (i.e., such that $X$ has finite complement in its universal globalization). Removing finitely many finite $\hat{\beta}(G)$-orbits, we can suppose that $\hat{X}\smallsetminus i(X)$ meets no finite $\hat{\beta}(G)$-orbit. We claim that $\pi$ is injective on $\hat{X}$. Indeed, $\pi$ is clearly injective on $i(X)$. Since any pair in $\hat{X}$ can be moved into $i(X)$ by a group element (applying Lemma \ref{neulem} to the action $\hat{\beta}$), we deduce injectivity of $i$ on $\hat{X}$. Then define $X'=i(\hat{X})$: then $X\subseteq X'\subseteq Y$ satisfy the given statement.
\end{proof}

\subsection{Canonical partial action of self-homeomorphism groups}\label{cpsh}

Let $X$ be a topological space. Consider the group $\PC(X)$ of parcelwise continuous self-transformations of $X$, introduced in \S\ref{s_parcelwise}.

Denote by $\pi$ the projection $\ICOtop(X)\to\PC(X)$. That is, for $g\in\PC(X)$, $\pi^{-1}(\{g\})$ is the set of representatives of $g$ as partial homeomorphisms between two open cofinite subsets of $X$.

For $g\in\PC(X)$, define $\alpha(g)=\bigcup_{\sigma\in\pi^{-1}(\{g\})}\sigma$. Here the union is understood among subsets of $X^2$. 

\begin{lem}\label{lseco}
Suppose that $X$ is Hausdorff with no isolated point. Then $\alpha(g)$ is a partial bijection for every $g\in\PC(X)$.
\end{lem}
\begin{proof}
It is clear that $\alpha(g)\subset X\times X$ has cofinite (hence open) projections into $X$. So we have to show that projections are injective, and by symmetry ($\alpha(g^{-1})$ being the flip of $\alpha(g)$) it is enough to check for the first projection. To show its injectivity, consider $(x,y),(x,y')\in\alpha(g)$. Hence $(x,y)\in\sigma$ and $(x,y')\in\sigma'$ for some $\sigma,\sigma'\in\pi^{-1}(\{g\})$. Since every cofinite subset of $X$ is dense, there exists a net $(x_i)$ in the domain of $\sigma\cap\sigma'$, tending to $x$. Then, by continuity of both $\sigma$ and $\sigma'$ and using that $Y$ is Hausdorff, we have $y=\lim_i\sigma(x_i)=\lim_i\sigma'(x_i)=y'$.
\end{proof}

It is easy to check the failure of the conclusion of Lemma \ref{lseco} whenever $X$ has at least two isolated points. Thanks to Lemma \ref{lseco}, we have:

\begin{prop}\label{semic}
Let $X$ be a Hausdorff topological space with no isolated point. Then $g\mapsto\alpha(g)$ defines a topological partial action of $\PC(X)$ on $X$, which is a splitting for the canonical projection $\pi:\ICOtop(X)\to\PC(X)$.\end{prop}
\begin{proof}
By Lemma \ref{lseco}, this map is well-defined. The first two conditions of partial actions (identity and inverses) are clearly satisfied. Let us check that $\alpha(g)\alpha(h)\subset\alpha(gh)$ for all $g,h\in\PC(X)$. Consider $x\in X$ such that $y=\alpha(h)x$ and $z=\alpha(g)(\alpha(h)x)$ are defined. So $(x,y)\in\alpha(h)$ and $(y,z)\in\alpha(g)$. Hence there exist $\sigma\in\pi^{-1}(\{g\})$ and $\sigma'\in\pi^{-1}(\{h\})$ such that $(x,y)\in\sigma'$ and $(y,z)\in\sigma$. Hence $\sigma\sigma'\in\pi^{-1}(gh)$ and $(x,z)\in\sigma\sigma'$. That is, $(x,z)\in\alpha(gh)$.
\end{proof}

\begin{rem}
The above splitting is not a monoid homomorphism in general: for instance $\mathrm{id}_X=\alpha(1)=\alpha(g^{-1}g)\neq\alpha(g^{-1})\alpha(g)=\mathrm{id}_{D_g}$ when $D_g\neq X$.
\end{rem}

\begin{rem}
From a categorist's point of view, it would be more natural and general to formulate this in a groupoid context, where a partial groupoid action is defined assigning to each object a set and to each arrow a partial bijection between the corresponding sets, with the analogous axioms. Then Proposition \ref{semic} adapts to this more general setting: every parcelwise continuous map $g:X\to Y$ between Hausdorff topological spaces without isolated point has a canonical representative $\alpha(g)$ in $\ICOtop(X)$, satisfying the axioms of groupoid partial action. We sticked to the group case only for the sake of conciseness, and because the generalization requires no further ingredient.
\end{rem}

\subsection{Generalities about the parcelwise and piecewise groups}

\begin{defn}Let $X$, $Y$ be Hausdorff topological spaces. For functions $f,g\in Y^X$, write $f\sim g$ if $f$ and $g$ coincide on a cofinite subset, and let $[f]$ be the class of $f$ modulo $\sim$.

Say that $f\in Y^X$ is outer continuous at $x\in X$ if $x$ is not isolated and there exists a map $g\in [f]$ that is continuous at $x$. Then the value $g(x)$ does not depend on the choice of $g$ and only on the class $[f]$, and is called the outer limit of $f$ (or of $[f]$) at $x$. Let $D^\circ_f$ (or $D^\circ_{[f]}$) be the set of points at which $f$ is outer continuous.
\end{defn}

Let $X$ be a Hausdorff topological space with no isolated point. Every $\sigma\in\PC(X)$ defines a class modulo $\sim$ in $X^X$ in the above sense.
Then, for $\sigma\in\PC(X)$, the set $D^\circ_\sigma$ of points on which $\sigma$ is outer continuous is cofinite, since it contains the domain of definition of any representative $\hat{\sigma}$. 
In the setting of Proposition \ref{semic}, $D_{\alpha(\sigma)}$ is thus contained in $D^\circ_\sigma$.

Under a strong assumption, we have a converse. Define a topological space $X$ as locally saturated if for every open subset $U$ of $X$, every continuous injective map $U\to X$ is open. 
For instance, topological manifolds of pure dimension $n$ are locally saturated, by Brouwer's invariance of domain theorem (we will use it for $n=1$, in which case this is obvious).

\begin{prop}\label{paroutcon}
Let $X$ be a Hausdorff topological space with no isolated point. Suppose that $X$ is locally saturated. Then for every $\sigma\in\PC(X)$, the domain of definition of $\alpha(\sigma)$ coincides with the subset $D^\circ_\sigma$ of outer continuity of $\sigma$. \end{prop}
\begin{proof}
We only have one inclusion to check. Suppose that $\sigma$ is outer continuous at $x$. Let $\tilde{\sigma}$ be a lift of $\sigma$ in $\ICOtop(X)$. If $\tilde{\sigma}$ is defined at $x$, then $x\in D_{\alpha(\sigma)}$ (by construction of the latter). Otherwise, let $y$ be the outer limit of $f$ at $x$. Define $f=\hat{\sigma}\cup\{(x,y)\}$ (viewed as subsets of $X^2$). Then $f$ is a partial map and we have to check that $f$ is injective and $f^{-1}$ is continuous on its domain $D$.

Define a neighborhood $U\subset D$ of $x$ as follows: if $f$ is injective, choose $U=D$. Otherwise, there exists a unique $x'\in D\smallsetminus\{x\}$ such that $f(x)=f(x')$. Since $X$ is Hausdorff, let $U$ be any open neighborhood of $x$ in $D$ whose closure does not contain $x'$. Hence $f$ is injective on $U$ in all cases.
Since $X$ is locally saturated, $f(U)$ is open. Hence $V=\sigma^{-1}(f(U))$ is open; we have
\[V=\sigma^{-1}(f(U\smallsetminus\{x\})\cup\{y\})=(U\smallsetminus\{x\})\cup\sigma^{-1}(\{y\}).\]

If by contradiction $x'$ exists, $\sigma^{-1}(\{y\})=\{x'\}$, and hence we deduce that $(U\smallsetminus\{x\})\cup\{x'\}$ is open. Intersecting with the complement of the closure of $U$, we deduce that $\{x'\}$ is open, a contradiction.

Now let us show that $f^{-1}$ is continuous at $y$. Indeed, $f(U)$ is a neighborhood of $y$ for every neighborhood $U\subset D$ of $x$ and this precisely establishes continuity of $f^{-1}$.
\end{proof}

We finish this subsection by a digression, namely a corollary of Proposition \ref{semic} in terms of near actions. Recall \cite{CorNA} that a near action of a group $G$ on a set $X$ is a homomorphism from $G$ into the group $\PC(X)$, where $X$ is endowed with the discrete topology. Every action of $G$ on $X$ induces a near action, and a near action arising in this way is called realizable.

Given two near actions on sets $X$ and $Y$, one can naturally define the disjoint near action on $X\sqcup Y$. The near action on $X$ is called completable if there exists a set $Y$ and a near action on $Y$ such that the disjoint union near action on $X\sqcup Y$ is realizable.

If $X$ is a topological space, write $X_\delta$ for $X$ endowed with the discrete topology, so we have a canonical inclusion $\PC(X)\subset\PC(X_\delta)$, which thus defines a near action of $\PC(X)$ on $X$.

\begin{cor}
For every Hausdorff topological space $X$ with no isolated point (or with finitely many isolated points), the near action of $\PC(X)$ on $X$ is completable.
\end{cor}

Indeed, Proposition \ref{semic} provides a cofinite-partial action realizing the near action of $\PC(X)$, and hence taking the universal globalization yields completability. See \cite[\S1.P, \S4.K]{CorNA} for more on the link between near actions and cofinite-partial actions. In the presence of isolated points, the near action of $\PC(X)$ on $X$ is not always completable; notably, it is not completable when $X$ is infinite discrete, see \cite{CorNA}.

Given a Hausdorff topological space $X$ with no isolated point, we have inclusions $\mathrm{Homeo}(X)\subset\PC_\sharp(X)\subset\PC(X)$. When $X$ is $S$-modeled, it induces inclusions $\PC_S^0(X)\subset\PC_{S\sharp}(X)\subset\PC_S(X)$.

We say that a Hausdorff topological space $X$ has no local cut point if it satisfies one of the following two equivalent conditions:
\begin{itemize}
\item every $x\in X$, the set of neighborhoods $V$ of $x$ such that $x\in\overline{V\smallsetminus\{x\}}$ and $V\smallsetminus\{x\}$ is connected, is a basis of neighborhoods of $x$;
\item for every $x\in X$ and every neighborhood $W$ of $x$, there is a unique component $Z$ of $W\smallsetminus\{x\}$ such that $x$ belongs to the closure of $Z$.
\end{itemize}

Note that it implies that $X$ has no isolated point. For instance, this holds if $X$ is a topological manifold with no component of dimension $\le 1$, or more generally if $X$ is locally homeomorphic to a locally finite simplicial complex in which every vertex or edge belongs to a triangle, and in which the link at every vertex is connected.

\begin{rem}\label{r_homeo_pc}
Let $X$ be a Hausdorff compact topological space with no local cut point. Then the inclusion $\mathrm{Homeo}(X)\subset\PC(X)$ (holding whenever $X$ is Hausdorff and has no isolated point) is an equality.

Indeed, given that $X$ is Hausdorff compact, the assumption that $X$ has no local cut point is equivalent to the condition that for every finite subset $F$ of $X$, the embedding $X\smallsetminus F\to X$ is the end compactification (in the sense of Specker \cite{Spe}) of $X\smallsetminus F$. 

The equality $\PC(X)=\mathrm{Homeo}(X)$ is a rigidity property, which means, in a sense, that the parcelwise continuous group $\PC(X)$ does not deserve a specific study, and illustrates by contrast, the richness of the context of purely 1-dimensional topological manifolds $X$.
\end{rem}

\begin{rem}\label{rempicon}
Let $X$ be a Hausdorff topological space with no local cut point. Then the inclusion $\mathrm{Homeo}(X)\subset\PC_\sharp(X)$ is an equality.

Indeed, consider $f\in\ICO_\sharp(X)$ and $x\in X$. Consider $I$, $(D_i)$, $(g_i)$ and $(E_i)$ as in \S\ref{s_inmopa}. Since $I$ is finite and $X$ has no isolated point, the set $J$ of $i$ such that $x\in\overline{D_i}$ is not empty. Hence $g_i(x)$ is a limit point of $f(x')$ when $x'\neq x$ tends to $x$; since $x$ is not a local cut point, $g_i(x)$ does not depend on $i$; call it $\bar{f}(x)$; note that $\bar{f}$ extends $f$.
Again using that $I$ is finite one checks that $\bar{f}$ is continuous at $x$. Then $\bar{f}\circ \bar{f}^{-1}$ and $\bar{f}\circ \bar{f}^{-1}$ are defined everywhere, are continuous, and are the identity outside finite subsets. Since $X$ has no isolated point, these are identity. Thus $\bar{f}$ is a homeomorphism. 
\end{rem}

\subsection{On parcelwise-$S$ and piecewise-$S$ groups}

\subsubsection{Partial action $\alpha_S$}\label{s_alphas}

Let $S$ be a pseudogroup on a topological space $A$, and let $X$ be a $S$-modeled topological space. Denote by $\pi_S$ the projection $\ICO_S(X)\to\PC_S(X)$. By definition, a topological partial action on $X$
is $S$-preserving if it maps $G$ into $\mathcal{I}_{S}(X)$.

To any topological $S$-preserving cofinite-partial action of $G$ on $X$, we can consider the induced homomorphism $G\to\PC_S(X)$.

Now assume that $X$ is Hausdorff and has no isolated point.
For $g\in\PC_S(X)$, define $\alpha_S(\sigma)=\bigcup_{\sigma\in\pi_S^{-1}(\{g\})}\sigma$. It is contained in 
$\alpha(g)=\bigcup_{\sigma\in\pi^{-1}(\{g\})}\sigma$ (i.e., where we take the union over the whole preimage in $\ICOtop(X)$). It follows from Lemma \ref{lseco} that if $X$ is Hausdorff and has no isolated points, then $\alpha(g)$ is a partial bijection, and hence $\alpha_S(g)$ is a partial bijection. 

This shows that for $X$ Hausdorff and without isolated point, the above mapping from the set of topological $S$-preserving partial actions of $G$ on $X$ to $\mathrm{Hom}(G,\PC_S(X))$ has a canonical section, and in particular is surjective.

It can happen, and it is one interest of the construction, that $\alpha_S(g)$ is properly contained in $\alpha(g)$. For instance, if $g$ is a piecewise affine homeomorphism and $S$ is the pseudogroup of local affine homeomorphisms, then $\alpha(g)$ is defined everywhere, while $\alpha_S(g)$ is defined outside singular points.

\subsubsection{Transfer of $S$-structure to the globalization}

The following simple proposition plays an essential role; it is also paramount to {\it not} assume $Y$ to be Hausdorff. For this reason, we provide a detailed proof.

\begin{prop}\label{exstru}
Let $S$ be a pseudogroup on a topological space $A$. Let $G$ be a group with a topological $S$-preserving partial action on an $S$-modeled space $X$. Consider a partial topological $G$-space $Y$ with an injective open full essential homomorphism of partial actions $X\to Y$. Then there is a unique $S$-structure on $Y$ extending the $S$-structure on $X$ such that the partial action of $G$ is $S$-preserving.
\end{prop}
\begin{proof}
View $X\to Y$ as an open inclusion, and denote by $\alpha$ the partial action on $Y$, and $\beta$ the partial action on $X$. For every $y\in Y$, the inclusion being essential, there exists $g\in G$ such that $\alpha(g)y\in X$. Since $X$ and the $D_{\alpha(g)}$ are open and $\alpha(g)$ is continuous on its domain, there exists ($\ast$) an open neighborhood $U$ of $y$ included in $D_{\alpha(g)}$ such that $\alpha(g)U\subset X$. Given ($\ast$), the uniqueness follows.
 
For the existence, it is enough to endow $Y$ with a structure on the pseudogroup $\PB_S(X)$. Namely, define an atlas where the charts are indexed by the pairs $(g,U)$, where $g$ ranges over $G$ and $U$ among open subsets of $Y$ such that $\alpha(g)U\subset X$. Such a chart $\phi_{g,U}$ has domain $U$ and is simply given by $\alpha(g)|_U$. That the domain of charts cover $Y$ follows from ($\ast$). 

Now let us check the compatibility condition in the definition of an atlas. Consider two charts $\phi_{g,U}$ and $\phi_{h,V}$: we have to show that $\phi_{h,V}\circ\phi_{g,U}^{-1}$ belongs to $\PB_S(X)$. Define $W=\alpha(g)U\cap\alpha(h)V$. Define $U'=\alpha(g)^{-1}(W)\subset U$ and $V'=\alpha(h)^{-1}(W)\subset V$. Then, adding subscripts to restrict the domain and range of the given partial bijections, we have
\[\phi_{h,V}\circ\phi_{g,U}^{-1}=\phi_{h,V'}\circ\phi_{g,U'}^{-1}=\big(\alpha(h)_{V'\to W}\big)\circ\big(\alpha(g)^{-1}_{U'\to W}\big)=\alpha(h\circ g^{-1})_{W\to W}.\]
The first equality just uses the definition of composition of partial maps; the second is just the definition, and the third follows from the definition of partial maps and the definition of partial action. Now it follows from the inclusion $X\to Y$ being full that $\alpha(h\circ g^{-1})_{W\to W}$ is equal to $\beta(h\circ g^{-1})_{W\to W}$. Therefore, the change of charts $\alpha(h\circ g^{-1})_{W\to W}$ belongs to $\PB_S(X)$. This proves that we have an atlas.

Let us finally check that the partial action $\alpha$ is $S$-preserving. Indeed, since this is a local condition, we can check for a given $g$ at given points $y,y'=\alpha(g)y$; we can choose an open neighborhood $U$ of $y$ which is the domain of a chart $\phi_{U,h}$, and such that $V=\alpha(g)U$ is domain of a chart $\phi_{V,k}$. Then, on $h(U)$, the composition $\phi_{V,k}\circ(\alpha(g)_{U\to V})\circ\phi_{U,h}^{-1}$ (which describes $\alpha(g)_{U\to V}$ in charts) equals $\alpha(kgh^{-1})_{h(U)\to k(V)}$ which preserves $S$. Hence $\alpha(g)$ is $S$-preserving at the neighborhood of $y$. 
\end{proof}

\begin{cor}\label{glostru}
Given a group $G$ with a topological $S$-preserving partial action on an $S$-modeled topological space $X$, there exists a unique $G$-invariant $S$-structure on the universal globalization $\hat{X}$ extending the original $S$-structure on $X$.
\end{cor}

\begin{rem}\label{r_hausdorff}
Beware that even if $X$ is Hausdorff, $\hat{X}$ is often far from Hausdorff: indeed the construction of $\hat{X}$ typically glues copies of open subsets of $X$ along open intersections. 

For this reason, and because of our extensive use of Corollary \ref{glostru}, discussions about Hausdorffness of spaces are important unavoidable issues: even if the ultimate goal is to deal with Hausdorff spaces and produce Hausdorff spaces, we have to accept the presence of non-Hausdorff spaces among our tools.
\end{rem}

\subsubsection{Further remarks}

The following lemma complements Lemma \ref{bousch} and is immediate.
\begin{lem}\label{bousch_s}
Let $X$ is modeled over a pseudogroup $S$ on $\R/\Z$. Then the natural homomorphism $\Phi:\PC(X)\to\PC^+(X^\pm)$ induces an isomorphism from $\PC_S(X)$ onto the centralizer of $\sigma$ in $\PC_S^+(X^\pm)$.\qed 
\end{lem}

\begin{rem}
If $A=\R/\Z$ and $S$ is one of the pseudogroups $\mathbf{Isom}$, $\mathbf{Aff}$, $\mathbf{Proj}$, the whole pseudogroup of local homeomorphisms, or one of their oriented counterparts, then for every $S$-modeled topological space $X$, we have $\PC_S(X)=\PC_{S\sharp}(X)$.

Continuing with $\R/\Z$, examples of $S$ for which $\PC_S(X)\neq\PC_{S\sharp}(X)$ are the pseudogroup of partial $\mathcal{C}^k$-diffeomorphisms for $k>0$ (or $k=\infty,\omega$), the pseudogroups $\mathbf{Aff}^\wp$ and $\mathbf{Proj}^\wp$ of piecewise affine/projective local homeomorphisms.
\end{rem}

\begin{rem}%
Even when $\PC(X)=\mathrm{Homeo}(X)$ (as in Remark \ref{r_homeo_pc}), we can have $\PC_S(X)\neq\PC_S^0(X)$. This is for instance the case when $S$ is the pseudogroup of $C^k$-diffeomorphisms on $\R^d$ for $d\ge 2$ and $k>0$.

In higher dimension, of course there are other natural interesting ways of considering ``piecewise" properties, which allow infinite subsets (typically ``codimension 1") subsets of ``singular points". Such a study goes beyond the scope of this paper.
\end{rem}

\section{Non-distortion phenomena}\label{s_nondile}
This shorter section is independent of the next ones, and only referred to in the non-distortion corollaries. Indeed, while commensurated actions typically allow to prove that distorted elements preserve some given geometric structure, an additional step is necessary to understand when such automorphisms are distorted within the whole piecewise group.

We indicate a way to systematically tackle such problems, with a limited technical cost, thanks to the language of pseudogroups. 

Let $S$ be a pseudogroup on a topological space $X$. Whenever we refer to $sx$ for $(s,x)\in S\times X$, it is understood that we mean ``(such that) $x$ belongs to the domain $D_s$ of $s$". We freely consider $s$ as a subset of $X^2$.

Given open subsets $Y,Z$ of $X$, we denote by $S_{Y,Z}$ the set of those $s\in S$ that are included in $Y\times Z$. Write $S_Y=S_{Y,Y}$ (so $S=S_X$); note that $S_Y$ is a pseudogroup on $Y$, and is an inverse subsemigroup of $\PB(X)$.

Given a subset $T\subset\PB(X)$, define a graph structure on $X$, with one edge $(x,gx)$ for all $g\in T$ and $x\in X$. Let $d_T$ be the corresponding graph ``distance" (allowing the value $\infty$) on $X$. Note that $d_{T}=d_{T\cup T^{-1}}$, so it is generally no restriction to assume $T$ symmetric.

\begin{lem}\label{third}
Let $X$ be a set (viewed as discrete topological space) with a subset $Y$. Let $S$ be a pseudogroup on $X$, with $\mathrm{id}_Y\in S$.

Let $T$ be a subset of $S$ and $K$ a subset of $S_{X,Y}$, such that $\bigcup_{f\in K}D_f=X$. Define $T'=KTK^{-1}$ (so $T'$ is included in $S_Y$). Then for all $y,y'\in Y$ we have $d_{T'}(y,y')\le 3d_T(y,y')$. 
\end{lem}
Note that assuming $\mathrm{id}_Y\in K$ yields $d_{T}(y,y')\le d_{T'}(y,y')$; the interest of Lemma \ref{third} is to provide an inequality in the reverse direction. This can be thought as a non-distortion property: we can replace a path of size $n$ (in $X$) with a path of size $3n$ (within $Y$).

\begin{proof}
Since $T'\cup {T'}^{-1}\subset K(T\cup T^{-1})K^{-1}$, we can suppose that $T$ is symmetric. Consider $y,y'\in Y$ with $d_T(y,y')=n$, so we can write $y'=s_n\dots s_1y$ with $s_i\in T$. Write $x_j=s_j\dots s_1y\in X$. For each $j$, there exists $k=k_j\in K$ such that $x_j\in D_{k_j}$. Then 
\[y'=(s_nk_{n-1}^{-1})(k_{n-1}s_{n-1}k_{n-2}^{-1})\dots(k_2s_2k_1^{-1})(k_1s_1)y=\tau_n\dots\tau_1y,\] 
with $\tau_i=k_is_ik_{i-1}^{-1}\in T'$ (and $k_0,k_n=\mathrm{id}_Y$). Then $\tau_j\dots\tau_1y\in Y$ for all $j$. Hence $d_{T'}(y,y')\le 3n$.
\end{proof}

Let $X$ be a standard curve. Define a small interval in $X$ as a subset $I$ either empty or homeomorphic to an open interval, with the additional requirement that if $I$ is included in a circle component of length $a$, then the length of $I$ is $\le a/2$. This ensures that the intersection of any two small intervals is a small interval.
Define the pseudogroup of isometries $S$ as consisting of those isometries between two small intervals of $X$.

\begin{lem}\label{ineqiet}
Let $X$ be a standard curve, with the above pseudogroup $S$. Let $Y$ be a circle component of $X$.
For $z\in Y^\pm$, let $Q_z$ be the set of $\sigma\in\PC_S(X)$ such that $\sigma(z)\in Y^\pm$. For $\sigma\in Q_z$, define $h=g_z(\sigma)$ as the unique isometry $h$ of $Y$ such that $h(z)=\sigma(z)$.

Then for every finite subset $W$ of $\PC_S(X)$, there exists a finite subset $W'$ of $\mathrm{Isom}(Y)$ such that for every $z\in Y^\pm$ and every $\sigma\in Q_z$, we have $|\sigma|_W\ge\frac13|g_z(\sigma)|_{W'}$.
\end{lem}
\begin{proof}
Let $W$ be a finite subset of $\PC_S(X)$. Let $T$ be a finite subset of $S$ such that every element of $W$ is the union of finitely many elements of $T$.
 Let $K$ a finite subset of $S$ such that the range of every $k\in K$ is included in $Y$, and the domains of $k\in K$ cover $X$. Set $T'=KTK^{-1}\subset S_Y$. For $t\in S_Y\smallsetminus\{\emptyset\}$, define $\Psi(t)$ as the unique self-isometry of $Y$ extending $t$; define $W'=\Psi(T')$. Then, for every $z\in Y^\pm$ and $\sigma\in Q_z$, we have 
 \[|\sigma|_W\ge d_T(z,\sigma z)\ge\frac13d_{T'}(z,\sigma z))=\frac13|g_z(\sigma)|_{W'}.\]
Let us justify each of the (in)equalities above. The middle inequality is provided by Lemma \ref{third}. The left-hand inequality follows from the case when $\sigma\in W$, in which case it holds by definition of $T$. For the right-hand equality, the inequality $\le$ is easy and not needed, so let us only justify $\ge$. Indeed, suppose that $d_{T'}(z,\sigma z)=n$. Then we can write $\sigma(z)=t_n\dots t_1 z$ with $t_i\in T'$. Write $\tau_i=t_i\dots t_1$ and $z_i=\tau_i z$. So 
\[g_z(\tau_i)z=z_i=t_i\tau_{i-1}z=\Psi(t_i)g_z(\tau_{i-1})z.\]
Thus $g_z(\sigma)z=\Psi(t_n)\dots\Psi(t_1)z$, and hence, again using that the isometry group of $Y$ acts freely on $Y$, we deduce $g_z(\sigma)=\Psi(t_n)\dots\Psi(t_1)$. Hence $|g_z(\sigma)|_{W'}\le n$.
\end{proof}

\begin{prop}\label{undiiet}
Let $X$ be a standard curve and $Z$ a clopen subset of $X$. Then $\mathrm{Isom}(Z)$ is undistorted in $\PC_S(X)=\IETbw(X)$. More generally, 
let $\Gamma$ be any subgroup of $\mathrm{Isom}(Z)$ and homomorphism $q=\Gamma\to\IETbw(X\smallsetminus Z)$, and denote by $\Gamma_q$ the image of $\Gamma$ in $\IETbw(X)$ by the homomorphism $\mathrm{id}\times q$. Then $\Gamma_q$ is undistorted in $\IETbw(X)$. 
\end{prop}
\begin{proof}
We can suppose, passing to a subgroup of finite index, that $\Gamma$ is included in the unit component $\Isom(Z)^\circ$. In particular, $\Gamma$ preserves each component of $Z$, and acts trivially on any component of $Z$ that is not a topological circle. Let $Y_1,\dots,Y_k$ be the circle components of $Z$; we can view $\Isom(Z)^\circ$ as the product $\prod_j\Isom(Y_j)^\circ$. Fix $z_j\in Y_j$.

Let $W$ be a finite subset of $\PC_S(X)$. For each $j$, apply Lemma \ref{ineqiet}, outputting a finite subset $W_j$ of $\mathrm{Isom}(Y_j)$; we can suppose that $W_j$ is symmetric and contains 1. For $\gamma\in\Gamma$, let $\gamma_q$ be its image in $\IETbw(X)$; for each $j$ let $\gamma_j\in\Isom(Y_j)$ be its restriction to $Y_j$. Then $\gamma_q\in Q_z$ and hence Lemma \ref{ineqiet} says that
$|\gamma_q|_W\ge\frac13|g_{z_j}(\gamma_q)|_{W_j}$. Since $g_{z_j}(\gamma_q)=\gamma_j$, this yields $|\gamma_q|_W\ge\frac13|\gamma_j|_{W_j}$. Write $W'=\prod_j(\{1\}\cup W_j)$; then for every $\gamma$ we have $|\gamma|_{W'}=\sup_j|\gamma_j|_{W_j}$. Hence we deduce $|\gamma_q|_W\ge\frac13|\gamma|_{W'}$. If $\Gamma=\mathrm{Isom}(Y)$, this gives the non-distortion result. In general, this follows by using that in a virtually abelian group, all subgroups are undistorted: precisely, the non-distortion ensures that there exists a finite subset $W''$ of $\Gamma$ and $C>0$ such that for all $g\in\Gamma$, we have $|g|_{W'}\ge C|g|_{W''}$.
\end{proof}

\begin{cor}[Novak \cite{Nov}]
Let $X$ be a standard curve and $f$ an self-isometry of $X$ of infinite order. Then $f$ is undistorted in $\IETbw(X)$.\qed
\end{cor}

This is, modulo the formulation, due to Novak in the context of (piecewise orientation-preserving) interval exchanges. Precisely, Novak's proof that $\IET$ has no distorted cyclic subgroups consists, in a first step, in showing that every distorted element has an isometric model and then the next step is to prove (by hand: \cite[\S 4]{Nov}) a result akin to the previous corollary. Using Proposition \ref{undiiet}, Novak's result is improved in Corollary \ref{abiet}.

We now consider affinely modeled curves. They are defined here as $\mathbf{Aff}_{\R/\Z}$-modeled curves; in particular finitely-charted is meant in the definition of \S\ref{si_pseudo}: it just means that no component is affinely isomorphic to an unbounded interval of $\R$. Thus (see Appendix \ref{s_projcur}), every component is affinely isomorphic to:
\begin{itemize}
\item a bounded interval (which is isomorphic to $\mathopen]-1,1\mathclose[$), or
\item the standard circle $\R/\Z$, or
\item a non-standard circle $\R_{>0}/\langle s\rangle$ for some (unique) $s>1$.
\end{itemize}
A bounded interval has a trivial $\mathbf{Aff}^+$ automorphism group. In the case of circles, which are given here as topological groups, in each case the $\mathbf{Aff}^+$ automorphism group consists of the left-translations (additive in the standard case, multiplicative in the non-standard case).

\begin{lem}\label{affine_nondis}
Let $X$ be a finitely-charted affinely modeled curve with a clopen subset $Y$. Let $Y_1,\dots,Y_m$ be the distinct components of $Y$, each of which being isomorphic to a non-standard circle, and fix $y_j\in Y_j$ for each $j$.

Let $T$ a finite subset of the pseudogroup $S$ of affine isomorphisms between open intervals in $X$. Then there exists a finite subset $W$ of $\Aut_{\mathbf{Aff}}(X)^\circ$ such that for every $\sigma\in\Aut_{\mathbf{Aff}}(X)^\circ$ we have 
\[|\sigma|_W\le \sum_{j=1}^md_T(y_j,\sigma y_j).\].
\end{lem}
\begin{proof}
We choose coordinates, so that $X$ is identified to a suitable finite disjoint union of intervals, where $Y_j$, which is isomorphic to $\R_{>0}/\langle s_j\rangle$ for some $s_j>1$, corresponds to the interval $[c_j,s_jc_j\mathclose[$, the identification being affine outside the discontinuity point. We can choose them so that the $T$-orbit of $y_j$ does not contain $c_j$. For $\nu>0$, let $r_\nu^{(j)}$ be the multiplicative rotation of $Y_j$ given by multiplication by $\nu$: in the given coordinates, we still (by abuse) denote it as $r_\nu^{(j)}$. If $\nu\in [1,s_j\mathclose[$ (which we can always suppose), it is explicitly given by: for $y\in [c_j, s_jc_j/\nu\mathclose[$ we have $r_\nu^{(j)}(y)=\nu y$ while for $y\in [s_jc_j/\nu,s_jc_j\mathclose[$, we have $r_\nu^{(j)}(y)=(\nu/s_j) y$. Let $M\subset\R^*$ be the set of slopes of elements of $T$ in these coordinates; this is a finite subset. Define $W=\{r_\nu^{(j)}:\nu\in M,1\le j\le m\}$.

We start from the observation that any affine automorphism of $Y_j$ is determined by its slope at $y_j$ in these coordinates. Indeed, for $\nu\in [1,s_j\mathclose[$, if $\nu<s_jc_j/y_j$, then the slope of $r_\nu^{(j)}$ at $y_j$ is $\nu\ge 1$; if $\nu=s_jc_j/y_j$, then $r_\nu^{(j)}$ is discontinuous at $y_j$; if $\nu>s_jc_j/y_j$, then the slope of $r_\nu^{(j)}$ at $y_j$ is $\nu/s_j<1$. 

Fix $j$. Define $n_j=d_T(y_j,\sigma y_j)$. Write $\sigma y_j=t_{n_j}^{(j)}\dots t_1^{(j)} y_j$. Let $a_i^{(j)}$ be the slope of $t_i^{(j)}$ at $t_{i-1}^{(j)}\dots t_1^{(j)}y_j$. Write $a^{(j)}=a_{n_j}^{(j)}\dots a_1^{(j)}$. Then the slope of $\sigma$ at $y_j$ is $a^{(j)}$. Since $\sigma$ is determined on $Y_j$ by its slope at $y_j$, we deduce that $\sigma=r_{a^{(j)}}$ on $Y_j$.
Hence
\[\sigma=\prod_{j=1}^mr_{a^{(j)}}=\prod_j\prod_{i=1}^{n_j}r_{a_i^{(j)}},\]
so 
\[|\sigma|_W\le \sum_{j=1}^m\sum_{i=1}^{n_j}|r_{a_i^{(j)}}|_W\le \sum_{j=1}^m\sum_{i=1}^{n_j}1=\sum_{j=1}^mn_j=\sum_jd_T(y_j,\sigma y_j).\qedhere\]
\end{proof}

\begin{cor}\label{cor_affundist}
Let $X$ be a finitely-charted affinely modeled curve. Let $Z$ be a clopen subset of $X$ such that no component of $Z$ is affinely isomorphic to a standard curve.
Then $\Aut_{\mathbf{Aff}}(Z)$ is undistorted in $\PC_{\mathbf{Aff}}(X)$.
\end{cor}
\begin{proof}
It is enough to show that $\Aut(Z)^\circ$ is undistorted. Hence, we can suppose that $Z$ has no interval component, so each of its components is a non-standard circle.

Let $W'$ be a finite subset of $\PC_{\mathbf{Aff}}(X)$. Let $T$ be the set of partial affine isomorphisms extracted from $W'$. Let $W$ be given by Lemma \ref{affine_nondis}. Then for every $\sigma\in \PC_{\mathbf{Aff}}(X)$ 
\[|\sigma|_W\le \sum_{j=1}^md_T(y_j,\sigma y_j)\le m|\sigma|_{W'}.\qedhere\]
\end{proof}

Specifying to the subgroup of continuous elements, and to cyclic subgroups,  this yields the following particular case which, modulo the formulation, is due to Guelman-Liousse \cite[\S 7]{GuL}.

\begin{cor}[Guelman-Liousse]\label{affinondi}
Let $X$ be a finitely-charted affinely modeled curve. 
Let $Y$ a component of $X$ that is a nonstandard circle, i.e., $Y\simeq\R_{>0}/\langle s\rangle$ for some $s>1$. Then every cyclic subgroup of $\Aut_{\mathbf{Aff}}(Y)$ is undistorted in $\PC^0_{\mathbf{Aff}}(X)$.
\end{cor}

The robustness of the method allows to apply it in some other cases. For instance, a related argument can show that irrational rotations of non-standard circles are undistorted in the group of piecewise projective self-transformations. The proof is a little more difficult: the basic idea is to use non-distortion of homotheties in $\PSL_2(\R)$ rather than in an abelian group. Actually, this should be performed in a systematic study of distortion in groups of piecewise projective self-transformations.

It is not known if Corollary \ref{cor_affundist} holds when standard circles are allowed in $Z$; see Corollary \ref{eqque}.



\section{The main theorem and applications}
\subsection{Powers and the first corollaries}
 
We first use the interpretation of piecewise/parcelwise actions as partial actions to provide counting results for the number of ``singularities" of various maps, in various senses. 
 
\subsubsection{Using the partial action $\alpha$}
 
Let $X$ be a Hausdorff topological space with no isolated point. Proposition \ref{semic} applies: $\PC(X)$ has a canonical partial action $\alpha$ on $X$. Call the finite complement of the domain of definition of $\sigma\in\PC(X)$ its domain of indeterminacy. It contains the subset of outer discontinuity points of $\sigma$, defined as the complement of the set of outer continuity points $x$ of $\sigma$. If in addition $X$ is locally saturated (e.g., a topological manifold of pure dimension), Proposition \ref{paroutcon} applies and these two finite subsets coincide for every $\sigma$.

When $X$ is an oriented 1-dimensional manifold, not necessarily connected, outer discontinuity points of $\sigma$ are the same as discontinuity points of the unique left-continuous representative of $\sigma$. (Beware that mapping $\sigma$ to its unique left-continuous representative of $\sigma$ is not a monoid homomorphism; yet it is a monoid homomorphism in restriction to piecewise orientation-preserving elements).

\begin{cor}\label{semiprod}
Let $X$ be a Hausdorff topological space with no isolated point.
Consider the cofinite-partial action of $\PC(X)$ on $X$.
\begin{enumerate}
\item the prodigal semi-index function $\ell^-$ of this partial action (as defined in \S\ref{encocp}), and hence of its restriction to any subgroup, coincides with the function mapping $\sigma\in\PC(X)$ to the number of indeterminacy points of $\sigma$;
\item\label{trxiff} Let $G$ be a group with a homomorphism $\alpha:G\to\PC(X)$. Then 
$X$ is transfixed by $\alpha(G)$ (as defined in \S\ref{encocp}) if and only if the number of indeterminacy points of $\alpha(g)$ is bounded independently of $g\in G$;
\item\label{semiproi2} for every $\sigma\in\PC(X)$ there exists $m_\sigma\in\N$ and a bounded function $b:\N\to\N$ such that $\ell^-(\sigma^n)=m_\sigma n+b(n)$ for all $n\in\N$. 
\end{enumerate}
\end{cor}
\begin{proof}
The first fact is immediate and the second immediately follows. Using the universal globalization (Theorem \ref{tunivglo}), the prodigal semi-index function $\ell^-$ of this partial action is the prodigal semi-index action of some commensurating action (namely on the universal globalization of $X$, commensurating $X$). Hence (\ref{semiproi2}) follows from Proposition~\ref{cardefz}.
\end{proof}

This has the following addendum (for which we did not attempt to find optimal hypotheses):

\begin{cor}\label{growthdisc}
In the setting of Corollary \ref{semiprod}(\ref{semiproi2}), assume that $X$ is a topological manifold with no boundary and finitely many ends. Then $\ell^-(\sigma)=\ell^-(\sigma^{-1})$ for all $\sigma$. In particular, we have $m_\sigma=m_{\sigma^{-1}}$ for all $\sigma$. In other words, for every $\sigma\in\PC(X)$ there exists $m_\sigma\in\N$ and an even bounded function $b:\Z\to\N$ such that $\ell^-(\sigma^n)=m_\sigma |n|+b(n)$ for all $n\in\Z$. 
\end{cor}
\begin{proof}
This reflects the fact that domains of definition of $\sigma$ and $\sigma^{-1}$ have complements of the same cardinal. In turn, this follows from the fact that the complement of $m$ and $m'$ points in $X$ are homeomorphic only if $m=m'$. Let us check the latter assertion.

First suppose that $X$ has constant dimension. Write $\theta=2$ if $\dim(X)=1$ and $\theta=1$ if $\dim(X)\ge 2$. Let $k$ be the number of ends of $X$. Then the number of ends of $X$ minus $m$ points is $k+\theta n$. This number retains $m$, when $X$ is given. (If $X$ has dimension 0, the condition of having finitely many ends means that $X$ is finite and the result holds too.) The case when $X$ has variable dimension immediately follows.
\end{proof}

\begin{rem}
In the case when $X=\R/\Z$ and in the context of interval exchanges,
Corollary \ref{growthdisc}, was essentially established independently in \cite[Prop.\ 2.3]{Nov} and \cite[Corollary 2.5]{DFG}. Although not stated, the behavior of the form $n\mapsto kn+O(1)$ for $n\to+\infty$ with $k\in\N$ is established explicitly in \cite{DFG} and follows from the proof in \cite{Nov}.

The symmetry established in Corollary \ref{growthdisc}, as well as the non-negativity of $b$ seem to be (minor) new observations. The generalization to $\PC(\R/\Z)$ is significant; however it seems that the methods used in both references can be applied with minor changes, at least in the piecewise orientation-preserving case.
\end{rem}

\subsubsection{Using the partial action $\alpha_S$}

We now use the partial action $\alpha_S$ introduced in \S\ref{s_alphas} to obtain a result of the same flavor as Corollary \ref{semiprod}. For an element $g$ of $\PC(X)$, we call points of $S$-indeterminacy of $\alpha_S(g)$ the elements outside its domain of definition. 

\begin{cor}\label{semiprod2}
Let $X$ be a Hausdorff topological space with no isolated point.
Consider the cofinite-partial action $\alpha_S$ of $\PC_S(X)$ on $X$.
\begin{enumerate}
\item the prodigal semi-index function $\ell_S^-(\sigma)$ of this partial action (and hence of its restriction to any subgroup) coincides with the function mapping $\sigma\in\PC(X)$ to the number of $S$-indeterminacy points of $\sigma$;
\item\label{semiproi22} for every $\sigma\in\PC_S(X)$ there exists $m_{S,\sigma}\in\N$ and a bounded function $b:\N\to\N$ such that $\ell^-(\sigma^n)=m_{S,\sigma} n+b(n)$ for all $n\in\N$. 
\item Let $G$ be a group with a homomorphism into $\PC_S(X)$. Then 
$X$ is transfixed by $G$ if and only if the number of $S$-indeterminacy points of $g$ is bounded independently of $g\in G$;
\end{enumerate}
Moreover, if $X$ is a topological manifold with no boundary and finitely many ends, then $\ell_S^-(\sigma)=\ell_S^-(\sigma^{-1})$ and $m_{S,\sigma}=m_{S,\sigma^{-1}}$ for all $\sigma\in\PC_S(X)$.
\end{cor}
\begin{proof}
The proof follows the same (two!) lines as that of Corollary \ref{semiprod} (applied to $\alpha_S$ instead of $\alpha$). The last statement rather follows   from the easy fact, checked in the proof of Corollary \ref{growthdisc}, that the complement of $n$ and $m$ points in such a topological manifold $X$ are never homeomorphic for $n\neq m$.
\end{proof}

In cases such as the pseudogroup of local isometries of $\R/\Z$, we have $\alpha=\alpha_S$ so in this case Corollary \ref{semiprod2} does not provide anything new. On the other hand, it yields something when $\alpha_S$ is finer than $\alpha$. Let us provide some illustrations:

\begin{exe}\label{exfirstco}
Fix $k\in\N$. Let $\mathcal{C}^k$ be the pseudogroup of local diffeomorphisms of class $\mathcal{C}^k$ on the circle $\R/\Z$. Then the $\mathcal{C}^k$-indeterminacies of $\sigma$ are the set of points at which either $\sigma$ or one of its derivatives $\sigma^{(i)}$ for some $i\in\{1,\dots,k\}$ has no outer limit at $x$; call this $k$-singular points. 

Another example is the pseudogroup $\mathbf{Aff}$ of local affine homeomorphisms; we have $\mathbf{Aff}\subset\mathcal{C}^1$ and these two pseudogroups have the same indeterminacies: indeed at the neighborhood of a real number $x$ (minus $\{x\}$), a piecewise affine map $\sigma$ coincides with a local affine homeomorphism if and only it coincides with a local $\mathcal{C}^1$-diffeomorphism. In the piecewise affine context, 1-singular points are often called breakpoints. This proves that the number of 1-singular points of $f^n$, for $f$ piecewise affine and $n\in\Z$, can be written as $q|n|+b(n)$ with $b$ bounded and $q\in\N$. 
In particular, this retrieves Guelman and Liousse's result \cite[Proposition 4.1]{GuL} that this number, when $n\to\infty$, grows linearly as soon as it is unbounded.

One more example is $\mathbf{P}^1_\R$ with the pseudogroup $\mathbf{Proj}$ consisting of restrictions of projective transformations (that is, homographies). 
On the circle (or any standard curve), the isometric charts in $\R$ being isometric, they endow it with a $\mathbf{Proj}$-modeled structure, which itself defines a $\mathcal{C}^2$-structure, with the same $S$-indeterminacy points for both $S$: this reflects the fact that the germ at a neighborhood $x$ (minus $\{x\}$) of a piecewise projective transformation $\sigma$ at a point $x$ coincides with a germ of projective transformation if and only if $x$ is not 2-singular at $x$.

Note that the pseudogroup $\mathbf{Proj}$ transfers as a pseudogroup on $\R/\Z$ (local homeomorphisms that are written locally as homographies), so it is less natural but harmless to stick to pseudogroups on $\R/\Z$.
\end{exe}

We now apply this to counting singularities of piecewise continuous or differentiable self-transformations. We use the notation introduced before Corollary \ref{countingsing}.

\begin{cor}\label{singro}
For every $k\in\N$ and every parcelwise-$\mathcal{C}^k$ self-transformation $\sigma$ of $\R/\Z$, there exist integers $0\le m_0\le \dots \le m_k$ and bounded non-negative even functions $b_i:\Z\to\N$ such that for all $i\in\{0,\dots,k\}$, we have $k_{\le i}(\sigma^n)=m_i|n|+b_i(n)$.

In particular, $k_i(\sigma^n)=(m_i-m_{i-1})|n|+O(1)$ and $k_{\le i}$ and $k_i$ have the property of growing either linearly or being bounded.
\end{cor}

As mentioned in the introduction, in the piecewise affine case Guelman and Liousse \cite[\S 4]{GuL} proved that $k_{\le 1}(\sigma^n)$ (and $k_0(\sigma^n))$ are either bounded of have linear growth. Their proof, more precisely, consists in 
\begin{itemize}
\item proving that either $k_0(\sigma^n)$ is either bounded, or it belongs to $[n-c,Cn]$ for some constants $c,C$ and all $n\in\N$;
\item if $k_0(\sigma^n)$ is bounded, proving that $k_1(\sigma^n)$ is either bounded, or it belongs to $[n-c',C'n]$ for some constants $c',C'$ and all $n\in\N$.
\end{itemize}

\subsection{The main theorem}

Let $S$ be a pseudogroup on a topological space $A$. 

Consider two $S$-modeled spaces $X,X'$ with topological $S$-preserving partial actions of a group $G$. A cofinite $S$-preserving $G$-biequivariant isomorphism is the data of cofinite subsets $Y\subset X$, $Y'\subset X'$ (thus endowed with the corresponding partial actions of $G$ by partial automorphisms of $S$-modeled space) and a $G$-biequivariant isomorphism $\psi:Y\to Y'$ of $S$-modeled spaces.

Thanks to all the preparatory work, we can formulate and obtain:

\begin{thm}\label{main}
Let $X$ be a Hausdorff $S$-modeled space. Let $G$ be a group with a topological $S$-preserving partial action on $X$. Suppose that $G$ transfixes $X$.
Then there exists a Hausdorff $S$-modeled space $X'$ endowed with an $S$-preserving continuous $G$-action, and a cofinite $S$-preserving $G$-biequivariant isomorphism $\psi$ from $X$ to $X'$.

Moreover, we can require that every finite $G$-orbit in $X'$ is included in $Y'$. 
\end{thm}
\begin{proof}
Let $\hat{X}$ be the universal globalization of $X$ (\S\ref{uniglo}). By Corollary \ref{glostru}, $\hat{X}$ canonically inherits an $S$-structure. Since $X$ is transfixed, by Proposition \ref{transfihau} there exists an open Hausdorff $G$-invariant subset $X'$ of $\hat{X}$ such that $X\tu X'$ is finite. Removing all finite $G$-orbits meeting $X'\smallsetminus X$ if necessary, we can ensure that $X'\smallsetminus X$ only meets infinite $G$-orbits. Then we have reached the conclusion, with $Y=Y'=X\cap X'$ and $\psi$ being the identity map from $X\cap X'$. (It is on purpose that the statement of the theorem does not refer to the universal globalization $\hat{X}$, so we do not view $X$ and $X'$ as subsets of the same space.)
\end{proof}

As a first corollary, we have the following. Recall that $S^\wp$ denotes the pseudogroup of parcelwise-$S$ local homeomorphisms \S\ref{parcelws}.

\begin{cor}\label{cormain1}
Let $X$ be a curve. Let $G$ be a group with a homomorphism $G\to\PC_S(X)$. Suppose that $G$ transfixes $X$ for the partial action $\alpha$ (restricted from $\PC(X)$).
Then there exists an $S$-modeled compact curve $X'$ endowed with an $S^\wp$-preserving continuous $G$-action, and a cofinite $S^\wp$-preserving $G$-biequivariant isomorphism from $X$ to $X'$. 
\end{cor}
\begin{proof}
We use Proposition \ref{semic} to have a cofinite-partial action, and thus apply Theorem \ref{main}. This yields the result, except compactness of $X'$. First obtain $X''$ possibly not compact; 1-point compactify each component to obtain a space $X'$, which is naturally a curve. Extend the continuous action; since $S^\wp$ is stable under concatenation, the resulting action is $S^\wp$-preserving.
\end{proof}

\subsection{Applications of Corollary \ref{cormain1}}

Let us provide applications of Corollary \ref{cormain1}. It is mostly interesting when $S=S^\wp$. First, recall from Corollary \ref{semiprod}(\ref{trxiff}) that the transfixing property is equivalent to the boundedness of the number of discontinuities.

\subsubsection{Pseudogroup of all local homeomorphisms}

Then $S=S^\wp$; here an $S$-modeled curve is just a curve. So Corollary \ref{cormain1} concerns homomorphisms into $\PC(X)$. It says that when the partial action of $G$ is transfixing, we can find a continuous action on another compact curve that coincides with the original one on a cofinite subset. 

\subsubsection{Pseudogroup of all local orientation-preserving homeomorphisms}

Again, $S=S^\wp$. This concerns homomorphisms into $\PC^+(\R/\Z)$. Here an $S$-modeled curve is just an oriented curve.
Corollary \ref{cormain1} then says that the conjugation can be chosen to be piecewise orientation-preserving.

\subsubsection{Local isometries}

Then $S=S^\wp$. This is the study of homomorphisms into $\IETbw$. Here an $S$-modeled curve is the same as a 1-dimensional Riemannian manifold with finite volume (thanks to the finiteness assumption in the definition of $S$-model); in particular every $S$-modeled curve is $S$-preserving homeomorphic to a standard curve. Hence Corollary \ref{cormain1} says that for a homomorphism $G\to\IETbw$, the transfixing condition implies a piecewise isometric conjugation to an isometric action on a compact curve.

This has various consequences: for instance if the original homomorphism $G\to\IETbw$ is transfixing and injective, then this forces $G$ to be virtually abelian. For instance, this implies that $\IETbw$ has no infinite subgroup with Property FW. In particular, it has no infinite subgroup with Kazhdan's Property T. The latter fact was established in \cite[Theorem 6.1]{DFG} for $\IET^+$ by a distinct method (rather related to amenability); it implies the result for $\IETbw$ because $\IETbw$ embeds into $\IET^+$ as the centralizer of $x\mapsto -x$ (see Lemma \ref{bousch}).

This notably applies when $\Gamma$ is cyclic and distorted in $G$. In the case of $\IET$ or $\IETbw$, this narrows possibilities for distorted cyclic subgroups, but is not enough to discard them; additional work \cite[\S 4]{Nov}, conceptualized in \S\ref{s_nondile} allows to conclude that cyclic subgroups of $\IETbw$ are undistorted. 

\begin{cor}\label{abiet}
Every finitely generated abelian subgroup of $\IETbw$ is undistorted.
\end{cor}
\begin{proof}
Let $A$ be a finitely generated abelian subgroup, with a finite generating subset $W'$. If $A$ is distorted, then there exists a finite subset $W$ in $\IETbw$ and a sequence $(n_i)$ tending to infinity, a sequence $(a_i)$ in $A$ such that $|a_i|_{W'}\simeq n_i$ and $\lim |a_i|_W/n_i=0$. Let $\ell$ be the cardinal-definite function associated to this partial action. By Proposition \ref{cardefab}, there exists a finite index subgroup of $A$ of the form $B\oplus A'$, such that $\ell$ is bounded on $B$ and equivalent to the word length on $A'$. We can suppose that $a_i=(b_i,a'_i)\in B\oplus A'$ for all $i$. Then $\ell(b,a')\ge c|a'|_{W'}-C$ for some some $c>0$, $C\in\R$ and all $(b,a')\in B\times A'$. In particular, $\ell(a_i)\ge c|a'_i|_{W'}-C$ for all $n$. It follows that $\lim |a'_i|_{W'}/n_i=0$, and in turn that $\liminf |b_i|_{W'}/n_i>0$. Given that $\lim |a'_i|_W/n_i=0$ and $\lim |a_i|_W/n_i=0$, we deduce that $\lim |b_i|_W/n_i=0$. Thus $B$ is a distorted subgroup.

Since $B$ transfixes $X$, after changing the model and using Corollary \ref{cormain1}, we can suppose that $B$ acts by isometries. Hence it is undistorted by Proposition \ref{undiiet}. This reaches a contradiction.
\end{proof}

An immediate consequence is that the only virtually polycyclic groups embedding into $\IETbw$ are the virtually abelian ones; the latter fact being proved in \cite{DFG2} with another method, handling more general virtually torsion-free solvable groups. Indeed, it is an easy exercise to show that a virtually polycyclic group is virtually abelian if and only if all its abelian subgroups are undistorted (beware that there exist non-virtually-abelian polycyclic groups of the form $\Z^4\rtimes\Z$ in which all cyclic subgroups are undistorted).

\subsubsection{Local motions}

Again, $S=S^\wp$, and we obtain the same with the bonus of a piecewise orientation-preserving conjugation.

\subsection{Refinement of the partial action}
Applying Theorem \ref{main} with $\alpha_S$ instead of $\alpha$, we often obtain stronger conclusions:

\begin{cor}\label{cormain2}
Let $X$ be a finitely-charted $S$-modeled curve. Let $G$ be a group with a homomorphism $G\to\PC_S(X)$. Suppose that $G$ transfixes $X$ for the partial action $\alpha_S$.
Then there exists a finitely-charted $S$-modeled curve $X'$ endowed with an $S$-preserving continuous $G$-action, and a cofinite $S$-preserving $G$-biequivariant isomorphism from $X$ to $X'$. 
\end{cor}
\begin{proof}
This is a direct application of Theorem \ref{main}, applied to the partial action $\alpha_S$.
\end{proof}

For continuous actions, we can get a better control on the change of model. Denote by $\PC_S^0(X)$ the set of self-homeomorphisms of $X$ that induce an element of $\PC_S(X)$. If $X$ has no isolated point, the canonical map $\PC_S^0(X)\to\PC_S(X)$ is injective.

\begin{cor}\label{cormain3}
Fix two pseudogroups $S\subset T$ on $\R/\Z$.
Let $X$ be a finitely-charted $S$-modeled curve. Let $G$ be a group with a homomorphism $\alpha:G\to\PC_S^0(X)\cap\Aut_T(X)$. Suppose that $G$ transfixes $X$ for the partial action $\alpha_S$.
Then there exists a cofinite $G$-invariant subset $Y$ of $X$,  
finitely-charted $S$-modeled curve $X'$ endowed with an $S$-preserving continuous $G$-action, and a $T$-preserving and parcelwise $S$-preserving $G$-biequivariant homeomorphism from $Y$ to $X'$. 

In particular, if $G$ has no finite orbit on $X$, then $Y=X$.
\end{cor}
\begin{proof}
Corollary \ref{cormain2} yields a cofinite subset $V$ of $X$, an $S$-modeled curve $W$ with an $S$-preserving continuous $G$-action, a cofinite subset $V'$ of $W$, and a $G$-biequivariant $S$-preserving homeomorphism $f:V\to V'$. Here $X$ is endowed with the partial action $\alpha_S$, while $W$ is endowed with its given action, and the cofinite subsets $V$ and $V'$ are endowed with the induced partial actions.

We can apply Proposition \ref{glo_tran} to $X$. It ensures that there is a cofinite $G$-invariant subset $Z_0$ of $X$ such that for every cofinite subset $Z_1$ of $Z_0$, denoting $Z^\vee_1$ the smallest $G$-invariant subset including $Z_1$, the inclusion map $(Z_1,\alpha_S)\to (Z^\vee_1,\alpha)$ is a universal globalization. Similarly, we apply Proposition \ref{glo_tran} to $W$ (for $W$, we view the global action as a partial action as well): it yields a $G$-invariant subset $Z'_0$ of $W$ with the analogous property.

We can suppose, replacing simultaneously $V$ and $V'$ with smaller subsets, that $V\subset Z_0$ and $V'\subset Z'_0$. Then the inclusions $V\to V^\vee$ and $V'\to (V')^\vee$ are universal globalizations. Then the isomorphism $f:V\to V'$ of topological partial actions extends, by the universal property, to a $G$-biequivariant homeomorphism $V^\vee\to (V')^\vee$. So we obtain the conclusion, with $Y=V^\vee$ and $X'=(V')^\vee$.

Let us observe that $f$ is $T$-preserving. Indeed, this follows from the construction and the proof of Theorem \ref{main}: we have inclusions $X\to\hat{X}\leftarrow W$. The $G$-invariant $T$-structure is inherited by $\hat{X}$ and then by $W$, still $G$-invariant, and hence by the $G$-invariant subsets $Y$ on the one hand and $X'$ on the other hand. The map $f:Y\to X'$, constructed inside $\hat{X}$, is just the identity map. Hence it is $T$-invariant. 

This argument for $T$ does not apply to $S$, but applies to the pseudogroup $S^\wp$ of local parcelwise-$S$ homeomorphisms. Hence $f$ is $T$-preserving and parcelwise $S$-preserving.
\end{proof}

We now apply it to various pseudogroups $S$ on $\R/\Z$. In each case the study has several steps: identifying domains of definition for $\alpha_S$, identifying $S$-modeled curves, and drawing consequences.

\subsection{Pseudogroup of local affine homeomorphisms}

Here $\PC_S(X)$ is the group of piecewise affine self-transformations. The domain of definition is the set of non-singular points, which is the same as points at which both the function and its derivative have an outer limit. Hence, transfixing is equivalent to have a uniformly bounded number of singularities.

Classifying affinely modeled curves is not hard, see Appendix \ref{s_projcur}. In particular the affine automorphism group of any affinely modeled finitely-charted curve is virtually abelian. As a consequence, we have:

\begin{cor}\label{affinekaz}
Let $\Gamma$ be a group with a faithful piecewise affine action on a standard curve, and a subgroup $\Lambda$. Suppose that $(\Gamma,\Lambda)$ has relative Property FW. Then $\Lambda$ is virtually torsion-free abelian. In particular, if $\Gamma$ has Property FW (e.g., Kazhdan's Property T) then it is finite.
\end{cor}

Andr\'es Navas informed me that even the failure of Property T is a new result: for instance the above applies to the subgroup $V_{\{2,3\}}$ of piecewise affine maps with slopes in the multiplicative group $\langle 2,3\rangle$ and singularities in $\Z[1/6]$, as well as its subgroup $T_{\{2,3\}}$ of elements acting continuously on the circle; the failure of Property T was unknown for both groups.

Lodha, Matte Bon and Triestino have independently obtained Corollary \ref{affinekaz} in the case of continuous piecewise affine maps (thereby also proving the failure of Property T for $T_{\{2,3\}}$).

In the case of distortion, we deduce the following, which is essentially due  to Guelman and Liousse \cite{GuL}.

\begin{cor}\label{distoafpi}
For any distorted cyclic subgroup $\langle c\rangle$ of the group of piecewise affine self-transformations on $\R/\Z$, there is a cofinite piecewise affine conjugation to an affine action on a standard curve. In the piecewise orientation-preserving case, we can choose the cofinite conjugation to be piecewise orientation-preserving.
\end{cor}
\begin{proof}
Since $\langle c\rangle$ is distorted, its cofinite-partial action is transfixing (Corollary \ref{exfw}). Therefore, by Corollary \ref{cormain2}, there a cofinite piecewise affine conjugation to an affine action on an affinely modeled curve $C$. Some finite index subgroup $\langle c'\rangle$ preserves each component.
Then Corollary \ref{affinondi} ensures that $\langle c'\rangle$, and hence $\langle c\rangle$, acts with finite order on any non-standard circle occuring in $C$. Hence, removing a finite $\langle c\rangle$-invariant subset, we can suppose that there is no non-standard circle in $C$.
\end{proof}

The minor nuance is that Guelman and Liousse work in the piecewise orientation-preserving case, and obtain the conjugation after passing to a subgroup of finite index.

Also in the continuous case, we can refine the conjugacy and deduce the following:

\begin{cor}\label{distoafhom}
For any distorted cyclic subgroup $\langle c\rangle$ of the group of piecewise affine self-homeomorphisms of the circle $\PC^0_{\mathbf{Aff}}(\R/\Z)$, there exists a piecewise self-homeomorphism of $\R/\Z$ conjugating $\langle c\rangle$ to a cyclic group of irrational rotations.
\end{cor}
\begin{proof}
By Corollary \ref{cormain3}, there is a cofinite $\langle c\rangle$-invariant subset $Y$ of $X$, another affinely modeled curve $X'$ and and a piecewise affine homeomorphism $Y\to X'$ conjugating the $\langle c\rangle$-action to an affine action on $Y$.

If $Y\neq X$, then $X'$ is homeomorphic to a finite disjoint union of intervals, each being finitely-charted; hence each is isomorphic as affinely modeled curve to $\mathopen]0,1\mathclose[$ and hence has a finite automorphism group; since $\langle c\rangle$ is distorted, it is infinite and we get a contradiction. 

So $Y=X$. Thus $X'$ is an affinely modeled curve homeomorphic to the circle, and $\langle c\rangle$ acts as an irrational rotation. By Corollary \ref{affinondi}, irrational rotations of non-standard affine circles are undistorted. Hence $X'$ is a standard circle, and hence we can suppose (conjugating with an affine isomorphism $X'\to X$) that $X'=X$.
\end{proof}

Let us provide different equivalent restatements the question asked after Corollary {distortedaff}.

\begin{cor}\label{eqque}
Denote by $\theta_r$ the rotation $x\mapsto x+r$ on $\R/\Z$.
Let $X_d$ be a disjoint union of $d$ copies of standard circles $Y_1,\dots,Y_d$. The following are equivalent. 
\begin{enumerate}[(i)]
\item\label{hypiaf} The group $\PC_{\mathbf{Aff}}(\R/\Z)$ admits a distorted cyclic subgroup; 
\item\label{hypiafp} The group $\PC^+_{\mathbf{Aff}}(\R/\Z)$ admits a distorted cyclic subgroup; 
\item\label{piaf3} There exists $r\in\R/\Z\smallsetminus\Q/\Z$ such that $\theta_r\sqcup\mathrm{id}_Y$ is a distorted element in $\PC_{\mathbf{Aff}}(X_2)$.
\item\label{piaf4} There exists $r\in\R/\Z\smallsetminus\Q/\Z$ such that $\theta_r\sqcup\theta_r\sqcup\mathrm{id}_{Y_3}$ is a distorted element in $\PC^+_{\mathbf{Aff}}(X_3)$.
\end{enumerate}
Also, the following are equivalent:
\begin{enumerate}[(i')]
\item The group $\PC_{\mathbf{Aff}}^0(\R/\Z)$ admits a distorted cyclic subgroup;
\item There exists $r\in\R/\Z\smallsetminus\Q/\Z$ such that $\theta_r$ is a distorted element in $\PC_{\mathbf{Aff}}^0(\R/\Z)$.
\end{enumerate}
\end{cor}

What comes out of \cite{GuL} is the equivalence between (\ref{hypiafp}) and a slightly weaker analogue of (\ref{piaf4}): the existence of $n$ and a non-identity self-homeomorphism $f$ of $X_n$ acting on each $Y_i$ as either the identity or an irrational rotation, such that $f$ is distorted in $\PC^+_{\mathbf{Aff}}(X_n)$.

In the continuous case we did not make the orientation-preserving case explicit, because this is a trivial reduction, since the orientation-preserving subgroup has index 2. This is in contrast to the first case, where the piecewise orientation-preserving subgroup has infinite index.

\begin{proof}[Proof of Corollary \ref{eqque}]
For the second equivalence, one implication is trivial, and the other follows from Corollary \ref{distoafhom}.

Let us prove the first equivalence. (\ref{hypiaf})$\Leftarrow$(\ref{piaf3})$\Leftarrow$(\ref{piaf4})$\Rightarrow$(\ref{hypiafp})$\Rightarrow$(\ref{hypiaf}) is trivial.
That (\ref{hypiaf}) implies (\ref{hypiafp}) follows from the group embedding \[\PC_{\mathbf{Aff}}(\R/\Z)\to\PC^+_{\mathbf{Aff}}((\R/\Z)^\pm)\simeq\PC^+_{\mathbf{Aff}}(\R/\Z)\] from Lemma \ref{bousch}, and where the second, non-canonical isomorphism is induced by a piecewise affine transformation between $(\R/\Z)^\pm$ and $\R/\Z$. To obtain (\ref{piaf3})$\Rightarrow$(\ref{piaf4}), we again use this embedding but more carefully: for $X=X_2$, it maps $\theta_r\sqcup\mathrm{id}_{Y_2}$ to (after suitable identifications) the element $\theta_r\sqcup\mathrm{id}_{Y_2}\sqcup\theta_r\sqcup\mathrm{id}_{Y_4}$ of $\PC^+(X_4)$. After conjugating $Y_2\sqcup Y_4$ to $Y_3$ by a piecewise affine transformation and $Y_3$ to $Y_2$ by an affine transformation, we get the result.

It remains to prove (\ref{hypiaf})$\Rightarrow$(\ref{piaf3}). First, one uses Corollary \ref{distoafpi} to show that there exists a standard curve $X$ with 
with $c\in\PC_{\mathbf{Aff}}^0(X)$ such that the cyclic subgroup $\langle c\rangle$ is distorted $\PC_{\mathbf{Aff}}(X)$. Replacing $c$ with a power, we can suppose that $c$ preserves each component $Y$ of $X$, and that as an automorphism of $Y$, $c$ is either the identity or has infinite order. Hence $c$ is the identity on each noncompact component of $X$, and is an irrational rotation on every compact component of $X$. Then we fix one component $Y$ of $X$ on which $c$ acts as an irrational rotation, and we let $w$ be an isometry of $X$, acting as a reflection on $Y$ and as the identity outside. Let $q$ be equal to $c$ on $Y$ and the identity elsewhere. Then $c^nwc^{-n}w^{-1}=q^{2n}$. Hence $c$ being distorted, so is $q$. Then conjugating the complement of $Y$ by an affine transformation to a single standard circle, we obtain (\ref{piaf3}).
\end{proof}

In the case of the pseudogroup $\mathcal{C}^k$ of local diffeomorphisms of class $\mathcal{C}^k$, the classification of $\mathcal{C}^k$-modeled curves is ``trivial" in the sense that there are only two such connected curves up to isomorphism: the open interval and the circle.
 In particular, Corollary \ref{cormain2} yields:

\begin{cor}\label{Ckconju}
Fix $k\in\N$. Let $G$ act by piecewise-$\mathcal{C}^k$ (resp.\ parcelwise-$\mathcal{C}^k$) self-transformations on $\R/\Z$. Suppose that $G$ transfixes $\R/\Z$ for the corresponding partial action. Then there exists a cofinite subset $Y\subset\R/\Z$, a curve $X'$ with a structure of a $\mathcal{C}^k$-manifold and a $G$-action by $\mathcal{C}^k$-diffeomorphisms, a cofinite subset $Y'\subset X'$ and a $G$-biequivariant piecewise-$\mathcal{C}^k$ (resp.\ parcelwise-$\mathcal{C}^k$) homeomorphism $h:Y\to Y'$.
\end{cor}

In the continuous case, this yields the following, with a little refinement to pass from parcelwise to piecewise conjugacy:

\begin{cor}\label{conjupiecewiseck}
Let $G$ be group acting continuously on $\R/\Z$, with no finite orbit. Suppose the action is parcelwise-$\mathcal{C}^k$ (i.e., $G$ maps into $\PC_{\mathcal{C}^k}^0(\R/\Z)$. If $G$ transfixes $\R/\Z$ for the corresponding partial action $\alpha_{\mathcal{C}^k}$, then the action is conjugate (in $\PC_{\mathcal{C}^k}^0(\R/\Z)$) to a $\mathcal{C}^k$-action.

If moreover $G$ acts by piecewise-$\mathcal{C}^\ell$ self-transformations for some $\ell\le k$, then the conjugating map can be assumed to have the same property.
\end{cor}
\begin{proof}
Both are direct applications of Corollary \ref{cormain3} with $S$ being the $\mathcal{C}^k$ pseudogroup. In the first case, we consider no $T$ (technically, this means we take $T$ as the pseudogroup of all local homeomorphisms). In the second case, we take $T$ as the pseudogroup of all self-homomorphisms that piecewise-$\mathcal{C}^\ell$.
\end{proof}

Lodha, Matte Bon and Triestino \cite{LMT} obtain a result which is very close to Corollary \ref{conjupiecewiseck}. Actually the result of Corollary \ref{conjupiecewiseck} (at least both the first statement and the second for $\ell=k$) is a consequence of their work, but they mainly formulate results for groups with Property FW or T, which, using Thurston's stability theorem on the one hand and Navas' theorem about groups with Property T of class $\mathcal{C}^{3/2}$ on the other hand, yield stronger conclusions \cite[Corollaries 1.3, 1.4]{LMT}.

As regards the piecewise projective case, let us write the corollaries:

\begin{cor}\label{corpro1}
Let $G$ be a group with a piecewise projective action on a finitely-charted projectively modeled curve $X$. Suppose that the corresponding partial action transfixes $X$. Then there is another finitely-charted projectively modeled curve $X'$, cofinite subsets $Y\subset X$ and $Y'\subset Y$, and a piecewise projective $G$-biequivariant homeomorphism $h:Y\to Y'$. 

Moreover, if $G$ acts continuously on $X$ with no finite orbit, we can choose $Y=X$, $Y'=X'$. If $G$ acts by $\mathcal{C}^1$ diffeomorphisms, we can choose $h$ to be of class $\mathcal{C}^1$. If $G$ both acts by $\mathcal{C}^1$ diffeomorphisms without finite orbit, we can choose $Y=X$, $Y'=X$, and $h$ a $\mathcal{C}^1$-diffeomorphism.
\end{cor}

Given the classification of projectively modeled curves and their automorphisms, we obtain as corollaries, also proved in the companion note \cite{CFW_short}:

\begin{cor}\label{psl2pro}
Under the same hypotheses, $G$ has a subgroup of finite index $H$ with a finite normal subgroup $Z$ such that $H/Z$ can be embedded as a subgroup of $\PSL_2(\R)^k$ for some $k$. If moreover $X$ is connected and $G$ acts continuously, we can choose $k=1$.
\end{cor}
\begin{proof}
By Corollary \ref{corpro1}, we can suppose that $G\subset\Aut_{\mathbf{Proj}}(X)$. Let $(X_i)$ be the connected components of $X$; by Appendix \ref{s_projcur}, the automorphism group of each $X_i$ has finitely many components.
There exists a subgroup $H$ of finite index in $G$ stabilizing each component, and mapping into $\Aut(X_i)^\circ$. By Appendix \ref{s_projcur}, $\Aut(X_i)$ is isomorphic to either $\R$, $\R/\Z$, $\R\rtimes\R$, or $\PSL_2(\R)_m$ (the $m$-fold connected covering of $\PSL_2(\R)$). The first three, as well as the quotient of the latter by its center, embed as subgroups into $\PSL_2(\R)$. We thus obtain the conclusion.
\end{proof}

\begin{cor}\label{notpro}
For every finitely-charted curve $X$, the group $\PC_{\mathbf{Proj}}(X)$ has no infinite subgroup with Kazhdan's Property T.
\end{cor}
\begin{proof}
If we suppose so, using Corollary \ref{psl2pro} and the fact that Kazhdan's Property T passes to finite index subgroups (and obviously to quotients), we would deduce the existence of 
an infinite subgroup with Property T in $\PSL_2(\R)$. But it is well-known that there is no such group. Indeed, by Faraut-Harzallah \cite{FH}, it would be conjugate into the maximal compact subgroup $\mathrm{PSO}_2(\R)$, which is abelian, a contradiction.
\end{proof}

This is new even in the continuous case. See also Example \ref{exfwpp}.

As suggested by A.~Valette, we have the following more precise consequence:

\begin{cor}
For every finitely-charted curve $X$, every subgroup of $\PC_{\mathbf{Proj}}(X)$ with Property FW, has the Haagerup Property.
\end{cor}
\begin{proof}
Indeed, the Haagerup Property is inherited from finite index subgroups and extensions with finite kernels, and is true for subgroups of $\mathrm{PSL}_2(\R)$~\cite{GHW}.
\end{proof}

\section{More on piecewise projective self-homeomorphisms}

\subsection{Construction of actions}\label{conoac}

The results of this paper were mainly obtained by using the formalism of partial action and the notion of universal globalization. We now provide more explicit actions for the group of piecewise projective self-homeomorphisms of the circle. This was actually our initial approach before realizing how that formalism could avoid a more computational approach. However, we believe it is interesting to write down these formulas, since they can convey more intuition about the proofs, and since they indicate what is hidden behind taking universal globalizations.

For a standard curve $X$, denote $X^\pm=X\times\{+,-\}$ (see \S\ref{doubpoints}); for $s\in X$, we usually write $s^+$ and $s^-$ rather $(s,+)$ and $(s,-)$. For $x=(s,\eps)\in X^\pm$, we write $\hat{x}=(s,-\eps)$ (for obvious sign conventions). Define $\mathcal{L}^2(X)=X^\pm\times\R_{>0}\times\R$.

Let $X,Y$ be standard curves. The first and second one-sided derivatives of piecewise-$\mathcal{C}^2$ functions $f$ on $X$ can be interpreted as functions $f',f'':X^\pm\to\R$. Define $\CPD_2(X,Y)$ as the set of continuous, piecewise-$\mathcal{C}^2$ functions, whose one-sided first derivatives do not vanish. For $f\in\CPD_2(X,Y)$, define $f_{(2)}:\mathcal{L}^2(X)\to \mathcal{L}^2(Y)$ by 
\[f_{(2)}(x,t,u)=\left(f(x),\;\frac{f'(x)}{f'(\hat{x})}t,\;
\frac{1}{f'(x)}u+\frac{f''(x)}{2f'(x)^2}-\frac{f''(\hat{x})}{2f'(x)f'(\hat{x})}t^{-1}\right).\]
A simple computation shows if $Z$ is another standard curve and $g:Y\to Z$ is continuous and piecewise of class $\mathcal{C}^2$, then
that $(g\circ f)_{(2)}$ and $g_{(2)}\circ f_{(2)}$ are both equal: indeed, they are equal to
\[(x,t,u)\mapsto\left(g(f(x)),\frac{f'(x)g'(f(x))}{f'(\hat{x})g'(f(\hat{x}))}t,\;a(x)u+b(x)-c(x)t^{-1}\right),\]
where 
\[a(x)=\frac{1}{f'(x)g'(f(x))},\;b(x)=\frac{f''(x)}{2g'(f(x))f'(x)^2}+\frac{g''(f(x))}{2g'(f(x))^2};\]
\[c(x)=\frac{1}{2g'(f(x))f'(x)}\left(\frac{f''(\hat{x})}{f'(\hat{x})}+\frac{g''(f(\hat{x}))f'(\hat{x})}{g'(f(\hat{x}))}\right).\]

Let $X,Y$ be standard curves. In the following proposition, we use this action to define a natural pull-back for functions $X^\pm\to\R_{>0}\times\R$, which is used throughout the sequel.

\begin{prop}\label{pnumu}
Let $\mu=(\mu_1,\mu_2)$ be a function $Y^\pm\to\R_{>0}\times\R$ (with $\mu_1:Y^\pm\to\R_{>0}$ and $\mu_2:Y^\pm\to\R$), and $P_\mu\subset \mathcal{L}^2(X)$ its graph. Then $f_{(2)}^{-1}(P_\mu)=P_{f^*\mu}$, where for all $x\in X$ we have 
\[(f^*\mu)_1(x)=\frac{f'(\hat{x})}{f'(x)}\mu_1(f(x))\] and
\[(f^*\mu)_2(x)=f'(x)\mu_2(f(x))-\frac{f''(x)}{2f'(x)}+\frac{f''(\hat{x})f'(x)}{2f'(\hat{x})^2}\mu_1(f(x))^{-1}.\]
\end{prop}
\begin{proof}
We have $(x,t,u)\in f_{(2)}^{-1}(P_\mu)$ if and only if $(x',t',u'):=f_{(2)}(x,t,u)\in P_\mu$. This means that $\mu_1(x')=t'$ and $\mu_2(x')=u'$. 
This means that 
$\mu_1(f(x))=\frac{f'(x)}{f'(\hat{x})}t$ and $\mu_2(f(x))=\frac{1}{f'(x)}u+\frac{f''(x)}{2f'(x)^2}-\frac{f''(\hat{x})}{2f'(x)f'(\hat{x})}t^{-1}$.
In turn, this means that 
$t=\frac{f'(\hat{x})}{f'(x)}\mu_1(f(x))$ and 
$u=f'(x)\mu_2(f(x))-\frac{f''(x)}{2f'(x)}+\frac{f''(\hat{x})}{2f'(\hat{x})}t^{-1}$. Thus $f^{-1}(P_\mu)$ is indeed the graph of the given function.
\end{proof}

From the covariant functoriality of $f\mapsto f_{(2)}$, we immediately deduce the contravariant functoriality of $f\mapsto f^*$, acting on all $(\R_{>0}\times\R$)-valued functions.

In addition, define an involution $\tau_X:\mathcal{L}^2(X)\to \mathcal{L}^2(X)$ by $\tau_X(x,t,u)=(\hat{x},t^{-1},-tu)$. Denote $\tau(t,u)=(t^{-1},-tu)$. Then a computation shows that $\tau_Y\circ f_{(2)}=f_{(2)}\circ\tau_X$. 

Let $\mathcal{A}^2(X)$ be the set of functions $\nu:X^\pm\to\R_{>0}\times\R$, that take the value $(1,0)$ outside a finite subset, and such that $\nu(\hat{x})=\tau(\nu(x))$ for all $x\in X^\pm$. 
Equivalently, this symmetry condition means that the graph of $\nu$ is $\tau_X$-invariant. 
Hence any proper function $f\in\CPD_2(X,Y)$ induces $f^*:\mathcal{A}^2(Y)\to \mathcal{A}^2(X)$.


On a standard curve, denote by $\nu_0^X\in \mathcal{A}^2(X)$ the ``trivial" constant function $(1,0)$.

\begin{lem}\label{nu0inv}
Let $V,V'$ be standard curves and $f$ a piecewise projective homeomorphism $V\to V'$. Then $f$ is projective if and only $f^*\nu^{V'}_0=\nu^V_0$.
\end{lem}
\begin{proof}
By definition, for any $f\in\CPD^2(V,V')$, we have
\begin{align*}\big(f^*\nu_0^{V'}\big)(x)&=\left(\frac{f'(\hat{x})}{f'(x)},-\frac{f''(x)}{2f'(x)}+
\frac{f''(\hat{x})f'(x)}{2f'(\hat{x})^2}\right)\\
=& \left(\frac{f'(\hat{x})}{f'(x)},\frac{
f''(\hat{x})-f''(x)+\frac{f''(\hat{x})}{f'(\hat{x})^2}(f'(x)^2-f'(\hat{x})^2)}{2f'(x)}\right).
\end{align*}
It follows immediately that $f$ is of class $\mathcal{C}^2$ at $x\in V$ if and only if $f^*\nu_0^{V'}(x)=(1,0)$, which equals $\nu_0^V(x)$. Therefore, $f$ is of class $\mathcal{C}^2$ on $V$ if and only if $f^*\nu_0^{V'}=\nu_0^V$.

In particular, since we assume that $f$ is a piecewise projective homeomorphism, $f$ is projective on $V$ if and only if $f^*\nu_0^{V'}=\nu_0^V$.
\end{proof}

\begin{lem}\label{uniqpro}
Let $X$ be a standard curve, $x\in X$ and $\nu\in \mathcal{A}^2(X)$. Suppose that, at the neighborhood of $x$, the function $\nu$ is invariant by homographies that are close enough to the identity. Then $\nu=\nu_0^X$ around $x$.
\end{lem}
\begin{proof}
More explicitly, since this is local at $x$, we can suppose that $X\subset\R$. For all intervals $I,J$ around $x$ such that the closure in $\R$ of $I$ is contained in $J$, we can consider the subset of homographies $f\in\mathrm{PSL}_2(\R)$ mapping $I$ into $J$. This is a neighborhood $W_{I,J}$ of the identity in $\mathrm{PSL}_2(\R)$. Precisely, the assumption is that we assume that for some such $I,J$, the identity element has a subneighborhood $W$ contained in $W_{I,J}$ such that $f^*\nu=\nu$ on $I$ for all $f\in W$.

We can suppose that $x=0$. Then the translation $T_a:x\mapsto x+a$ belongs to $W$ for $a$ small enough, say $|a|\le a_0$. The local condition $T_a^*\nu=\nu$, read at the first coordinate, implies that $\nu_1(t^\pm)=\nu_1(t^\pm+a)$ for every $t$ small enough (say, $|t|\le t_0$: note that this does not depend on $a$). Then, since $\nu_1-1$ is finitely supported, we deduce that $\nu_1-1$ vanishes at the neighborhood of $0^\pm$.

Let us now rewrite the definition of $f^*\mu$ when $f$ is of class $\mathcal{C}^2$, for each $t$ such that $\mu_1(f(t))=1$. Namely it simplifies to
$(f^*\mu)(t)=\big(1,f'(t)\mu_2(f(t))\big)$. We now choose a homothety $L_c(t)=ct$ with $c>1$ close enough to $c$ to ensure $L_c^*\nu=\nu$ at the neighborhood of 0, and in particular at zero. 
Thus we have $\nu_2(0^{\pm})=c\nu_2(0^\pm)$, which implies $\nu_2(0^\pm)=0$. 

Hence $(\nu_1,\nu_2)$ takes the value $(1,0)$ at $0$.
Since $\nu$ and $\nu_0^X$ coincide outside a finite subset, they thus coincide at the neighborhood of $0$.
\end{proof}

\begin{defn}\label{d_compatible}
Denote by $\pi_0^X$ (or $\pi_0$ when the context is clear) the standard projective structure on a standard curve $X$.
On a standard curve with standard projective structure $\pi_0$, we say that a projective structure $\pi$ is compatible if the identity $(X,\pi_0)\to (X,\pi)$ is piecewise projective. We say that it is $\mathcal{C}^1$-compatible if in addition this identity map is of class $\mathcal{C}^1$.
\end{defn}

We reach the goal of this preparatory work: encoding a compatible projective structure in an element of $\mathcal{A}^2(X)$:

\begin{prop}\label{projnu}
Let $X$ be a standard curve and $\mu\in \mathcal{A}^2(X)$. Define a system of charts on $X$ by considering piecewise projective homeomorphisms $h:U\to V$, with $U$ open in $X$ and $V$ a standard curve, such that $h^*\nu_0^V=\mu|_U$.
 Then this defines a projective structure $\pi=\pi_\mu$ on $X$, such that the identity map is a piecewise projective homeomorphism $(X,\pi_0^X)\to (X,\pi)$. 

Given $f\in\PC^0_{\mathbf{Proj}}(X,Y)$ with $Y$ another standard curve and $\nu\in \mathcal{A}^2(Y)$, the map $f$ is a projective isomorphism $(X,\pi_\mu)\to (Y,\pi_\nu)$ if and only if $f^*\nu=\mu$.

Conversely, every projective structure $\pi$ on $X$, such that the identity map $(X,\pi_0^X)\stackrel{i}\to (X,\pi)$ is a piecewise projective homeomorphism, has the form $\pi_\mu$ for a unique $\mu\in \mathcal{A}^2(X)$.
\end{prop}
\begin{proof}
To check that this defines a projective structure, the compatibility is clear from functoriality of $f\mapsto f^*$. The only thing to check is that $X$ is covered by charts. This will of course make use of the ``$\tau$-condition" saying that $\nu(\hat{x})=\tau(\nu(x))$. Let $x\in X$; we have to show that $x$ belongs to a chart. We can suppose (by an orientation-preserving isometric change of standard coordinates) that $x=0$.

If $\mu(0^+)=(1,0)$ then $\mu(0^-)=(1,0)$ by the $\tau$-condition, and hence the identity map at a small neighborhood does the job. Otherwise, we can choose a small interval around $0$ on which $\mu$ equals $(1,0)$ except at $0$. Then first composing locally with a piecewise affine local homeomorphism fixing $0$ with a singular point at $0$, boils down to the case when $\mu(0^+)=(1,u)$ for some $u\in\R$. By the $\tau$-condition, $\mu(0^-)=(1,-u)$. Then consider the formula in the proof of Lemma \ref{nu0inv}, assuming in addition that $f$ is of class $\mathcal{C}^1$: for every $t$
\begin{align*}\big(f^*\nu_0^{V'}\big)(t)
= \left(1,\frac{
f''(\hat{t})-f''(t)}{2f'(t)}\right).
\end{align*}
We can indeed find $f$ of class $\mathcal{C}^1$ and piecewise projective and $\mathcal{C}^1$ around 0 with $f'(0)=1$ and $f''(0^+)-f''(0^-)$ arbitrary. Namely, choose $f_c(t)=t$ for $t\le 0$, and $f_c(t)=t/(1-ct)$ for some $c\in\R$. Then $f'(0)=1$, $f''(0^-)=0$ and $f''(0^+)=c/2$. Hence locally composing with such a map with $c$ well-chosen, we can ensure that $\mu(0^+)=(1,0)$; the $\tau$-condition ensures that $\mu(0^-)=(1,0)$, and we have found our projective chart.

The second assertion is straightforward from the definition and functoriality of $f\mapsto f^*$.

As regards the third statement, let us start with uniqueness: consider $\mu,\mu'$ defining the same projective structure. Since this is a local assertion, we can use the first statement to assume that $\mu'=\nu^X_0$. Then the result follows from Lemma \ref{uniqpro}. The existence is immediate: just define $\mu=(i^{-1})^*\nu_0^X$.
\end{proof}

\begin{prop}\label{atmostpro}
Let $X$ be a standard curve.
Let $G$ be a subgroup of $\PC^{0}_{\mathbf{Proj}}(X)$ with no finite orbit on $X$. Then there is at most one compatible $G$-invariant projective structure $\pi$ on $X$. If moreover $G$ acts by $\mathcal{C}^1$-diffeomorphisms, then $\pi$ has to be $\mathcal{C}^1$-compatible.
\end{prop}
\begin{proof}
By Proposition \ref{projnu}, this is equivalent to showing that there is at most one $G$-invariant element $\nu\in \mathcal{A}^2(X)$.

For the first coordinate, the $G$-invariance of $\nu$ says that 
$\nu_1(x)=\frac{f'(\hat{x})}{f'(x)}\nu_1(f(x))$. If $\mu$ is also $G$-invariant, it satisfies the same formula, and we deduce, taking the quotient $\eta_1=\nu_1/\mu_1$, that $\eta(x)=\eta(f(x))$ for all $x\in X^\pm$ and all $f\in G$. Hence $\{x\in X^\pm:\eta_1(x)\neq 1\}$ is $G$-invariant. Since $G$ has no finite orbit and $\eta_1=1$ outside a finite subset, we deduce $\eta=1$ and hence $\mu_1=\nu_1$.

Next, the $G$-invariance of $\nu$, read at the second coordinate, says that, for all $x\in X^\pm$ and $f\in G$,
\[\nu_2(x)=f'(x)\nu_2(f(x))-\frac{f''(x)}{2f'(x)}+\frac{f''(\hat{x})f'(x)}{2f'(\hat{x})^2}\nu_1(f(x))^{-1}.\]
Given that $\mu$ satisfies the same invariance, and that $\mu_1=\nu_1$, we deduce, substracting and setting $\eta_2=\nu_2-\mu_2$, that for all $x\in X^\pm$ and $f\in G$,
\[\eta_2(x)=f'(x)\eta_2(f(x)).\]
In particular, the subset $\{x\in X^\pm:\eta_2(x)\neq 0\}$ is $G$-invariant; since it is finite and $G$ has no finite orbit, we deduce $\eta_2=0$ and hence $\nu=\mu$.

For the last statement, observe that if $G$ consists of $\mathcal{C}^1$-diffeomorphisms, then the $G$-invariance of $\nu$ implies that the finite subset $\{x\in X^\pm:\nu_1(x)\neq 1\}$ is $G$-invariant. Hence it is empty. This means that $\pi$ is $\mathcal{C}^1$-compatible.
\end{proof}

\begin{cor}\label{projpipi}
Let $X$ be a standard curve.
Consider two conjugate subgroups $G_1,G_2$ in $\PC^{0}_{\mathbf{Proj}}(X)$,
that have no finite orbit on $X$, and preserve compatible projective structures $\pi_1,\pi_2$. Then $(X,\pi_1)$ and $(X,\pi_2)$ are isomorphic as projectively modeled curves.

If moreover $G_1,G_2$ consist of $\mathcal{C}^1$-diffeomorphisms, then they are conjugate by some $\mathcal{C}^1$-diffeomorphism.
\end{cor}
\begin{proof}
When $G_1=G_2$, we deduce from Proposition \ref{atmostpro} that $\pi_1$ and $\pi_2$ are equal. When $G_2=fG_1f^{-1}$ with $f\in\PC^{0}_{\mathbf{Proj}}(X)$, this implies that $\pi_1=f^*\pi_2$ is the pull-out by $f$ of $\pi_2$, and hence $f$ induces an isomorphism from $(X,\pi_1)$ to $(X,\pi_2)$. 

In the $\mathcal{C}^1$-case, Proposition \ref{atmostpro} implies that $\pi_1$ and $\pi_2$ are $\mathcal{C}^1$-compatible. Then the condition $\pi_1=f^*\pi_2$ forces $f$ to be of class $\mathcal{C}^1$.
\end{proof}

\subsection{Classification of exotic circles}
It is easy and standard that every closed subgroup of the group of self-homeomorphisms of the circle $X=\R/\Z$, if topologically isomorphic to $\R/\Z$, is conjugate by an orientation-preserving self-homeo\-morphism to the group $\R/\Z$ of translations. Roughly speaking, given a subgroup $W$ of $\mathrm{Homeo}(\R/\Z)$ (typically, the automorphism group of some enriching structure), an ``exotic circle" is a subgroup conjugate to the group $\R/\Z$ of translations in $\mathrm{Homeo}(\R/\Z)$, but not in $W$. In the context of piecewise affine self-homeomorphisms, exotic circles were defined and classified by Minakawa \cite{Mi}. 

Denote by $\PC^{1}_{\mathbf{Proj}}(X)$ the group of $\mathcal{C}^1$-diffeomorphisms in $\PC^{0}_{\mathbf{Proj}}(X)$. Also use the notation of the appendix. Roughly speaking, below, $\Theta_1$ is the standard circle, $\Theta_t$ for $t>1$ are the non-standard affine circles, $\Omega_n$ is the connected $n$-fold covering of the projective line, and are interpolated by the ``metaelliptic circles" $\Omega_r$ (which come from lifts of elliptic elements in the universal covering of $\SL_2(\R)$).

\begin{thm}\label{tsg}
Let $X$ be the standard circle $\R/\Z$ with its subgroup $K$ of isometries and its subgroup $K^+\simeq\mathrm{SO}(2)$ of orientation-preserving isometries. Fix $H\in\{K,K^+\}$. Consider a faithful continuous action of $H$ on $X$, whose image $G$ is contained in $\PC^{0}_{\mathbf{Proj}}(X)$
Then the following hold.
\begin{enumerate}
\item\label{tsg1} $G$ preserves a unique compatible (Definition \ref{d_compatible}) projective structure $\pi$ on~$X$.
\item\label{tsg2} $\pi$ is $\mathcal{C}^1$-compatible if $G$ consists of $\mathcal{C}^1$-diffeomorphisms.
\item\label{tsg3} The projectively modeled curve $(X,\pi)$ is isomorphic to $\Omega_r$ for some $r>0$, or $\Theta_t$ for some $t\ge 1$. 
\item\label{tsg4} The conjugacy class of $G$ among subgroups of $\PC^{0}_{\mathbf{Proj}}(X)$ is characterized by the isomorphy type of the projectively modeled curve $(X,\pi)$, thus by either the case $\Theta_t$ for $t\ge 1$ (``affine case") or $\Omega_r$ for $r>0$ (``metaelliptic case"). We thus say that $G$ is of type $\Theta_t$, or of type $\Omega_r$. This also characterizes $G$ modulo conjugation by $\PC^{0,+}_{\mathbf{Proj}}(X)$.
\item\label{tsg5} if $G\subset\PC^{1}_{\mathbf{Proj}}(X)$, then this also characterizes the conjugacy class modulo the conjugation action of $\PC^{1}_{\mathbf{Proj}}(X)$, or also $\PC^{1,+}_{\mathbf{Proj}}(X)$.
\item\label{tsg6} $G$ preserves an affine structure on $X$ if and only if it is of type $\Theta_1$, or if $H=K^+$ and $G$ is of type $\Theta_t$ for some $t>1$.
\item\label{tsg7} Consider the corresponding conjugacy classification for the corresponding action of $H$, the classification is the same when considered modulo conjugation by $\PC^{0}_{\mathbf{Proj}}(X)$, or $\PC^{1}_{\mathbf{Proj}}(X)$ in the $\mathcal{C}^1$-case.
\end{enumerate}
\end{thm}
\begin{proof}
Write $G^+$ as the image of $K^+$ in $G$.
First observe that $G$ has no finite orbit. Otherwise $G^+$ would fix a point and it is easy to check and well-known that $K^+$ has no nontrivial continuous action on an interval.

By Lemma \ref{fwrelie}, there is a dense subgroup $\Gamma$ of $G$ such that $(G,\Gamma)$ has relative Property FW (where $G$ is considered as discrete group). By Corollary \ref{corpro1}, there exists a finitely charted projectively modeled curve $X'$, and a piecewise projective homeomorphism $f:X\to X'$ such that the conjugate $\Gamma$-action preserves the projective structure. Pulling back to $X$, we obtain a compatible $\Gamma$-invariant projective structure $\pi$. By Proposition \ref{autisclosed}, the automorphism group of $(X,\pi)$ is closed in $\mathrm{Homeo}(X)$, and hence $G$ preserves $\pi$.

When $G_i$ consists of $\mathcal{C}^1$-diffeomorphisms, Proposition \ref{atmostpro} ensures that $\pi$ is $\mathcal{C}^1$-compatible.

So (\ref{tsg1}) and (\ref{tsg2}) are proved. (\ref{tsg3}) follows from the classification of projective structures on the circle and their maximal compact subgroups of automorphisms, established in the appendix. 

For (\ref{tsg4}), one direction is provided by Corollary \ref{projpipi}: the conjugacy class of $G$ determines the isomorphism type of $(X,\pi)$. In the other direction, suppose that the isomorphism type of $(X,\pi)$ is given. Proposition \ref{conjuo2}, based on classification, provides the converse: once $G$ preserves $\pi$, it is uniquely determined modulo conjugation by the orientation-preserving automorphism group of $(X,\pi)$.

For (\ref{tsg5}), the only nontrivial improvement with respect to (\ref{tsg4}) lies in the $\mathcal{C}^1$ assertion of Corollary \ref{projpipi}.

(\ref{tsg6}) if $G$ preserves an affine structure $\xi$, it preserves the corresponding projective structure $\pi_\xi$, which is then equal to $\pi$. By the affine classification (see the appendix), $(X,\xi)$ is isomorphic as affinely modeled curve to $\Theta_t$ for some $t\ge 1$. In the case of $H=K$, the affine action does not preserve the orientation, which excludes $\Theta_t$ if $t>1$.

(\ref{tsg7}) As a topological group, the outer automorphism group of $K$ is trivial, and hence the result is immediate. For $K^+$, the classification implies that every action of $K^+$ on a projectively modeled curve extends to $K$, and then we obtain the equivalences of both conjugacy notions. (Of course this does not extend to conjugacy modulo orientation-preserving elements. This comes from the plain topological setting: the only action of $K^+$ on the circle is not orientation-reversing isomorphic to itself; that is, its centralizer in the self-homeomorphism group of the circle consists of orientation-preserving self-homeomorphisms.)
\end{proof}

\begin{thm}\label{misl2}
Fix $n\in\N_{\ge 1}$. Let $X$ be the connected $n$-covering of the projective line $\mathbf{P}^1_\R$ (with respect to the basepoint $\infty$) and let $H$ be either the connected $n$-fold covering of $\PSL_2^{(n)}(\R)$, or the corresponding overgroup $\PGL_2^{(n)}(\R)$ of index two. Let $G$ be the image of $H$ in a faithful continuous action $\alpha$ of $H$ on $X$, valued in $\PC^{0}_{\mathbf{Proj}}(X)$. 
Then: 
\begin{enumerate}
\item $G$ is conjugate to $H$ by some element of $\PC^{0,+}_{\mathbf{Proj}}(X)$, which can be chosen in $\PC^{1,+}_{\mathbf{Proj}}(X)$ if $G\subset\PC^{1}_{\mathbf{Proj}}(X)$;
\item the action $\alpha$ itself of $H$ is conjugate to the inclusion action by some element of $\PC^{0}_{\mathbf{Proj}}(X)$, which can be chosen in $\PC^{1,+}_{\mathbf{Proj}}(X)$ if $G=\alpha(H)\subset\PC^{1}_{\mathbf{Proj}}(X)$.
\end{enumerate}
\end{thm}
\begin{proof}
As in the proof of Theorem \ref{tsg}, we first need to have some Property FW phenomenon: actually we find a dense subgroup with Property FW:  If $H=\PSL_2$ or $\PGL_2$, we consider $\PSL_2(\Z[\sqrt{2}])$ or $\PGL_2(\Z[\sqrt{2}])$, and for their finite coverings we consider their inverse images.

Clearly the $G$-action has no finite orbit, and as in the proof of Theorem \ref{tsg}, using Property FW we find a compatible $\Gamma$-invariant projective structure $\pi$ ($\mathcal{C}^1$-compatible if $G$ acts by $\mathcal{C}^1$-diffeomorphisms), and again using Proposition \ref{autisclosed}, the automorphism group of $(X,\pi)$ is closed, and hence $G$ preserves $\pi$.

By the classification in the appendix, the only possibility\footnote{Hint (added after publication): since the automorphism group is not solvable, this implies that $(X,\pi)$ is isomorphic to $\Omega_m$ for some $m$. But then the existence of a continuous injective homomorphism $\PSL_2^{(n)}(\R)\to \PGL_2^{(m)}(\R)$ forces $n=m$.} is that $(X,\pi)$ is isomorphic to $\Omega_n$, the $n$-fold connected covering of the projective line, that is, $X$ itself (chosen by anticipation!). This precisely means that $G$ is conjugate, by some element of $\PC^{0,+}_{\mathbf{Proj}}(X)$, to a subgroup of $H$, and hence to $H$ itself (since clearly any injective continuous homomorphism $G\to H$ is an isomorphism). Also in the $\mathcal{C}^1$-case, since $\pi$ is $\mathcal{C}^1$-compatible, the conjugating element has to be $\mathcal{C}^1$ as well.

By post-conjugation by an orientation-reversing element of $\PGL_2^{(n)}(\R)$ if necessary, we can arrange also the conjugation to be orientation-preserving. 
\end{proof}

The following answers a question in an early version of \cite{LMT}.

\begin{cor}\label{lmtminakawa}
Fix $n\in\N_{>0}$. Let $H$ be either equal to $\PSL_2^{(n)}(\R)$ or $\PGL_2^{(n)}(\R)$. Let $G$ be the image of a faithful continuous action of $H$ on $\R/\Z$. Suppose that $G\subset\PC^0_{\mathbf{Proj}}(\R/\Z)$. 

Then $G$ is closed in $\mathrm{Homeo}(\R/\Z)$, and is uniquely defined up to conjugation in $\PC^0_{\mathbf{Proj}}(\R/\Z)$. If moreover $G\subset \PC^1_{\mathbf{Proj}}(\R/\Z)$, then it is uniquely defined up to conjugation in $\PC^1_{\mathbf{Proj}}(\R/\Z)$.
\end{cor}
\begin{proof}
Choose a piecewise projective $\mathcal{C}^1$-diffeomorphism $f_n$ from $\R/\Z$ to the $n$-fold covering of $\mathbf{P}^1_\R$ to inherit the result from Theorem \ref{misl2}.
\end{proof}

Let us now provide the affine version of Theorem \ref{tsg}. In Minakawa's original formulation \cite{Mi}, it consists of the classification of subgroups of $\PC^{0,+}_{\mathbf{Aff}}(\R/\Z)$ modulo conjugation, among those conjugate to the group $\R/\Z$ of translations within $\mathrm{Homeo}(\R/\Z)$.

In analogy with Definition \ref{d_compatible}, an affine structure $\xi$ on a standard curve $X$ is compatible if, $\xi_0$ being the standard structure, the identity map from $X$ to $(X,\xi_0)$ to $(X,\xi)$ is piecewise affine.

\begin{thm}\label{tmink}
Let $X$ be the standard circle $\R/\Z$ with its subgroup $K$ of isometries and its subgroup $K^+\simeq\mathrm{SO}(2)$ of orientation-preserving isometries. Consider a faithful continuous action of $K^+$ on $X$, whose image $G$ is contained in $\PC^{0}_{\mathbf{Aff}}(X)$.
Then
\begin{enumerate}
\item\label{tmink1} $G$ preserves a unique compatible affine structure $\xi$ on $X$;
\item\label{tmink3} the affinely modeled curve $(X,\xi)$ is isomorphic to $\Theta_t$ for some $t\ge 1$. We then say that $G$ is of type $\Theta_t$;
\item\label{tmink4} the conjugacy class of $G$ among subgroups of $\PC^{0}_{\mathbf{Aff}}(X)$ is characterized by the isomorphy type of the affinely modeled curve $(X,\xi)$, thus by the number $t\ge 1$;
\item define $t'=1$ if $t=1$ and $t'$ as equal to $t$ or $t^{-1}$ according to whether the universal covering of $(X,\xi)$, endowed with its orientation inherited from $\R/\Z$, is isomorphic as an affinely modeled curve to $\R_{>0}$ or $\R_{<0}$. Then the conjugacy class of $G$ modulo conjugacy by $\PC^{0,+}_{\mathbf{Aff}}(X)$ is characterized by the number $t'\in\R_{>0}$.
\item\label{tmink7} Consider the corresponding conjugacy classification for the corresponding action of $K^+$ modulo conjugation by $\PC^{0,+}_{\mathbf{Aff}}(X)$: then $t'$ is a full invariant.
\item\label{tmink8} Given the $K^+$-action, define $t''$ as equal to $t'$ or $-t'$ according to whether the orientation determined by the action of small positive element and the given orientation on $X$ coincide. Then $t''\in\R^*$ is a full invariant for such faithful $K^+$-actions, modulo conjugation by $\PC^{0,+}_{\mathbf{Aff}}(X)$.
\end{enumerate}
In contrast, if $G$ is the image of $K$ for such an action, then $G$ is conjugate to $K$ by some element of $\PC^{0,+}_{\mathbf{Aff}}(X)$, and similarly the action of $K$ is conjugate to its canonical isometric action by some element of $\PC^{0}_{\mathbf{Aff}}(X)$.
\end{thm}
\begin{proof}[Sketch of proof]
This can be done similarly as Theorem \ref{tsg}, but with the significant simplification of ``erasing all order 2 terms" in the preliminary work. Namely, we work with $\mathcal{L}^1(X)=X^\pm\times\R_{>0}$, and $\mathcal{A}^1(X)$ defined as those functions $\nu:X^\pm\to\R_{>0}$ such that $\nu(\hat{x})=\nu(x)^{-1}$ for all $x$ and equal to 1 outside a finite subset. Then one can encode compatible affine structures by elements of $\mathcal{A}^1(X)$; here compatible is in the affine sense, meaning that the identity map is piecewise affine.

The sequel is similar, with some specific points we now emphasize. The affinely modeled curves $\Theta_t$ for $t>1$ have no orientation-reversing automorphism. This has no such analogue in the projective setting. Since the oriented isomorphism type of $(X,\xi)$ is determined by the $\PC^{0}_{\mathbf{Aff}}(X)$-conjugacy class of $G$, this conjugacy classification corresponds to oriented affine structures, as described.

This phenomenon reappears if one considers classification of actions, since the action of ``small positive" elements of $K^+$ determines an orientation. 
\end{proof}

\subsection{Explicit formulas for commensurating actions}

Let us provide explicit commensurating actions of groups of piecewise $\mathcal{C}^k$-transformations for $k=0,1,2$. I obtained them I started this work, before I realized that the formalism of partial actions could get around such computations: while for $k=0$ it is easy and practical, for $k=2$ it becomes quite cumbersome and I did no attempt beyond. Still, it may be instructive to mention these formulas, notably to show what the universal globalization allows to conceal. I will provide no proof.

\subsection{The continuous case}
Let $X$ be an oriented standard curve. Define $X^\pm=X\times\{\pm 1\}$. 
We let $\PC(X)$ act on $X^\pm$ exactly as in the proof of Lemma \ref{bousch}.  For $y=(x,\eps)\in X^\pm$, we write $\hat{y}=(x,-\eps)$.

Define $\mathcal{L}^0_X=(X^\pm)^2$, and let $\PC(X)$ act diagonally.
Define $M=\mathcal{M}^0_X\subset \mathcal{L}^0_X$ as the set of pairs $(y,\hat{y})$ when $y$ ranges over $X^\pm$. Then $M\smallsetminus\sigma^{-1}M$ is the set of $(y,\pm 1)$ when $y$ ranges over outer discontinuity points of $\sigma$. For most applications this is enough; for precise counting results it can be convenient to rather work in the set of unordered pairs of distinct elements. It is instructive to interpret what $M$ being transfixed means and to thus prove the conjugacy results (for $\PC$, $\PC^+$, $\IETbw$, $\IET^+$).

\subsection{The derivable/affine case}

Define $\mathcal{L}^1_X=(X^\pm)^2\times\R_{>0}$. Recall that
$\PC_{\mathcal{C}^1\sharp}(X)$ denotes the set of piecewise-$\mathcal{C}^1$ self-diffeomorphisms of $X$.
For $\sigma\in\PC_{\mathcal{C}^1\sharp}(X)$, we define
\[\sigma\cdot (u,v,y)=\left(f(u),f(v),\frac{f'(u)}{f'(v)}y\right).\]

A simple computation (or a computation-free interpretation of the formula) shows that this is an action.

Here $f'(u)$ for $u=(x,\eps)$ is defined in the only natural way: choosing a local orientation-preserving affine chart around the first coordinate $f(u)_1$ of $f(u)$, we have $f'(u)=\lim_{t\to 0^+}(f(x+\eps t)-f(u)_1)/t$.

Define $M=\mathcal{M}^1_X=\{(x,\hat{x},1):x\in X\}$. Then $M\smallsetminus\sigma M$ is the set of $(x,\hat{x},1)$ where $\sigma$ cannot be assigned a value at $x$ for which it is of class $\mathcal{C}^1$ at $x$.

Then the transfixing property implies the transfixing property of the coarser quotient action on $\mathcal{L}^0_X$, and then the piecewise continuous setting yields a first (piecewise affine) conjugacy to a continuous action. This allows, in a second step, to work in the continuous setting (the conjugacy preserves the transfixing property in the new model).

Assume now that we are working with piecewise affine self-homeomorphisms and the invariant subset $\mathcal{L}^1_X=(X^\pm)^2\times\R_{>0}$, with the induced action.
Consider the set $\mathcal{A}^{1}_X$ of functions $\nu$ defined on a cofinite subset $X^\pm$ to $\R_{>0}$, taking the value $1$ outside a finite subset, and satisfying the condition $\nu(\hat{u})=1/\nu(u)$ for all $u\in X^\pm$ at which it is defined, and encode an affine structure on $X$ such that the identity map becomes piecewise affine, with such a function.

Since the set of functions from $X^\pm$ to $\R_{>0}$ can be thought of as a subset of $\mathcal{L}^{1,0}_X=\{(u,v,y)\in\mathcal{L}^1_X:v=\hat{u}\}$, it inherits an action of $\PC^{0}_{\mathcal{C}^1\sharp}(X)$. Then given such $\nu$, defined on the complement of a finite subset $F_\nu$, the condition of preserving $\nu$, for $\sigma\in\PC^{0}_{\mathcal{C}^1\sharp}(X)$, is equivalent to leaving $F_\nu$ invariant and acting as $\mathcal{C}^1$-diffeomorphisms on the complement of $F_\nu$. In particular, if $\sigma$ is piecewise affine, it means leaving $F_\nu$ invariant and acting as affine automorphisms on the complement of $F_\nu$. Finally, one checks that for a subgroup of $\PC^{0}_{\mathcal{C}^1\sharp}(X)$, the condition of transfixing $\mathcal{M}^1_X$ is equivalent to the existence of a $\Gamma$-invariant element $\nu\in \mathcal{A}^{1,0}_X$.

The above commensurating action (modulo a minor nuance), or rather the affine isometric Hilbertian action it induces, appears in \cite[\S 5]{LMT} in the continuous case, but without the interpretation in terms of affine structures.

\subsection{The doubly derivable/projective case}

To simplify, we stick to the continuous case. Indeed, the set has already been defined in \S\ref{conoac}, namely $\mathcal{L}^2(X)=X^\pm\times\R_{>0}\times\R$, with a somewhat complicated action. 

A commensurated subset $\mathcal{M}^2$ is given by the set of triples $(x,1,0)$ when $x$ ranges over $X^\pm$. This illustrates how it is frequent in commensurating actions that the commensurated subset is simpler to define than the whole set, and the benefit of using partial actions.

Now consider the set $\mathcal{A}^{2}_X$ of functions $\mu$ defined on a cofinite subset of $X^\pm$, valued in $\R_{>0}\times\R$, taking the value $(1,0)$ outside a finite subset, and satisfying the condition $\mu(\hat{x})=\tau(\mu(x))$. This is a little generalization of $\mathcal{A}^2(X)$ which only considers everywhere defined functions. Such functions define a compatible projective structure outside a finite subset.

The point is that the transfixing property implies the existence of such an invariant partially defined function (which, in the absence of finite orbit, implies the existence of a globally defined one), defined outside an invariant finite subset $F$ and then, the conjugation to an action of class $\mathcal{C}^2$ outside $F$, and projective outside $F$ if we started from a piecewise projective action.

\renewcommand{\thesubsection}{{\thesection.\arabic{subsection}}} 
\appendix
\section{Affinely, projectively modeled curves and their automorphisms}\label{s_projcur}
This section classifies affinely and projectively modeled curves and describe their automorphism groups. 

\subsection{Historical remarks: a recurrent mistake}
The easy affine classification was obtained by Kuiper \cite{Kui}: the complete case is trivial: there are only $\R$ and $\R/\Z$, while the non-complete case puts forward the non-standard circles $\R_{>0}/\langle t\rangle$.  Paradoxically, while in the affine case the complete case is the most trivial and less surprising part, in the projective case the non-complete case brings essentially nothing new (the compact ones come from the affine world) while the complete case is richer. The classification was established by Kuiper (1954) \cite{Ku2}, with an inaccuracy (namely, in the notation below, Kuiper did not distinguish the metaparabolic circles $\Xi_{n,+}$ and $\Xi_{n,-}$). With a similar approach, Goldman \cite{Gol2} later claimed to also obtain this classification, but actually rather establishes a correspondence to the classification of conjugacy classes in the universal covering of $\widetilde{\PSL}_2(\R)$, without providing details on the latter. In the analytic context of Hill equations, an equivalent classification to Kuiper's was obtained by Lazutkin and Pantrakova (1975) \cite{LP} which fixes Kuiper's error in another language (Kuiper is not quoted). Later, G. Segal (1981) \cite{Seg}, also not quoting Kuiper, claims to correct an error in \cite{LP}, but instead resurrects Kuiper's error, based on the same incorrect classification of $\widetilde{\PGL}_2(\R)$-conjugacy classes in $\widetilde{\PSL}_2(\R)$!

Kuiper's error reappeared at various places (often ``rediscovered"), and was fixed at other places, often not even noticing that there is an error, or by authors quoting several contradictory results without noticing the difference. The error is fixed in \cite{LP}, and a careful classification appears in \cite{BFP,Gor}; Gorinov \cite{Gor} is the first to explicitly mention the error, and also the first to fix it in the language of geometric structures (used in \cite{Ku2} and also here). Most other references are concerned with the Hill equation, which amounts to classifying orbits of the Bott-Virasoro extension of the diffeomorphism group on its Lie algebra, and this approach is much more complicated. 

Below, the classification and the computation of automorphism groups are done in the same impetus (carrying out the classification allowing to introduce notation). A description of orientation-preserving automorphism groups of projectively modeled curves is claimed in \cite{Guieu}, but the result is incorrect (and the proof way too long). Indeed, for the projectively modeled curves $\Xi_{n,\pm}$ and $\Xi_{n,t}$, he obtains that the orientation-preserving automorphism group is isomorphic to $(\R/\Z)\times(\Z/n\Z)$, while it is actually isomorphic to $\R\times(\Z/n\Z)$. Since the classification of maximal compact subgroups of automorphisms is needed here, this is an essential difference. So I am not aware of any prior reference for the classification of automorphism groups of projectively modeled curves.

\subsection{Generalities}\label{apgen}
We consider the pseudogroups $\mathbf{Aff}$ and $\mathbf{Proj}$ of Example \ref{exfirstco}. By affinely and projectively modeled curves we mean $\mathbf{Aff}$-modeled and $\mathbf{Proj}$-modeled curves.

Denote by $\Sigma_\infty$ the universal covering of $\mathbf{P}^1(\mathbf{R})$. We identify it with $\mathopen]-\infty,+\infty]\times\Z$ with the lexicographic order. Restricting the universal covering map $(t,n)\mapsto t$ to open subsets on which it is injective, we obtain charts for a $\mathbf{Proj}$-structure on $\Sigma_\infty$ (note that it is not finitely-charted).

\begin{lem}
Any $\mathbf{Aff}$-modeled Hausdorff simply connected curve $X$ is isomorphic to an open interval in $\R$.

Any $\mathbf{Proj}$-modeled Hausdorff simply connected curve $X$ is isomorphic to an open interval in $\Sigma_\infty$.
\end{lem}
\begin{proof}
In the affine case, by \cite[Prop.\ 4.5]{Gol}, there is a locally projective immersion $X\to\R$. Since a locally injective continuous map between intervals is injective, the latter map has to be injective.

In the projective case, by \cite[Prop.\ 4.5]{Gol}, there is an locally projective immersion $X\to\mathbf{P}^1_\R$, which therefore lifts to a locally projective immersion $X\to \Sigma_\infty$. We deduce injectivity by the same argument.
\end{proof}

\subsection{The affine case}
This is especially an ap\'eritif to the projective case, otherwise we would just give a list without proof. 

The group of affine self-transformations of $\R$ is 2-transitive and contains orientation-reversing elements. Hence there are only three open intervals up to affine automorphism: $\R$, $\R_{>0}$, and $\mathopen]-1,1\mathclose[$. 

Let us list in each case the fixed-point-free automorphisms up to conjugation. 

\begin{itemize}
\item For $\mathopen]-1,1\mathclose[$ the automorphism group is reduced to a group of order 2, and hence there is no fixed-point-free automorphism.
\item For $\R_{>0}$, the automorphism group is reduced to positive homotheties $u_t=x\mapsto tx$, $t>0$; note that it preserves an orientation: indeed, this interval $\R_{>0}$ has one ``complete" end ($+\infty$) and a non-complete one ($0$). The automorphisms $u_t$ for $t>0$ are pairwise non-conjugate in $\R_{>0}$, and fixed-point-free for $t\neq 1$. We have $u_t^{-1}=u_{t^{-1}}$. We endow the quotient $\Theta_t=\R_{>0}/\langle u_t\rangle$ with the orientation inherited from $\R_{>0}$ if $t>1$, and with the reverse orientation if $t<1$. Thus the oriented affinely modeled curves $\Theta_t$ for $t\in\R_{>0}$ are pairwise non-isomorphic. In the non-oriented setting, for $t>1$, they are pairwise non-isomorphic, while $\Theta_t$ and $\Theta_{t^{-1}}$ are isomorphic (indeed equal!).

The affine automorphism group of $\Theta_t$ is the normalizer of $\langle u_t\rangle$ (hence the whole affine automorphism group of $\R_{>0}$ here) modulo $\langle u_t\rangle$, hence is naturally isomorphic to $\R_{>0}/\langle t\rangle$ and non-canonically isomorphic to $\R/\Z$; it coincides with the oriented automorphism group.
\item For $\R$, the automorphism group consists of all affine automorphisms, and the fixed-point-free ones are translations, and are all conjugate; actually the cyclic subgroups generated by translations are all conjugate by orientation-preserving affine automorphisms, and the quotient of $\R$ by such cyclic subgroups we obtain is unique up to affine isomorphism: one representative is the standard circle $\R/\Z$. The normalizer of the given cyclic subgroup is the group of isometries of $\R$ and hence the affine automorphism group is isomorphic to $(\R/\Z)\rtimes(\Z/2\Z)$, and can be identified to the group of isometries of $\R/\Z$. We denote it by $\Theta_1$.
\end{itemize}

We conclude the following classification of affinely modeled curves and their automorphism groups. In the left column, $\bullet$ and $\circ$ mean complete vs non-complete.
\begin{center}
  \begin{tabular}{|c|cc|c|c|c|}
  \hline
$\bullet/\circ$ & \multicolumn{2}{c|}{aff. mod. curve}& univ.\ cover & $\Aut^+\simeq$ & $\Aut\simeq$\\
    \hline
$\bullet$& $\R$& & itself & $\R\rtimes\R_{>0}$ & $\R\rtimes\R^*$\\
$\circ$ &$\R_{>0}$& & itself & $\R_{>0}\simeq\R$ & $=\Aut^+$\\
$\circ$ &$\mathopen]-1,1\mathclose[$& & itself & $\{1\}$ & $\{\pm 1\}$\\
$\bullet$ &$\Theta_1$&  & $\R$ & $\R/\Z$& $(\R/\Z)\rtimes (\Z/2\Z)$ \\
$\circ$ &$\Theta_t$& $t\!>\!1$ & $\R_{>0}$ & $\R_{>0}/\langle t\rangle\simeq\R/\Z$ & $=\Aut^+$\\
    \hline
   \end{tabular} 
\end{center}

Note that we have several ways of defining affinely modeled curves using a pseudogroup. The chosen way does not affect the classification as above. However, it affects the notion of being finitely charted (i.e. having a finite atlas in the model space). For the most standard way to define affinely modeled curves, one uses $\R$ as model space and all the above curves are finitely-charted. However, in our work and applications, the natural model space is rather the standard affine structure on $\R/\Z$. Then $\R$ and $\R_{>0}$ are not finitely-charted and only $\mathopen]-1,1\mathclose[$, $\Theta_1$ and $\Theta_t$ remain; in particular the orientation-preserving automorphism group of any finitely-charted affine manifold (modeled on $\R/\Z$) is always abelian.

\subsection{The projective case}\label{ap_proj}
As reminded in \S\ref{apgen}, we have to classify open intervals in $\Sigma_\infty$, and then classify their fixed-point-free automorphisms to obtain the compact projective manifolds.

The automorphism group of $\Sigma_\infty$ contains the orientation-preserving automorphism group as a subgroup of index 2, and the latter can be identified with the universal covering $\widetilde{\mathrm{SL}}_2$ of $\mathrm{SL}_2(\R)$. Since the conjugation by a reflection yields a automorphism of order 2, we can view the automorphism group as a group $\widetilde{\mathrm{PGL}}_2$, projecting to $\mathrm{PGL}_2(\R)$ with an infinite cyclic kernel (which is the center of $\widetilde{\mathrm{SL}}_2$). The action of $\widetilde{\mathrm{PGL}}_2$ on $\R$ is transitive, but not transitive on ordered pairs. Namely, the stabilizer of $(\infty,0)\in \Sigma_\infty$ acts on $\Sigma_\infty$ with countably many orbits: the singletons $(\infty,n)$, $n\in\Z$, and the intervals in between. Therefore, we can classify open intervals in $\Sigma_\infty$ up to automorphisms. (Since parentheses are used for pairs, we use the ``French" notation $\mathopen]a,b\mathclose[$ for the open interval $\{x:a<x<b\}$.) 

\begin{itemize}
\item $\Sigma_\infty$ itself (not finitely-charted)
\item the interval $\Sigma_\infty^+$ of elements $>(\infty,0)$ in $\Sigma_\infty$ (not finitely-charted)
\item the interval $\Sigma_n=\mathopen](\infty,0),(\infty,n)\mathclose[$ in $\Sigma_\infty$ ($n\ge 1$ integer).
\item the interval $\Sigma_{n-\frac12}=\mathopen](0,1),(\infty,n)[$ in $\Sigma_\infty$ ($n\ge 1$ integer).
\end{itemize}

This yields the classification of simply connected projectively modeled curves. It remains to obtain the classification of compact connected projectively modeled curves, so we have to consider, among the above curves, which ones have a fixed-point-free orientation-preserving automorphism, and classify these up to conjugation by an automorphism. We start with the easier non-complete case, that is, those for which the universal covering is properly contained in $\Sigma_\infty$.

In $\Sigma_\infty^+$, in $\Sigma_n$ or $\Sigma_{n-\frac12}$ for $n\ge 2$, the element $(\infty,1)$ is fixed by every orientation-preserving automorphism. The remaining intervals are $\Sigma_1\simeq\R$ and $\Sigma_{\frac12}\simeq\R_{>0}$ (and $\Sigma_\infty$, which we consider afterwards). Here $\R$ and $\R_{>0}$ can be viewed as subsets of the projective line. 

\begin{itemize}
\item In $\Sigma_1\simeq\R$, the automorphism group is the affine group. Fixed-point-free elements are nonzero translations, and are all conjugate. Let $\Theta_1$ be the corresponding curve; call it the round circle. 
\item In $\Sigma_{\frac12}\simeq\R_{>0}$, the orientation-preserving automorphism group consists of the positive homotheties $u_t:x\mapsto tx$; for $t\neq 1$ these are fixed-point-free. The automorphism group also contains $t\mapsto t^{-1}$; thus $u_t$ is conjugate to $u_{s}$ if and only $s\in\{t,t^{-1}\}$. Let $\Theta_t$ be the corresponding curve ($t\in\R_{>0}\smallsetminus\{1\}$); thus $\Theta_t$ are pairwise non-isomorphic as oriented projectively modeled curves, and $\Theta_t$ and $\Theta_{t^{-1}}$ are (orientation-reversing) isomorphic as projectively modeled curves.
\end{itemize}

\begin{rem}
Note that the $\Theta_t$ already appear in the affine classification. But there are a few little differences: as affinely modeled curve, $\Theta_1$ is complete, while it is not complete as projectively modeled curve. The other difference is that, for $t\neq 1$, $\Theta_t$ and $\Theta_{t^{-1}}$ are isomorphic as oriented projectively modeled curves, an isomorphism being induced by $x\mapsto x^{-1}$. This difference is also reflected in the fact that $\Theta_t$ admits orientation-reversing projective automorphisms.
\end{rem}

Now we have classified the non-complete projectively modeled curves, we need to compute their automorphism groups (not anymore for classification, but because we obtain results of conjugation into the automorphism group of a projectively modeled curve). In the simply connected case we already did the job. In the compact case, where it was obtained as quotient $X=\langle r\rangle C$ of a simply connected projectively modeled curve by a cyclic group of automorphisms acting freely, we have to compute the normalizer $N_r$ of this cyclic subgroup $\langle r\rangle$; then the automorphism group $A$ of $X$ is $N_r/\langle r\rangle$.

\begin{itemize}
\item For $\Sigma_1\simeq\R$, and $r(x)=x+1$, the normalizer $N_r$ is the group of isometries of $\R$, and $A$ can be viewed as the group of of isometries of $X\simeq\R/\Z$.
\item For $\Sigma_{\frac12}\simeq\R_{>0}$ and $r(x)=tx$ with $t>1$, the cyclic subgroup $\langle r\rangle$ is normal in the whole automorphism group $\R_{>0}\rtimes\langle\tau\rangle$ (where $\tau(x)=x^{-1}$). The quotient $(\R_{>0}/\langle t\rangle)\rtimes\langle\tau\rangle$ is also (non-canonically) isomorphic to the group of isometries of $\R/\Z$.
\end{itemize}

We summarize the classification of non-complete finitely-charted projectively modeled curves up to isomorphism, along with their orientation-preserving automorphism group. Note that the only infinitely-charted one is $\Sigma_\infty^+$.
\medskip

\begin{center}
  \begin{tabular}{|c|c|c|c|}
  \hline
proj. mod. curve & universal cover & $\Aut^+\simeq$ & $\Aut\simeq$\\
    \hline
$\Sigma_n$, $n\in\N_{>0}$ & itself & $\R\rtimes\R_{>0}$ & $\R\rtimes\R^*$\\
$\Sigma_{n-\frac12}$, $n\in\N_{>0}$ & itself & $\R$ & $\R\rtimes (\Z/2\Z)$\\
standard affine $\Theta_1$  & $\Sigma_1\simeq\R$ & $\R/\Z$& $(\R/\Z)\rtimes (\Z/2\Z)$ \\
non-standard affine $\Theta_t$, $t> 1$ & $\Sigma_{1/2}\simeq\R_{>0}$ & $\R/\Z$ & $(\R/\Z)\rtimes (\Z/2\Z)$\\
    \hline
   \end{tabular} 
\end{center}

Let us now deal with complete curves. In $\Sigma_\infty$, the orientation-preserving automorphism group consists of $\widetilde{\mathrm{SL}}_2$. We call elements of $\widetilde{\mathrm{SL}}_2$ metaelliptic, metaparabolic, or metahyperbolic according to the corresponding behavior of the projection on $\mathrm{PSL}_2(\R)$, with the convention that elements mapping to 1 are metaelliptic. We add ``meta" because the wording ``elliptic", etc, does not reflect the dynamical behavior on $\Sigma_\infty$: for instance non-identity metaelliptic elements rather behave as loxodromic elements on $\Sigma_\infty$.

\begin{itemize}
\item Non-identity metaelliptic elements are fixed-point-free on $\Sigma_\infty$. Every metaelliptic element is conjugate to an element in the inverse image $\widetilde{\mathrm{SO}}_2$ of $\mathrm{SO}_2$. Write this 1-parameter subgroup as $(\xi_r)_{r\in\R}$, so that the center of $\widetilde{\mathrm{SL}}_2$ consists of the $\xi_r$ for $r\in\Z$. Lifting conjugation by the lift of an orthogonal reflection conjugates $\xi_r$ to $\xi_{-r}$. These are the only conjugacies among the $\xi_r$ since the rotation number of $\xi_r$ is $r$, and rotation number is conjugacy-invariant up to sign in the whole automorphism group $\widetilde{\mathrm{PGL}}_2$ (and is conjugacy-invariant in $\widetilde{\mathrm{SL}}_2$). Since we are interested in the cyclic subgroup generated by $\xi_r$, we can restrict to $r>0$.
We define $\Omega_r$ ($r>0$) as the corresponding projectively modeled curve (quotient of $\Sigma_\infty$ by $\langle\xi_r\rangle$).

\item Metahyperbolic and metaparabolic elements have a fixed point on $\mathbf{P}^1_\R$; up to conjugating in $\widetilde{\mathrm{SL}}_2$, we can suppose that $\infty$ is fixed. Its action on $\mathbf{P}^1_\R$ is therefore given by some affine map $A:x\mapsto tx+b$ with $t>0$ and $(t,b)\neq (1,0)$. Conjugation inside the positive affine group reduces to either $b=0$ or ($t=1$ and $b\in\{1,-1\}$), which we therefore assume. For its action on $\Sigma_\infty$, the element $\xi$ preserves the subset $\{\infty\}\times\Z$, on which it acts as a translation, say by some $n\in\Z$. Then $\xi=\xi_{n,A}$ is determined by $A$ and $n$, namely $\xi_{n,A}(x,m)=(Ax,n+m)$ in the previous coordinates. It is fixed-point-free if and only if $n\neq 0$. Write $\xi_{n,t}$ (resp.\ $\xi_{n,\pm}$) for $\xi_{n,A}$ when $A(x)=tx$ (resp.\ $A(x)=x\pm 1$)

We have to classify the cyclic subgroups $\langle\xi_{n,A}\rangle$ up to conjugation in the automorphism group $\widetilde{\mathrm{PGL}}_2$ of $\Sigma_\infty$. Let us start with the elements $\xi_{n,A}$ themselves. Since $n$ is the rotation number, it is determined by conjugation up to sign; also $\{t,t^{-1}\}$ is determined by conjugation (by conjugacy classification in $\mathrm{PSL}_2(\R)$). Given this, the elements we have not yet distinguished are (for a given $n\in\N_{>0}$ and $t\ge 1$):

\begin{itemize}
\item on the one hand, the four elements $\xi_{n,t}$, $\;\xi_{n,t^{-1}}$, $\;\xi_{-n,t}(=\xi_{n,t^{-1}}^{-1})$, $\;\xi_{-n,t^{-1}}(=\xi_{n,t}^{-1})$;
\item on the other hand, the four elements $\xi_{n,+}$, $\;\xi_{n,-}$, $\;\xi_{-n,+}(=\xi_{n,-}^{-1})$, $\;\xi_{-n,-}(=\xi_{n,+}^{-1})$.
\end{itemize}

Define two particular automorphisms of $\Sigma_\infty$ as follows: $s(x,m)=(-x,-m)$ for $x\neq\infty$ and $s(\infty,m)=(\infty,-m-1)$; $w(x,m)=(-x^{-1},m+1)$ for $x>0$ and $w(x,m)=(-x^{-1},m)$ for $x\le 0$. Note that $w$ is orientation-preserving.

Then $w^{-1}\xi_{n,t}w=\xi_{n,t^{-1}}$, $s^{-1}\xi_{n,t}s=\xi_{-n,t}$, $s^{-1}\xi_{n,\pm}s=\xi_{n,\pm}^{-1}$. This shows that the first four elements $\xi_{\pm n,t^{\pm 1}}$ are conjugate in $\widetilde{\mathrm{PGL}}_2$, and the two corresponding cyclic subgroups $\langle\xi_{n,t}\rangle$, $\langle\xi_{n,t^{-1}}\rangle$ are conjugate within $\widetilde{\mathrm{SL}}_2$. On the other hand, this leaves the possibility that the subgroups $\langle\xi_{n,+}\rangle$ and $\langle\xi_{n,-}\rangle$ are not conjugate at all. It is indeed the case that they are not conjugate: indeed, a conjugating element should fix the unique fixed point in $\mathbf{P}^1_\R$, and thus preserve $\{\infty\}\times\Z$, and these are precisely the elements we have tested. We write $\Xi_{n,t}=\Sigma_\infty/\langle \xi_{n,t}\rangle$ for $t>1$, and $\Xi_{n,\pm}=\Sigma_\infty/\langle\xi_{n,\pm}\rangle$.
\end{itemize}

\begin{rem}
1) Recall that $\xi_{n,\pm}(x,m)=(x\pm 1,n+m)$ for $(x,m)\in\Sigma_\infty$ (identified as above to $\mathopen]-\infty,+\infty]\times\Z$ with the lexicographic ordering).
We have observed that, for $n\in\Z\smallsetminus\{0\}$ the cyclic subgroups $\langle\xi_{n,+}\rangle$ and $\langle\xi_{n,-}\rangle$ are not conjugate in $\widetilde{\mathrm{PGL}}_2$. The error in Kuiper's classification \cite{Ku2}, reproduced at many other places (but fixed in \cite{LP,BFP,Gor}, and consciously only in \cite{Gor}) is to not distinguish those elements (essentially, the error amounts to deducing their conjugacy from the fact that their images in $\PSL_2(\R)$ are conjugate). 

2) That $w^{-1}\xi_{n,t}w=\xi_{n,t^{-1}}$ is quite a subtle point (the subtlety is to consider $w$). This is another gap in \cite{Ku2}, which takes for granted that we immediately boil down to homotheties $x\mapsto tx$ for $t>1$; this gap does not result in another error thanks to this possibly unexpected conjugacy. This point is carefully taken care of in \cite{Gor}. 
\end{rem}

Let us mention, as a digression of independent interest, that this lack of conjugacy exists at a purely topological level:

\begin{prop} For every $n\in\N_{\ge 1}$, the cyclic subgroups $\langle\xi_{n,+}\rangle$ and $\langle\xi_{n,-}\rangle$ are not conjugate in the group $\widetilde{\mathrm{Homeo}}(\mathbf{P}^1_\R)$, namely the normalizer in $\mathrm{Homeo}(\Sigma_\infty)$ of the cyclic subgroup $\langle\xi_1\rangle$. More precisely, for any $k\in\N_{\ge 3}\cup\{\infty\}$ not dividing $2n$, the cyclic subgroups $\langle\xi_{n,+}\rangle$ and $\langle\xi_{n,-}\rangle$ are not conjugate in the normalizer $\widetilde{\mathrm{Homeo}}^{(k)}(\mathbf{P}^1_\R)$ of $\langle\xi_1\rangle$ in the group of self-homeomorphisms of $\Sigma_k$ (which contains $\PGL_2^{(k)}(\R)$).
\end{prop}
\begin{proof}
Indeed, suppose by contradiction that they are conjugate in $E_k$. It means that the generators of these cyclic subgroups are conjugate, or that one is conjugate to the other's inverse. Let us show that the latter case implies the former case. In the latter case, there exists $b\in E_k$ such that $b\xi_{n,+}b^{-1}=\xi_{n,-}^{-1}$. This changes the rotation number $n$ to $-n$ (which lives in $\R/k\Z$). By assumption, $n\neq -n$ in $\R/k\Z$. Hence necessarily $b$ is orientation-reversing and thus $c=bs$ is orientation-preserving, that is, commutes with $\xi_1$, and $c\xi_{n,+}c^{-1}=\xi_{n,-}$. Projecting on $\mathrm{Homeo}^+(\mathbf{P}^1_\R)$, we obtain an orientation-preserving self-homeomorphism $\gamma$ conjugating $u$ to $u^{-1}$, with $u(x)=x+1$. Hence $\gamma$ fixed the unique fixed point $\infty$ of $u$ and thus this is a conjugacy in $\mathrm{Homeo}(\R)$. But $u(x)>x$ for all $x$, which is clearly an obstruction for $u$ to be conjugate to its inverse within orientation-preserving self-homeomorphisms of $\R$. This is a contradiction.

(Note that conversely, whenever $2n=0$ [$\mathrm{mod}$ $k$], the elements $\xi_{n,+}$ and $\xi_{n,-}$ are inverse to each other in $\PSL_2^{(k)}(\R)$, and thus the cyclic subgroups they generate are not only conjugate, but equal.)
\end{proof}

The classification of complete projectively modeled curves being completed, it remains to compute their automorphism groups. In each case, the automorphism group of the projectively modeled curve $\Sigma_\infty/\langle\xi\rangle$ is $N_\xi/\langle\xi\rangle$, where $N_\xi$ is the normalizer of $\langle\xi\rangle$ in the automorphism group.

To determine $N_\xi$, define some auxiliary subgroups, easier to determine. 
Namely, define $B_\xi$ as the stabilizer in $\mathrm{PGL}_2(\R)$ for conjugation of the image of $\{\xi,\xi^{-1}\}$ in $\mathrm{PSL}_2(\R)$; let $B_\xi$ be its unit component. Let $M_\xi$ as the inverse image in $\widetilde{\mathrm{PGL}}_2$ of the stabilizer $B_\xi$, and $M_\xi^0$ the inverse image of $B_\xi^\circ$ (note that $M_\xi^0$ is not necessarily connected). 
Clearly, $N_\xi\subset M_\xi$. Moreover, the unit component $M_\xi^\circ$ is included in $N_\xi$ (it is even included in the centralizer of $\xi$); since the center of $\widetilde{\mathrm{PGL}}_2$ is also included in the normalizer, it follows that $M_\xi^0$ is included in $N_\xi$. Thus $M_\xi^0\subset N_\xi\subset M_\xi$. We can deduce the automorphism group, in each case:

\begin{itemize}
\item For $\Omega_r=\Sigma_\infty/\langle\xi_r\rangle$, we have to discuss on whether $r\in\Z$. If $r=n\in\Z$, the subgroup $\langle\xi_r\rangle$ is normal in the whole automorphism group. Hence the orientation-preserving automorphism group is the quotient $\mathrm{PSL}_2^{(n)}(\R)$, the $n$-fold connected covering of $\mathrm{PSL}_2(\R)$; the full automorphism group $\mathrm{PGL}_2^{(n)}(\R)$ is obtained as semidirect product with the automorphism induced by a reflection.

\item When $r\notin\Z$, the normalizer of $\langle\xi_r\rangle$ is reduced to the inverse image of the orthogonal group, which has two components then $s$ is contained in the nontrivial component and normalizes $\langle\xi_r\rangle$, so in this case $N_\xi=M_\xi$. Thus the automorphism group is isomorphic to $(\R/\Z)\rtimes (\Z/2\Z)$ with action by sign, and the orientation-preserving automorphism group is isomorphic to $\R/\Z$.

\item For $\Xi_{n,\pm}$, the normalizer is the inverse image of the group of upper triangular matrices that are either scalar or trace-zero. In $\mathrm{PSL}_2(\R)$, the latter group has two connected components; the nontrivial component corresponding to trace zero matrices. A subset of representatives for $M_\xi$ modulo $M_\xi^0$ is given by $\{1,s\}$. Since $s\xi_{n,\pm}s^{-1}=\xi_{n,\pm}^{-1}$, we deduce that $s\in N_\xi$ and hence $N_\xi=M_\xi$.
thus the automorphism group is isomorphic to $(\R\rtimes\Z/n\Z)\rtimes (\Z/2\Z)$, with action by sign multiplication, and the orientation-preserving automorphism group is isomorphic to $\R\times\Z/n\Z$.

\item For $\Xi_{n,t}$ ($t>1$, $n\in\N_{>0}$) and $\xi=\xi_{n,t}$, the subgroup $B_\xi$ consists of monomial matrices in $\PGL_2(\R)$, which has 4 components. A subset of representatives of $B_\xi$ modulo $B_\xi^\circ$ is $\{1,s,w,sw\}$.

The element $w$, which belongs to $M_\xi$, does not normalize $\langle\xi\rangle$ (because $w^{-1}\xi_{n,t}w=\xi_{n,t^{-1}}\notin\langle\xi_{n,t}\rangle$). For the same reason, $s$ does not normalize. However, $q=sw$ normalizes (and does not centralize) $\langle\xi_{n,t}\rangle$. Note that $q^2=1$.

So the normalizer of $\langle\xi\rangle$ in $\widetilde{\mathrm{SL}}_2$, which is also its centralizer in $\widetilde{\mathrm{PGL}}_2$, is the direct product $(\xi_{0,\eta})_{\eta>0}\times (\xi_m)_{m\in\Z}$. Its quotient by $\langle\xi\rangle$ can be described as the direct product $(\xi_{0,\eta})_{\eta>0}\times (\xi_{m,t^{m/n}})_{m\in\Z/n\Z}$, which is isomorphic to $\R\times (\Z/n\Z)$. This is the orientation-preserving automorphism group of $\Xi_{n,t}$.

The normalizer in $\widetilde{\mathrm{PGL}}_2$ is the semidirect product $((\xi_{0,\eta})_{\eta>0}\times (\xi_m)_{m\in\Z})\rtimes\langle q\rangle$, where $q$ acts by sign. Its quotient by $\langle\xi\rangle$ can be described as the semidirect product $((\xi_{0,\eta})_{\eta>0}\times (\xi_{m,t^{m/n}})_{m\in\Z/n\Z})\rtimes\langle q\rangle$, which is isomorphic to $(\R\times (\Z/n\Z))\rtimes(\Z/2\Z)$, again with action by sign.
\end{itemize}

Let us summarize the classification up to isomorphism of complete finitely-charted projectively modeled curves (in each case the universal cover is $\Sigma_\infty$, which is the only infinitely-charted complete projectively modeled curve up to isomorphism):

\medskip\vspace{-0.2cm}
\begin{center}  \begin{tabular}{|c|c|c|}
  \hline
projectively modeled curve
  & $\Aut^+\simeq$ & $\Aut\simeq$\\
    \hline
special metaelliptic $\Omega_n$, $n\in\N_{>0}$ &  $\PSL_2^{(n)}(\R)$&  $\PGL_2^{(n)}(\R)$ \\
metahyperbolic $\Xi_{n,t}$, $t>1$, $n\in\N_{>0}$  & $\R\times\Z/n\Z$ & $(\R\times\Z/n\Z)\rtimes(\Z/2\Z)$\\
metaparabolic $\Xi_{n,\eps}$, $\eps\in\{+,-\}$, $n\in\N_{>0}$  & $\R\times\Z/n\Z$ & $(\R\times\Z/n\Z)\rtimes(\Z/2\Z)$\\
ordinary metaelliptic $\Omega_{r}$, $r>0,r\notin\Z$ & $\R/\Z$& $(\R/\Z)\rtimes(\Z/2\Z)$ \\
    \hline
   \end{tabular} \end{center}

Note that in all finitely-charted cases (non-complete and complete), the whole automorphism group is a semidirect product of its orientation-preserving normal subgroup by $\Z/2\Z$. This also holds for the infinitely-charted $\Sigma_\infty$, while $\Sigma_\infty^+$ has no orientation-reversing automorphism.

Finally, note that in the metaelliptic case, the automorphism group is transitive, while in the metaparabolic and metahyperbolic cases $\Xi_{n,t}$ and $\Xi_{n,\pm}$, the automorphism group has exactly 2 orbits, one of which being finite (of cardinal $n$ in the metaparabolic case, of cardinal $2n$ in the metahyperbolic case).

\begin{rem}
The non-complete structures $\Theta_t$ and $\Theta_1$ can be thought of as the $n=0$ case of $\Xi_{n,t}$ and $\Xi_{n,\pm}$ (note the collapse of the $\pm$ distinction when $n=0$, which is not always well-reflected in the literature). We thus have a canonical bijection, as observed  and used by Kuiper \cite{Ku2}, between the orbit space of $\widetilde{\PSL}_2\smallsetminus\{1\}$ modulo the conjugation action of $\widetilde{\PGL}_2$, and the set of projective structures on the circle, modulo diffeomorphism; in addition, it also corresponds to the set of projective structures modulo oriented diffeomorphism, because (unlike in the affine case) all such structures admit an orientation-reversing automorphism.
\end{rem}

Let us mention that Ghys \cite[\S 4.2]{Ghy} provides a direct argument of the following alternative: for every projectively modeled curve homeomorphic to the circle, either it is projectively isomorphic to a finite covering of the projective line, or its oriented automorphism group is abelian. 

We now use the classification to derive the following consequences. 
For the following one, it is possible that there is a more elegant, classification-free argument.

\begin{prop}\label{autisclosed}
For every projectively modeled curve $X$ (with finitely many components), the automorphism group $A$ of $X$ is closed in $\mathrm{Homeo}(X)$.
\end{prop}
\begin{proof}
Since the componentwise preserving subgroup of $A$ is open in $A$, we can restrict to this one, and thus boil down to the case when $X$ is connected, and thus use the classification. Note that in each case, the description of the automorphism group as a Lie group makes it clear that it acts continuously on the given curve. Since in each case $A$ has finitely many components, it is enough to check that $A^\circ$ is closed.

When $A^\circ$ is compact, since it acts continuously it is closed. When $X$ is $\Sigma_\infty^+$, or $\Sigma_n$ for $n\in\N_{>0}$, the subgroup $A^\circ$ can be viewed as the subgroup of element pointwise stabilizing some closed discrete subset and acting as orientation-preserving affine transformation in each interval of its complement, and such that all these affine maps (on all intervals) are equal. This is easily checked to be a closed condition. 

For $\Sigma_\infty$ and $\Omega_n$ ($n\in\N_{>0}$), we start with $\Omega_1$: this is $\PGL_2$, and is the stabilizer of the cross-ratio, and thus is closed. For others, we start observing that the centralizer of the deck transformation is closed in the whole homeomorphism group, so it is enough to show that $A^\circ$ is closed in this centralizer, and then it consists of the preimage of $\PSL_2$ for the projection of this centralizer to $\mathrm{Homeo}^+(\mathbf{P}^1_\R)$. So it is closed.

The remaining cases are $\Xi_{n,t}$, $\Xi_{n,\varepsilon}$, and $\Sigma_{n-1/2}$. Then we observe that $A^\circ$ preserves a finite subset on which it acts trivially, and that it acts properly on its complement. This implies that it is closed.
\end{proof}

Let us now deal with maximal compact subgroups.
Let us recall a classical result of Mostow \cite{Mo} that in a virtually connected Lie group, all maximal compact subgroups are conjugate by some element of the unit component, and that moreover they have as many components as the whole group. We thus draw a table indicating, in the right column, one isomorphic copy in each case of a maximal compact subgroup (we write no quantifiers on the left column: the indices $n,t,\eps,r$ are meant to be the same as in the previous two tables). We then derive a conjugacy result, which is used in the paper.

\begin{center}  \begin{tabular}{|c|c|c|c|}
  \hline
proj. mod. curve
  & $\Aut\simeq$ & Maximal compact $\simeq$\\
    \hline
$\Sigma_\infty$ & $\PGL_2^{(\infty)}(\R)$ &  $\Z/2\Z$ \\
$\Sigma_\infty^+$, $\Sigma_n$ & $\R\rtimes\R^*$  &  $\Z/2\Z$ \\
 $\Sigma_{n-\frac12}$ &  $\R\rtimes(\Z/2\Z)$ &  $\Z/2\Z$ \\
$\Xi_{n,t}$, $\Xi_{n,\eps}$ & $(\R\times\Z/n\Z)\rtimes(\Z/2\Z)$ & $(\Z/n\Z)\rtimes(\Z/2\Z)$\\
$\Theta_{t\ge 1}$, $\Omega_r$, $r\notin\N$ & $(\R/\Z)\rtimes(\Z/2\Z)$ & $(\R/\Z)\rtimes(\Z/2\Z)$\\
$\Omega_n$ & $\PGL_2^{(n)}(\R)$ & $(\R/\Z)\rtimes(\Z/2\Z)$ \\
    \hline
   \end{tabular} \end{center}

As an immediate consequence of the classification and Mostow's conjugacy result, we derive:

\begin{prop}\label{conjuo2}
Let $A$ be the automorphism group of a connected projectively modeled curve. 
Then any two closed subgroups of $G$ that are topologically isomorphic to $\mathrm{SO}(2)$ are conjugate by some element of the identity component $A^\circ$. The same holds with $\mathrm{SO}(2)$ replaced with $\mathrm{O}(2)$.\qed
\end{prop}

Namely, such copies of $\mathrm{SO}(2)$ exist in the case of $\Theta_t$ for $t\ge 1$ and $\Omega_r$ for $r\ge 0$, but not in the other cases, i.e., for metahyperbolic and metaparabolic curves $\Xi_{n,t}$, $\Xi_{n,\eps}$ (and for curves homeomorphic to an interval).

\end{document}